\documentclass[a4paper, 11pt]{article}
\usepackage[margin = 0.7 in]{geometry}
\usepackage{amsmath, amssymb,amsthm,indentfirst,placeins,float,graphicx,booktabs, verbatim,xcolor,titlesec,mathrsfs,dsfont,multirow,subcaption,tikz,epstopdf}

\usepackage[hyperindex,breaklinks]{hyperref}
\hypersetup{linkcolor=blue, citecolor=red, colorlinks=true}

\usepackage{mathtools}                      % need for `show only references'
\mathtoolsset{showonlyrefs=true}          % only equations which are labeled AND referenced will numbered use \eqref{} instead of (\ref{})
\usepackage[round]{natbib}

\newtheorem{theorem}{Theorem}
\newtheorem{lemma}[theorem]{Lemma}
\newtheorem{proposition}[theorem]{Proposition}

\newtheorem{definition}[theorem]{Definition}

\newtheorem{remark}[theorem]{Remark}

\numberwithin{equation}{section}
\numberwithin{theorem}{section}

\newcommand\Lc{\mathcal{L}}

\newcommand{\dd}{\mathrm{d}}

\allowdisplaybreaks

\begin{document}
	
		\title{Optimal Dividend, Reinsurance and Capital-Injection Strategies for Collaborating Business Lines: The Case of Excess-of-Loss Reinsurance}
			%\title{Optimal Dividend, Excess-of-Loss Reinsurance and Capital Injection Strategies for Collaborating Business Lines}
\author{
Tim J. Boonen\thanks{Department of Statistics and Actuarial Science, School of Computing and Data Science, University of Hong Kong, China. Email: \texttt{tjboonen@hku.hk}.}
\and
Engel John C. Dela Vega\thanks{Department of Mathematical and Statistical Sciences, University of Alberta, Edmonton, Alberta, Canada. Email: \texttt{engel.delavega@ualberta.ca}.}\textsuperscript{,}\thanks{Department of Statistics and Actuarial Science, School of Computing and Data Science, University of Hong Kong, China.}
}
	
%	\date{This version: \today}
	\date{}	
	\maketitle
	
	\begin{abstract}
This paper considers an insurer with two collaborating business lines that must make three critical decisions: (1) dividend payout, (2) a combination of proportional and excess-of-loss reinsurance coverage, and (3) capital injection between the lines. The reserve level of each line is modeled using a diffusion approximation, and the insurer aims to maximize the weighted total discounted dividends paid until the first ruin time. We obtain the value function and the optimal strategies in {explicit or semi-explicit} form. We then prove that the optimal dividend payout strategy for bounded dividend rates is of threshold type, while for unbounded dividend rates it is of barrier type. {Under the diffusion approximation,} the optimal combination of proportional and excess-of-loss reinsurance is shown to be pure excess-of-loss reinsurance. We also show that the optimal level of risk ceded to the reinsurer decreases as the aggregate reserve level increases, {and that reinsurance is phased out line by line}. The optimal capital-injection strategy involves transferring reserves to prevent the ruin of one line. Finally, numerical examples are presented to illustrate these optimal strategies.
	\end{abstract}
	
\noindent\textbf{Keywords}: optimal dividends, excess-of-loss reinsurance, capital injection, collaborating lines, HJB equations, stochastic optimal control.
	
\section{Introduction}	

Optimal dividend payout problems have been a long-standing topic in the field of actuarial science, as dividends serve as a classical metric for assessing a company's value. In the seminal paper by \cite{definetti1957}, it is shown that the optimal strategy to distribute dividends is the one that maximizes the total discounted dividends paid to shareholders until bankruptcy time (i.e., the ruin time). This strategy is known as a \emph{barrier strategy}, where dividends are paid only when the reserve level exceeds a certain level, called the barrier. In this framework, the amount of dividends to be paid is the excess of the reserve above the barrier, effectively capping the reserves at this level. Since \cite{definetti1957}, numerous extensions and variations have been explored \citep[see, e.g.,][for comprehensive reviews]{albrecher2009,avanzi2009}. However, it is important to note that most of these papers focus on the univariate case, where the company manages reserves for a single line of business.

We consider an extension of the classical problem studied by \cite{definetti1957}, involving a company with multiple lines of business, each with its own reserves. This scenario introduces a level of complexity that requires a multivariate approach to understanding the dividend distribution. In this context, defining the ruin time is more complicated. In the univariate case, the ruin time is simply defined as the first time when a company's reserves fall below zero. However, in the multivariate case, there are several definitions, including: (i) \emph{first} ruin time: the first time when at least one reserve level falls below zero; (ii) \emph{sum} ruin time: the first time when the sum of the reserves across all lines falls below zero; and (iii) \emph{simultaneous} ruin time: the first time when all reserve levels fall below zero simultaneously. Consequently, much of the research in the multivariate case has focused on minimizing the probabilities of ruin \citep[see, e.g.,][Chapter XIII, and references therein]{asmussenbook2010}.

Although many studies focus on minimizing the probabilities of ruin in the multivariate case, fewer have specifically aimed to maximize the total discounted dividends paid. To our knowledge, the study by \citet{czarna2011} was the first to examine the maximization of total discounted dividends for two business lines, using reserve processes that follow the Cram\'{e}r-Lundberg (CL) model (i.e., a compound Poisson process). Subsequent studies using the CL model include \cite{liucheung2014}, \cite{albrecher2017}, \cite{azcue2019}, \cite{azcuemuler2021}, and \cite{strietzel2022}. Reserve levels modeled as diffusion processes in the multivariate setting have been explored in  \cite{gu2018}, \cite{grandits2019saj}, \cite{yang2025}, and \cite{boonen2025}.

In addition to maximizing the total discounted dividends, managing risks across multiple lines of business is essential to ensure the sustainability of dividend payouts. Risk control in the form of reinsurance has also been studied in the multivariate framework of optimal dividend problems. Proportional reinsurance has been explored in the multivariate case, as presented in \cite{czarna2011}, \cite{liucheung2014}, \cite{azcue2019}, \cite{strietzel2022}, \cite{yang2025}, and \cite{boonen2025}.

Given the complexities associated with the management of multiple business lines, effective risk management strategies are essential. One common approach is capital injection, which involves transferring reserves between lines to support a line that is at risk of its reserves falling below zero. This mechanism, sometimes referred to as \emph{collaboration}, not only mitigates the risk of insolvency for individual lines but also improves the company's capacity to distribute dividends. In the studies by \citet{albrecher2017}, \citet{gu2018}, and \citet{boonen2025}, the collaboration rule states that capital must be transferred from one line to another if a line is at risk of ruin, provided the transfer does not endanger the solvent line. This rule is modified in \cite{grandits2019saj}, in which the solvent line is not obliged to transfer capital to the insolvent line.

The research agenda of this paper is to determine the optimal dividend payout, reinsurance, and capital-injection strategies of an insurer with two business lines. We consider an insurer that jointly manages the reserves of the two lines and aims to maximize the weighted total discounted dividends paid until the first ruin time. The insurer's decisions are also subject to the following constraints: (1) the dividend rate may be bounded or unbounded; (2) the reinsurance contract combines proportional and excess-of-loss reinsurance; and (3) capital injections are used primarily to save one line from ruin (in the univariate sense).

We use the diffusion approximation of the reserve process following \cite{grandell1977} to obtain a tractable formulation of the problem. This leads to a smoother value function and significantly facilitates the analysis of the corresponding Hamilton-Jacobi-Bellman (HJB) equation. {In contrast, in the multivariate CL model, the value function is typically characterized as the unique viscosity solution of the HJB equation, and the optimal strategies are typically not available in explicit form \citep[see, e.g.,][]{azcue2019,azcuemuler2021}.} Due to the resulting regularity, the analytical characterizations of the optimal strategies can be derived. The diffusion approximation is a good approximation for large insurance portfolios, where the expected number of claims is sufficiently large and the individual claims are small relative to the insurance line's reserves \citep[see, e.g.,][]{taksar2003}. This setting is relevant to many insurance portfolios in practice. At the same time, the approximation focuses on the aggregate behavior of the reserve process rather than the individual claim-jump structure. Accordingly, the results obtained in this paper are interpreted within the diffusion framework.

To the best of our knowledge, no existing study has solved an optimal dividend
problem for two collaborating business lines under a diffusion approximation
while allowing, at the same time, bounded and unbounded dividend rates, a
combination of proportional and excess-of-loss reinsurance, and capital
injections between the lines. This combination of controls is important because
each component changes the structure of the optimal policy: dividend constraints
determine whether the optimal payout rule is of threshold or barrier type,
excess-of-loss reinsurance alters the optimal risk-retention mechanism, and
capital injections couple the two reserve processes. The closest related works
address more restrictive settings. \cite{boonen2025} consider a bivariate
diffusion model with capital injections, but allow only bounded dividend rates
and proportional reinsurance. \cite{boonen2026uncertainty} incorporate model
uncertainty into a bivariate problem with unbounded dividend rates, but again
restrict reinsurance to the proportional form. The present paper therefore
extends this literature by introducing excess-of-loss reinsurance into the
collaborative bivariate dividend problem and by providing a unified treatment
of both bounded and unbounded dividend regimes. We incorporate excess-of-loss reinsurance because it serves as a common alternative risk control mechanism in the univariate context of optimal dividend payout problems. Allowing the insurer to combine these two forms of reinsurance contracts gives greater flexibility in allocating risk between itself and the reinsurer. Proportional reinsurance reduces each claim by the same proportion, which provides protection against fluctuations in aggregate claims. On the other hand, excess-of-loss reinsurance provides protection against large claims while allowing the insurer to retain a larger share of smaller claims. Despite this additional flexibility, we show that the optimal combination reduces to \emph{pure excess-of-loss reinsurance}, which aligns with the results of \cite{centeno1985}, \cite{zhang2007}, and \cite{liang2011} in the univariate setting. {This reduction relies on the diffusion approximation, in which only the first two moments of the retained claims enter the reserve dynamics. In this setting, it extends to any per-claim reinsurance contract and to reinsurance premiums with a higher safety loading (see Remark \ref{remark: general treaty}).} Moreover, \cite{asmussen2000} show that optimal excess-of-loss reinsurance outperforms optimal proportional reinsurance. In the context of risk measures for terminal losses without dividends, \cite{aboagye2025} derive the optimal design of excess-of-loss reinsurance.

{Although the value function depends on the reserve levels only through their sum, the problem does not reduce to a univariate dividend problem. Each line has its own claim-size distribution, safety loading, and dividend weights, and chooses its own retention level. The two retention levels are linked through the marginal value of the aggregate reserve level, but they are capped at different levels $M_1$ and $M_2$. As a result, reinsurance is phased out line by line, which creates an additional threshold level compared with a univariate problem, and the reinsurance threshold level can lie between the two dividend threshold levels. Moreover, the dividend weights and the maximum dividend rates determine which line pays dividends at each aggregate reserve level, and the capital injections are needed to implement these dividend strategies and to prevent the ruin of either line. From a practical perspective, the model describes an insurance group or a
multi-line insurer that jointly manages the capital of its business lines or
subsidiaries and can transfer capital between them, while each line arranges
its own reinsurance program.}

We summarize the main contributions of this paper:
\begin{enumerate}	
	\item We show that the optimal dividend payout strategy for bounded dividend rates is a \emph{threshold} strategy in which dividends are continuously paid at a fixed rate whenever the aggregate reserve level exceeds a threshold. We also show that for unbounded dividend rates, the optimal strategy is a barrier strategy. These findings are consistent with the results presented in \cite{asmussen2000} within the univariate setting.
	
	\item In the case of bounded dividend rates, we identify three scenarios based on the following conditions: (a) the sum of the maximum dividend rates of the two lines {exceeds an explicit level determined by the model parameters (condition (ii) of Theorem \ref{thm:bdd w0<u1<u2})}, (b) {an additional explicit condition on the model parameters holds (condition (iii) of Theorem \ref{thm:bdd w0<u1<u2})}, and (c) {both claim-size distributions have bounded support}. This leads to three configurations for the reinsurance threshold level $w_0$ and the dividend threshold levels $u_1$ and $u_2$: (i) $w_0< u_1\leq u_2$; (ii) $u_1\leq w_0\leq u_2$; and (iii) $u_1\leq u_2 <w_0$. We also find that the reinsurance threshold level can exceed exactly one of the dividend threshold levels, which is a possibility ruled out in \cite{asmussen2000}. For unbounded dividend rates, we identify two scenarios based on whether {both claim-size distributions have bounded support, and we show that the dividend barrier is given by the same explicit formula in both scenarios}.
	
	\item We show that the optimal excess-of-loss retention level increases as the aggregate reserve level increases; equivalently, the amount of risk ceded to the reinsurer decreases. Moreover, {reinsurance is phased out line by line: the line with the smaller ratio $\frac{M_i}{\kappa_i}$ retains all of its risk first, while the other line continues to reduce its reinsurance until it also retains all of its risk. For unbounded dividend rates and bounded claim sizes, both lines retain all of their risk before the dividend barrier is reached, whereas, for bounded dividend rates, both lines may purchase reinsurance at every aggregate reserve level (see Theorem \ref{thm: bdd u1<u2<w0 or M infty}).}  
\end{enumerate}

The remainder of the paper is organized as follows. Section \ref{sec:model} introduces the model and the formulation of the problem. The gain of pure excess-of-loss reinsurance is discussed in Section \ref{sec: gain}. Section \ref{sec:bdd} presents the main results for the case of bounded dividend rates, while Section \ref{sec:ubdd} presents the case of unbounded dividend rates. Numerical examples are provided in Section \ref{sec:numerical}. The proofs of the main results are detailed in Section \ref{sec:proof}. Section \ref{sec:conclusion} concludes. {Appendix \ref{app:verification} provides the proof of the verification theorem.}

\section{Model}\label{sec:model}
	Consider a complete probability space $(\Omega,\mathcal F,\mathbb F,\mathbb P)$, where $\mathbb F:=\{\mathcal F_t\}_{t\geq 0}$ is a right-continuous, $\mathbb P$-completed filtration to which all processes defined below, including the Brownian motions and the Poisson processes, are adapted.
	
We consider an insurer with two collaborating lines of business and model each line's reserve using the classical Cram\'{e}r-Lundberg (CL) model:
	\begin{equation}
		P_i(t)=x_i+p_it-\sum_{k=1}^{\Lambda_i(t)} Y_{i,k},
	\end{equation}
where, for $i = 1,2$, $\Lambda_i:=\{\Lambda_i(t)\}_{t\geq 0}$ is a Poisson process with intensity $\lambda_i>0$ that represents claim arrivals, and $\{Y_{i,k}\}_{k \geq 1}$ are independent and identically distributed (i.i.d.), nonnegative claim sizes with common distribution $F_i$ and finite first and second moments $\widetilde{\mu}_i$ and $\widetilde{\sigma}^2_i$. The premium rate $p_i$ is assumed to be calculated under the expected value principle, i.e.,
	\begin{equation*}
		p_i=(1+\kappa_i)\lambda_i \widetilde{\mu}_i,
	\end{equation*}
	where $\kappa_i>0$ is the relative safety loading of Line $i$.
	
	The insurer has to make three decisions regarding the operation of each line:
	\begin{enumerate}
		\item \underline{Reinsurance decision}. The insurer arranges a combination of proportional and excess-of-loss reinsurance in the manner of \cite{centeno1985}: For each line, the insurer first chooses a proportional retention level $\theta_i(t)\in(0,1]$, and then chooses an excess-of-loss retention level $\pi_i(t)\in[0,\infty]$ such that the insurer's $k$th claim, net of proportional and excess-of-loss reinsurance, can be represented by $\min\{\theta_i(t)Y_{i,k},\pi_i(t)\}=\theta_i(t) Y_{i,k}\wedge \pi_i(t)$. We assume that the reinsurer applies the same expected value principle and relative safety loading $\kappa_i$ as the insurer; this is sometimes referred to as  ``cheap reinsurance". {This assumption is common in the literature on optimal dividend and reinsurance problems under the diffusion approximation \citep[see, e.g.,][]{asmussen2000}, and it leads to explicit solutions. By Remark \ref{remark: general treaty}, the reduction to pure excess-of-loss reinsurance in Section \ref{sec: gain} remains valid if the reinsurer applies a higher relative safety loading, although the explicit solutions in Sections \ref{sec:bdd} and \ref{sec:ubdd} then change.}
		
		\item \underline{Dividend payout decision}. The insurer chooses a dividend strategy to distribute profits to the shareholders of each line. Let $C_i(t)\geq 0$ be the cumulative dividends paid by Line $i$ at time $t$. We consider two types of dividend strategies: (1) an unbounded dividend rate: $C_i(t)$ is an arbitrary nonnegative and nondecreasing function; and (2) a bounded dividend rate: $C_i(t)$ satisfies $C_i(t)=\int_0^t c_i(s)ds$, where $c_i(s)\in[0,\overline c_i]$ and $\overline c_i>0$ is the maximum possible dividend rate. For unbounded dividend rates, we treat $C_1$ and $C_2$ as singular controls.
		
		\item \underline{Capital injection decision}. We assume that the insurer can inject capital from one line to the other without incurring any additional costs. The capital injection allows the insurer to prevent a line from bankruptcy by using the available resources within the company. Let $L_i(t)$ be the cumulative amount of capital transferred to Line $i$ from Line $3-i$. Moreover, we treat $L_1$ and $L_2$ as singular controls.
		
	\end{enumerate}
	We first derive the diffusion approximation of the reserve processes for fixed reinsurance decisions; that is, we set $\theta_i(t)\equiv \bar\theta_i$ and $\pi_i(t)\equiv \bar\pi_i$, for $i=1,2$. Under fixed $\bar\theta_i,\bar\pi_i$, the reserve level of Line $i$ is given by
	\begin{equation*}
		R^{\theta,\pi}_i(t)=x_i+p_i^{\theta,\pi}t-\sum_{k=1}^{\Lambda_i(t)} (\bar\theta_iY_{i,k}\wedge \bar\pi_i),
	\end{equation*}
	where
\begin{equation*}
\begin{aligned}
p_i^{\theta,\pi}=p_i-(1+\kappa_i)\lambda_i
    \mathbb{E}\!\left[Y_{i,k}-(\bar\theta_iY_{i,k}\wedge \bar\pi_i)\right] =(1+\kappa_i)\lambda_i
    \mathbb{E}\!\left[\bar\theta_iY_{i,k}\wedge \bar\pi_i\right].
\end{aligned}
\end{equation*}
	Following \cite{grandell1977}, the diffusion approximation of $\{R^{\theta,\pi}_i(t)\}_{t\geq 0}$ is given by
	\begin{equation*}
		dR_i(t)=\lambda_i\kappa_i\bar\theta_i\mu_i\left(\frac{\bar\pi_i}{\bar\theta_i}\right)dt+\sqrt{\lambda_i}\bar\theta_i\sigma_i\left(\frac{\bar\pi_i}{\bar\theta_i}\right)dW_i(t),
	\end{equation*}
	where $W_1:=\{W_1(t)\}_{t\geq 0}$ and $W_2:=\{W_2(t)\}_{t\geq 0}$ are independent Brownian motions and are independent of the Poisson processes $\Lambda_1$ and $\Lambda_2$, and
	\begin{equation}
		\mu_i(s):=\int_0^{s}[1-{F}_i(y)]dy\quad\mbox{and}\quad \sigma^2_i(s):=\int_0^{s}2y[1-{F}_i(y)]dy.
	\end{equation}
    {The diffusion approximation requires finite first and second moments of the claim sizes, which excludes heavy-tailed claim-size distributions with infinite variance. Note, however, that only the moments $\mu_i(\pi_i)$ and $\sigma_i^2(\pi_i)$ of the retained claims enter the dynamics, and these are finite for every finite excess-of-loss retention level $\pi_i$.}
    
    We now allow the insurer to adjust its reinsurance dynamically over time. Replacing the fixed parameters $\bar\theta_i$ and $\bar\pi_i$ by the corresponding processes $\{\theta_i(t)\}_{t\geq 0}$ and $\{\pi_i(t)\}_{t\geq 0}$, the reserve level of Line $i$ follows the dynamics given by 
    \begin{equation*}
		dR_i(t)=\lambda_i\kappa_i\theta_i(t)\mu_i\left(\frac{\pi_i(t)}{\theta_i(t)}\right)dt+\sqrt{\lambda_i}\theta_i(t)\sigma_i\left(\frac{\pi_i(t)}{\theta_i(t)}\right)dW_i(t).
    \end{equation*}
	To simplify notation, we assume $\lambda_i=1$. Write $\overline F_i(y):=1-F_i(y)$ and define
	\begin{equation*}
		M_i:=\inf\{y\geq 0: \overline F_i(y)=0\}\in[0,\infty],
	\end{equation*}
    with the convention that $M_i=\infty$ if $\overline F_i(y)>0$ for all $y\geq0$.
Throughout the paper, we assume that the claim-size distributions are nondegenerate, i.e., $\overline F_i(0)>0$. Consequently, $\sigma_i$ is continuous
and strictly increasing on $[0,M_i)$, with the convention that
$\sigma_i(M_i)=\widetilde\sigma_i$ when $M_i<\infty$ and
$\lim_{s\to\infty}\sigma_i(s)=\widetilde\sigma_i$ when $M_i=\infty$.
    By definition, $\mu_i(M_i)=\widetilde \mu_i$ and $\sigma_i^2(M_i)=\widetilde \sigma_i^2$. Moreover, the excess-of-loss retention level satisfies
$\pi_i(t)\in[0,M_i]$, where $[0,M_i]=[0,\infty]$ if $M_i=\infty$.

	Let $\theta_i:=\{\theta_i(t)\}_{t\geq 0}$, $\pi_i:=\{\pi_i(t)\}_{t\geq 0}$, $C_i:=\{C_i(t)\}_{t\geq 0}$, and $L_i:=\{L_i(t)\}_{t\geq 0}$ be the proportional reinsurance, excess-of-loss reinsurance, dividend payout, and capital-injection strategies, respectively, of Line $i$, for $i=1,2$. We formally define admissible strategies as follows:
	\begin{definition}
		A strategy $u$ is said to be admissible if $u$ is adapted to the filtration $\mathbb{F}$ and satisfies the following conditions:
		\begin{itemize}
			\item[(i)] $\theta_i(t)\in(0,1]$ and $\pi_i(t)\in[0,M_i]$ for $i=1,2$ and $t\geq 0$;
			\item[(ii)] $C_i$ and $L_i$ are nonnegative, nondecreasing, and right-continuous with left limits, for $i=1,2$.
                \item[(iii)] in the bounded-dividend case,
    \[
        C_i(t)=\int_0^t c_i(s)\,\dd s,\qquad c_i(s)\in[0,\overline c_i],
        \quad i=1,2;
    \]

    \item[(iv)] {if $X_i(t-)\geq 0$ for $i=1,2$, then $X_i(t)\geq 0$ for $i=1,2$ after each jump of $(C_1,C_2,L_1,L_2)$ at time $t$, and $L_1$ and $L_2$ do not increase at the same time, where $X_i$ is the reserve level of Line $i$ given in \eqref{eqn:X og} below; that is,} dividends and capital transfers cannot create insolvency.
		\end{itemize}
		Denote by $\mathcal U$ the set of all admissible strategies.
	\end{definition}
    
    Given an admissible control $u:=(\theta_1,\theta_2,\pi_1,\pi_2,C_1,C_2,L_1,L_2)$, the controlled reserve process of Line $i$, denoted by $X_i:=X_i^u$, follows the dynamics 
	\begin{equation}\label{eqn:X og}
		dX_i(t)=\kappa_i\theta_i(t)\mu_i\left(\frac{\pi_i(t)}{\theta_i(t)}\right) dt+\theta_i(t)\sigma_i\left(\frac{\pi_i(t)}{\theta_i(t)}\right)dW_i(t)-dC_i(t)+dL_i(t)-dL_{3-i}(t),
	\end{equation}
	where $X_i(0)=x_i\geq 0$ is the initial reserve of Line $i$. We define the corresponding ruin time by
	\begin{equation}
		\tau:=\inf \{t\geq 0 : \min\{X_1(t),X_2(t)\}<0\}.
	\end{equation}
    That is, $\tau$ is the first time at which at least one reserve level becomes negative.
    When the reserve pair reaches $(0,0)$, no further dividends can be generated before ruin.
    Indeed, if the insurer chooses positive retention at $(0,0)$, then at least one line has
    nonzero diffusion volatility and the process exits the nonnegative quadrant immediately
    with probability one; if the insurer chooses zero retention for both lines, then both the
    drift and volatility vanish and the reserve pair remains at $(0,0)$. Hence, without loss
    of value, we treat $(0,0)$ as an absorbing terminal state.

    \begin{remark}
        The constraint on the proportional reinsurance variables, $\theta_i$, implies that the business lines cannot cede all risk to the reinsurer; that is, full reinsurance is not allowed through the proportional component. On the other hand, the excess-of-loss retention can take any nonnegative value, including infinity. We say that $F_i$ has bounded support if $M_i<\infty$, and unbounded support otherwise. Under pure excess-of-loss reinsurance, we can interpret $\pi_i(t)=M_i$ as Line $i$ retaining all risk; that is, Line $i$ does not engage in any reinsurance.
    \end{remark}
	
	The goal of the insurer is to determine optimal dividend, (proportional and excess-of-loss) reinsurance, and capital-injection strategies that maximize the weighted total discounted dividend payouts from both lines up to the ruin time. That is, the insurer is faced with the following maximization problem:
	\begin{equation}\label{eqn:V}
		\begin{aligned}
			V(x_1,x_2)&:=\sup_{u\in\mathcal U}J(x_1,x_2;u),\\
			J(x_1,x_2;u)&:=\mathbb{E}\left[a\int_0^{\tau} e^{-\delta t}dC_1(t)+(1-a)\int_0^{\tau} e^{-\delta t}dC_2(t)\middle| X_1(0)=x_1,X_2(0)=x_2\right],
		\end{aligned}
	\end{equation}
	where $a\in(0,1)$ is a weighting factor that reflects the relative importance of Line $1$ for the company, $\delta>0$ is the discount factor, and the expectation $\mathbb{E}$ is taken at time zero. We denote the objective function by $J$ and the value function by $V$. 

    \begin{remark}\label{remark on V}
        From the definition of the value function $V$, we can see that $V$ is increasing in both arguments $x_1$ and $x_2$; that is, $V$ increases as the initial reserves increase.
    \end{remark}

	\section{The Gain of Pure Excess-of-Loss Reinsurance}\label{sec: gain}
	
	In this section, we will show that an optimal reinsurance, dividend payout, and capital-injection strategy is of the form $u_0:=(1,1,\pi^0_1,\pi_2^0,C_1^0,C_2^0,L_1^0,L_2^0)\in\mathcal U$; that is, the optimal combination of proportional and excess-of-loss reinsurance strategies reduces to pure excess-of-loss reinsurance.

	\begin{lemma}\label{lemma:h is inc}
		Define 
		\begin{equation*}
			h_i(s):=\frac{\sigma_i^2(s)}{\mu^2_i(s)}.
		\end{equation*}
		Then, $h_i(s)$ is an increasing function for $s>0$.
	\end{lemma}
    \begin{proof}
The functions $\mu_i$ and $\sigma_i^2$ are absolutely continuous. Wherever the derivatives exist, we have $\mu_i'(s)=\overline F_i(s)$ and $\frac{\dd}{\dd s}\sigma_i^2(s)=2s\overline F_i(s)$.
Hence, for a.e. $s>0$,
\begin{equation*}
h'_i(s)=\frac{2\overline F_i(s)\left[s\mu_i(s)-\sigma_i^2(s)\right]}{\mu_i^3(s)}.
\end{equation*}
Since $0\leq Y_{i,k}\wedge s\leq s$, we have
\[
\mathbb{E}[(Y_{i,k}\wedge s)^2]
\leq s\mathbb{E}[Y_{i,k}\wedge s],
\]
and therefore
\[
\sigma_i^2(s)\leq s\mu_i(s).
\]
Thus $h_i'(s)\geq0$ for a.e. $s>0$. Since $h_i$ is absolutely continuous on compact subintervals of $(0,\infty)$, it follows that $h_i$ is increasing.
\end{proof}

	\begin{proposition}\label{prop:pure XL}
		For any fixed admissible control $u=(\theta_1,\theta_2,\pi_1,\pi_2,C_1,C_2,L_1,L_2)$, there exists an admissible control $u_0=(1,1,\pi^0_1,\pi_2^0,C_1^0,C_2^0,L_1^0,L_2^0)$ such that
		\begin{equation*}
			J(x_1,x_2;u)\leq J(x_1,x_2;u_0).
		\end{equation*}
	\end{proposition}
	\begin{proof}
		For any fixed $\theta_i(t)$, $\pi_i(t)$, $C_i(t)$, and $L_i(t)$, there exists $\pi_i^0(t)\in[0,M_i]$ such that
\begin{equation*}
\theta_i(t)\sigma_i\left(\frac{\pi_i(t)}{\theta_i(t)}\right)=\sigma_i(\pi_i^0(t)),\quad i=1,2.
\end{equation*}
Indeed, $\sigma_i$ is continuous and strictly increasing on $[0,M_i)$, with
$\sigma_i(0)=0$. If $M_i<\infty$, then $\sigma_i(M_i)=\widetilde\sigma_i$.
If $M_i=\infty$, the equality is understood in the limiting sense; since
$\theta_i(t)\leq1$, the left-hand side is at most $\widetilde\sigma_i$ and can be
matched by a finite $\pi_i^0(t)$ unless it equals $\widetilde\sigma_i$, in which case
we set $\pi_i^0(t)=\infty$. Since $\theta_i(t)\in(0,1]$, we have $\sigma_i(\pi_i^0(t))\leq\sigma_i\left(\pi_i(t)/\theta_i(t)\right)$. Since $\pi_i^0(t)\leq M_i$ and $\sigma_i$ is strictly increasing, it follows that
$\pi_i(t)/\theta_i(t)\ge \pi_i^0(t)$, with the same limiting convention when
$\pi_i^0(t)=\infty$. By Lemma \ref{lemma:h is inc},
		\begin{equation*}
			h_i(\pi_i^0(t))\leq h_i\left(\frac{\pi_i(t)}{\theta_i(t)}\right).
		\end{equation*}
		It follows that 
		\begin{equation*}
			\theta_i^2(t)=\frac{\sigma^2_i(\pi_i^0(t))}{\sigma^2_i\left(\frac{\pi_i(t)}{\theta_i(t)}\right)}\leq \frac{\mu^2_i(\pi_i^0(t))}{\mu^2_i\left(\frac{\pi_i(t)}{\theta_i(t)}\right)},
		\end{equation*}
		which implies that
		\begin{equation*}
			\theta_i(t)\mu_i\left(\frac{\pi_i(t)}{\theta_i(t)}\right)\leq \mu_i(\pi_i^0(t)).
		\end{equation*}

        Define 
        \begin{equation*}
            \Delta_i(t):= \kappa_i\left[\mu_i(\pi_i^0(t))-\theta_i(t)\mu_i\left(\frac{\pi_i(t)}{\theta_i(t)}\right)\right]\geq 0.
        \end{equation*}
        We first consider the unbounded-dividend rate case. Let 
		\begin{equation*}
			\begin{aligned}
				C^0_i(t)&:=C_i(t)+\int_0^t\Delta_i(s)ds \geq C_i(t),\\
				L^0_i(t)&:=L_i(t).
			\end{aligned}
		\end{equation*}
		Then $u_0$ is an admissible control and
		\begin{equation*}
			\begin{aligned}
			dX^{u_0}_i(t)
			&=\kappa_i\mu_i(\pi_i^0(t))\dd t+ \sigma_i(\pi_i^0(t)) \dd W_i(t) - \dd  C^0_i(t) + \dd L^0_i(t) - \dd L^0_{3-i}(t)\\
			&=\kappa_i\theta_i(t)\mu_i\left(\frac{\pi_i(t)}{\theta_i(t)}\right)\dd t+ \theta_i(t)\sigma_i\left(\frac{\pi_i(t)}{\theta_i(t)}\right) \dd W_i(t) - \dd C_i(t) + \dd L_i(t) - \dd L_{3-i}(t).
			\end{aligned}
		\end{equation*}
		Hence, {$X^{u_0}_i(t)=X^{u}_i(t)$ for all $t$, almost surely, since both processes are driven by the same Brownian motion,} while $C^0_i(t) \geq C_i(t)$ and $L^0_i(t)=L_i(t)$ for all $t$. The result then follows.

        We now consider the bounded-dividend rate case. Recall from the admissibility definition that $C_i(t)=\int_0^t c_i(s) \dd s$, where $c_i(s)\in[0,\overline c_i]$. Let
        \begin{equation*}
			\begin{aligned}
				c^0_i(t)&:=\min\{c_i(t)+\Delta_i(t),\overline c_i\},\\
                C_i^0(t)&:=\int_0^tc_i^0(s) \,\dd s,\\
				L^0_i(t)&:=L_i(t).
			\end{aligned}
		\end{equation*}
        By construction, $c_i^0(t)\in[0,\overline c_i]$, so $C_i^0$ is admissible. Moreover, $c_i^0(t)\geq c_i(t)$, and hence $C_i^0(t)\geq C_i(t)$ for all $t\geq 0$.

        Unlike in the unbounded-dividend rate case, the entire drift term $\Delta_i(t)$ is not necessarily distributed as dividends because of the upper bound imposed on the dividend rate. The portion of $\Delta_i(t)$ that cannot be distributed as dividends remains in the reserve of Line $i$. Then,
        \begin{equation*}
            \dd X_i^{u_0}(t) - \dd X_i^{u}(t) = \left[c_i(t) + \Delta_i(t) -\overline c_i\right]^+ \dd t\geq 0.
        \end{equation*}
        As a result, the ruin time under $u_0$ is no smaller than that under the strategy $u$; that is, $\tau^{u_0}\geq \tau^u$. Since $C_i^0(t)\geq C_i(t)$ also holds, we then obtain the result. The proof is complete.
	\end{proof}
	
Write $u^{\rm XL}:=(1,1,\pi_1,\pi_2,C_1,C_2,L_1,L_2)\in\mathcal U$
for a pure excess-of-loss strategy, and let $V^{\rm XL}$ denote the value
function obtained by restricting the admissible controls to such strategies.
Proposition \ref{prop:pure XL} implies
\[
    V(x_1,x_2)=V^{\rm XL}(x_1,x_2)
    =\sup_{u^{\rm XL}\in\mathcal U}J(x_1,x_2;u^{\rm XL}).
\]
Hence, in the verification arguments below, it is sufficient to prove
optimality within the pure excess-of-loss subclass. The equality with the value
of the original proportional-plus-excess-of-loss problem then follows from
Proposition \ref{prop:pure XL}. {Note that the strategy $u_0$ differs from $u$ not only in the proportional reinsurance component, but also in the excess-of-loss retention levels and the dividend strategies. Hence, it is not required that setting $\theta_i=1$ while keeping the other controls unchanged improves the objective function.}

    {\begin{remark}\label{remark: general treaty}
        Proposition \ref{prop:pure XL} relies on the diffusion approximation, in which the reinsurance contract affects the reserve dynamics only through the first two moments of the retained claims. In this setting, the argument extends to any (possibly time-dependent) per-claim reinsurance contract under which Line $i$ retains $h_i(Y_{i,k})$ of its $k$th claim, where $h_i$ is a measurable function with $0\leq h_i(y)\leq y$; the combination of proportional and excess-of-loss reinsurance corresponds to $h_i(y)=\theta_iy\wedge\pi_i$. Indeed, for any $s\geq 0$ and $y\geq 0$, we have
        \begin{equation*}
            h_i(y)^2-(y\wedge s)^2-2s\left[h_i(y)-y\wedge s\right]=
			\begin{cases}
				\left(h_i(y)-y\right)\left(h_i(y)+y-2s\right)\geq 0&\mbox{if $y\leq s$,}\\
				\left(h_i(y)-s\right)^2\geq0&\mbox{if $y>s$,}
			\end{cases}
        \end{equation*}
        so that $\mathbb E[h_i(Y_{i,k})^2]-\sigma_i^2(s)\geq 2s\left(\mathbb E[h_i(Y_{i,k})]-\mu_i(s)\right)$. Choosing $s=\pi_i^0$ such that $\sigma_i^2(\pi_i^0)=\mathbb E[h_i(Y_{i,k})^2]$ yields $\mathbb E[h_i(Y_{i,k})]\leq\mu_i(\pi_i^0)$. That is, the excess-of-loss contract with the same second moment of the retained claims leads to a larger drift, and the proof of Proposition \ref{prop:pure XL} applies. This is the dynamic counterpart of the classical property that, among contracts with the same expected retained claim, excess-of-loss reinsurance minimizes the second moment of the retained claim. This property is the per-claim version of the result that stop-loss reinsurance minimizes the variance of the retained risk among all reinsurance contracts with the same reinsurance premium \citep[see][Theorem 1]{kahn1961}. Moreover, the same argument applies if the reinsurer uses a relative safety loading $\eta_i\geq \kappa_i$, since the drift of Line $i$ then equals $(\kappa_i-\eta_i)\widetilde\mu_i+\eta_i\mathbb E[h_i(Y_{i,k})]$, which is increasing in $\mathbb E[h_i(Y_{i,k})]$. However, this result relies on the diffusion approximation and need not hold in the CL model, where the whole distribution of the retained claims matters.
    \end{remark}}

    By Proposition \ref{prop:pure XL}, it suffices to consider pure excess-of-loss reinsurance strategies. We now present a verification theorem for the problem in \eqref{eqn:V} restricted to such strategies, which applies to both bounded and unbounded dividend rates. It is stated in terms of the aggregate reserve level, since the value functions obtained in Sections \ref{sec:bdd} and \ref{sec:ubdd} have the form $V(x_1,x_2)=w(x_1+x_2)$. The proof is given in Appendix \ref{app:verification}.

    {\begin{theorem}[Verification theorem]\label{thm:verification}
        Let $w\in C^2((0,\infty))\cap C([0,\infty))$ be increasing and concave with $w(0)=0$, and let $a\leq \frac{1}{2}$.
        \begin{enumerate}
            \item[(i)] (Bounded dividend rates.) Suppose that $w$ is bounded and that, for all $x>0$,
            \begin{equation*}
                \sup_{\pi_i\in[0,M_i],c_i\in[0,\overline c_i]}\left\{\sum_{i=1}^2\left[\left(\kappa_i\mu_i(\pi_i)-c_i\right)w'(x)+\frac{1}{2}\sigma_i^2(\pi_i)w''(x)\right]-\delta w(x)+ac_1+(1-a)c_2\right\}=0.
            \end{equation*}
            Then, $w(x_1+x_2)\geq J(x_1,x_2;u)$ for every admissible strategy $u$ with bounded dividend rates and $\theta_i\equiv 1$, $i=1,2$. Moreover, $w(x_1+x_2)=J(x_1,x_2;u^*)$, where $u^*$ uses the feedback controls $\pi_i^*(t)=\Pi_i(X_1^*(t)+X_2^*(t))$ and $c_i^*(t)=D_i(X_1^*(t)+X_2^*(t))$, $i=1,2$, with $(\Pi_1(x),\Pi_2(x),D_1(x),D_2(x))$ a maximizer of the supremum above and $X_i^*$ the reserve level of Line $i$ under $u^*$, and capital transfers that keep both reserve levels nonnegative as long as the aggregate reserve level is positive. Hence, $w(x_1+x_2)$ is the value function of the problem in \eqref{eqn:V} when the dividend strategies are restricted to $dC_i(t)=c_i(t) dt$ with $c_i(t)\in[0,\overline c_i]$, $i=1,2$; this problem is restated in \eqref{eqn:V bdd} in Section \ref{sec:bdd}.
            
            \item[(ii)] (Unbounded dividend rates.) Suppose that there exists $b>0$ such that $w'(x)>1-a$ for $x\in(0,b)$, $w'(x)=1-a$ for $x\geq b$, and
            \begin{equation*}
                \sup_{\pi_i\in[0,M_i]}\left\{\sum_{i=1}^2\left[\kappa_i\mu_i(\pi_i)w'(x)+\frac{1}{2}\sigma_i^2(\pi_i)w''(x)\right]\right\}-\delta w(x)\leq 0,\quad x>0,
            \end{equation*}
            with equality for $x\in(0,b)$. Then, $w(x_1+x_2)\geq J(x_1,x_2;u)$ for every admissible strategy $u$ with $\theta_i\equiv 1$, $i=1,2$. Moreover, $w(x_1+x_2)=J(x_1,x_2;u^*)$, where $u^*$ uses the feedback controls $\pi_i^*(t)=\Pi_i(X_1^*(t)+X_2^*(t))$, $i=1,2$, with $(\Pi_1(x),\Pi_2(x))$ a maximizer of the supremum above and $X_i^*$ the reserve level of Line $i$ under $u^*$, and the dividend and capital-injection strategy of Theorem \ref{thm: ubdd uniform strat} with barrier $u_1=b$. Hence, $w(x_1+x_2)$ is the value function in \eqref{eqn:V}.
        \end{enumerate}
    \end{theorem}}
	
	\section{Bounded Dividend Rates}\label{sec:bdd}
	
	In this section, we consider the case in which the dividend rate of Line $i$ is bounded above by some constant $\overline c_i\in (0,\infty)$. In this case, we can represent $C_i(t)$ as
	\begin{equation*}
		C_i(t)=\int_0^t c_i(s)\, \dd s,\quad c_i(s)\in [0,\overline c_i].
	\end{equation*}
	We henceforth represent the dividend control by the rates $c_1$ and $c_2$. Using the results of the previous section, we can then rewrite the reserve process of Line $i$ given in \eqref{eqn:X og} as
	\begin{equation*}
		dX_i(t)=\left[\kappa_i\mu_i(\pi_i(t))-c_i(t)\right]dt+\sigma_i(\pi_i(t))dW_i(t)+dL_i(t)-dL_{3-i}(t),
	\end{equation*}
	 and the value function defined in \eqref{eqn:V} as
	 \begin{equation}\label{eqn:V bdd}
	 	V(x_1,x_2)=\sup_{u\in\mathcal U_{\rm bdd}^{\rm XL}}\mathbb{E}\left[a\int_0^{\tau} e^{-\delta t}c_1(t)dt+(1-a)\int_0^{\tau} e^{-\delta t}c_2(t)dt\right],
	 \end{equation}where $\mathcal U_{\rm bdd}^{\rm XL}$ denotes the admissible class after
restricting to pure excess-of-loss reinsurance and bounded dividend rates.
    \begin{remark}
        For the case of bounded dividend rates, we observe that, as the initial reserves become sufficiently large, the lines pay the maximum dividend rates and $V$ approaches the limit $\frac{a\overline c_1+(1-a)\overline c_2}{\delta}$.
    \end{remark}
	 For all $\pi_i\in [0,M_i]$ and $c_i\in[0,\overline c_i]$, define the generator $\mathcal L^{\pi,c}(\phi)$ for some $C^{2,2}$ function $\phi$ by
	 \begin{equation*}
	 	\mathcal L^{\pi,c}(\phi)=\sum_{i=1}^2\left[\left[\kappa_i\mu_i(\pi_i)-c_i\right]\frac{\partial \phi}{\partial x_i}+\frac{1}{2}\sigma_i^2(\pi_i)\frac{\partial^2 \phi}{\partial x_i^2}\right]-\delta \phi+ac_1+(1-a)c_2.
	 \end{equation*}
	 We solve the problem in \eqref{eqn:V bdd} via the dynamic programming principle. Using arguments similar to those in \citet[Chapter 2.5.1]{schmidli2007book}, we can characterize the value function $V$ as a solution to the following Hamilton-Jacobi-Bellman (HJB) equation:
	 	\begin{equation}\label{eqn:hjb bdd}
	 		\sup\left\{\sup_{\pi_i\in[0,M_i],c_i\in[0,\overline c_i]}\mathcal L^{\pi,c}(V),\frac{\partial V}{\partial x_1}-\frac{\partial V}{\partial x_2},\frac{\partial V}{\partial x_2}-\frac{\partial V}{\partial x_1}\right\}=0,
	 	\end{equation}
	 	with boundary condition $V(0,0)=0$.
    If we can find a classical solution to the HJB equation \eqref{eqn:hjb bdd} {of the form $V(x_1,x_2)=w(x_1+x_2)$, then, by Theorem \ref{thm:verification}(i), it is equal to the value function. Since the solutions obtained below satisfy $w'(x)\to\infty$ as $x\downarrow 0$ (see Remark \ref{remark: g near zero}), Theorem \ref{thm:verification} only requires twice continuous differentiability on $(0,\infty)$.}

	We begin the analysis by supposing that a classical solution $V$ to the HJB equation in \eqref{eqn:hjb bdd} exists. Since both $\frac{\partial V}{\partial x_1}-\frac{\partial V}{\partial x_2} \leq 0$ and $\frac{\partial V}{\partial x_2}-\frac{\partial V}{\partial x_1}\leq 0$ must hold by \eqref{eqn:hjb bdd}, it follows that $\frac{\partial V}{\partial x_1}=\frac{\partial V}{\partial x_2}$. Hence, there exists a univariate function $g:\mathbb R_{+}\to \mathbb R$ such that
	\begin{equation*}
		g(x)=V(x_1,x_2), \text{ with } x:=x_1+x_2\geq 0.
	\end{equation*} 
	We then have the following
	\begin{equation*}
		g'(x)=\frac{\partial V}{\partial x_i}(x_1,x_2) \quad\mbox{and} \quad g''(x)=\frac{\partial^2 V}{\partial x_i\partial x_j}(x_1,x_2),\quad i,j=1,2.
	\end{equation*}
	
	To solve the optimization problem $\sup_{\pi_i\in[0,M_i],c_i\in[0,\overline c_i]}\mathcal L^{\pi,c}(V)$, we first isolate the optimization over the dividend payout variable $c_i$:
	\begin{equation*}
		\sup_{c_1\in[0,\overline c_1],c_2\in[0,\overline c_2]} \left\{(a-g'(x))c_1+(1-a-g'(x))c_2\right\}.
	\end{equation*} 
This yields the following candidate optimal dividend rates
	\begin{equation*}
		\widehat c_1(x)=\begin{cases}
			0,&\mbox{if $g'(x)>a$},\\
			\overline c_1,&\mbox{if $g'(x)<a$},
		\end{cases}\quad\mbox{and}\quad
		\widehat c_2(x)=\begin{cases}
			0,&\mbox{if $g'(x)>1-a$},\\
			\overline c_2,&\mbox{if $g'(x)<1-a$}.
		\end{cases}
	\end{equation*}
	Define the two constants $u_1$ and $u_2$ by
	\begin{equation}\label{eqn:defn of u1 & u2}
		u_1:=\inf\{u:g'(u)=1-a\}\quad\mbox{and}\quad u_2:=\inf\{u:g'(u)=a\}.
	\end{equation}
	We first conjecture that $g$ is strictly concave; that is,
$g''(x)<0$ for all $x>0$. By symmetry and without loss of generality, we assume $a\leq \frac{1}{2}$. Since $g$ is concave and $a\leq \frac{1}{2}$, we have $u_1\leq u_2$. We then have the following candidate for the optimal dividend strategies:
	\begin{equation*}
		\left(\widehat c_1(x),\widehat c_2(x)\right)=\begin{cases}
			(0,0),&\mbox{if $x<u_1$},\\
			(0,\overline c_2),&\mbox{if $u_1<x<u_2$},\\
			(\overline c_1,\overline c_2),&\mbox{if $x>u_2$}.
		\end{cases}
	\end{equation*}
	
	Next, we solve the optimization over the reinsurance variable $\pi_i$:
	\begin{equation}\label{optim reinsurance}
		\sup_{\pi_1\in[0,M_1],\pi_2\in[0,M_2]} \sum_{i=1}^2\left[\kappa_i\mu_i(\pi_i)g'(x)+\frac{1}{2}\sigma_i^2(\pi_i)g''(x)\right].
	\end{equation}
	We then have the following candidate maximizer:
	\begin{equation}\label{eqn:optimal pi}
		\widehat \pi_1(x) =-\kappa_1\frac{g'(x)}{g''(x)} \quad\mbox{and}\quad \widehat \pi_2(x)=-\kappa_2 \frac{g'(x)}{g''(x)}.
	\end{equation}
	We can express $\widehat \pi_2$ as
	\begin{equation*}
		\widehat \pi_2 (x) = \frac{\kappa_2}{\kappa_1} \widehat \pi_1(x).
	\end{equation*}

    {The objective function in \eqref{optim reinsurance} is the sum of a function of $\pi_1$ and a function of $\pi_2$, so the optimization can be carried out separately for each line. For Line $i$, the derivative of the objective function with respect to $\pi_i$ is given by
    \begin{equation*}
        \overline F_i(\pi_i)\left[\kappa_i g'(x)+\pi_i g''(x)\right].
    \end{equation*}
    Since $g''(x)<0$, this derivative is nonnegative for $\pi_i<\widehat \pi_i(x)$ and  non-positive for $\pi_i>\widehat\pi_i(x)$. Moreover, $\overline F_i(\pi_i)=0$ for $\pi_i\geq M_i$, so that $\mu_i$ and $\sigma_i^2$ are constant on $[M_i,\infty)$, and any retention level $\pi_i\geq M_i$ is equivalent to $\pi_i=M_i$. Hence, the optimal retention level of Line $i$ is given by
    \begin{equation}\label{eqn:optimal pi constrained}
        \Pi_i(x)=\min\left\{M_i,\widehat\pi_i(x)\right\},\quad i=1,2.
    \end{equation}
    Here, $\Pi_i$ is a function of the aggregate reserve level, and the optimal retention level of Line $i$ is applied in feedback form; that is, $\pi_i^*(t)=\Pi_i(X_1^*(t)+X_2^*(t))$, where $X_i^*$ denotes the reserve level of Line $i$ under the optimal strategy. Similarly, we write the optimal dividend rates in feedback form as $c_i^*(t)=D_i(X_1^*(t)+X_2^*(t))$ for some functions $D_i:[0,\infty)\to[0,\overline c_i]$, $i=1,2$. In particular, the optimal retention level of each line depends only on its own constraint. If one line retains all of its risk, the retention level of the other line still increases with $\widehat\pi_1(x)$ until it reaches its own upper bound.
    }

	Since the capital-injection decision $(L_1,L_2)$ is represented as a singular control, the first-order conditions do not apply, unlike for the reinsurance and dividend controls. We will determine the optimal capital-injection strategy later by examining the boundaries of specified regions.
	
    {Define
    \begin{equation*}
        \widetilde M:=\max\left\{M_1,\frac{\kappa_1}{\kappa_2}M_2\right\}\leq\infty.
    \end{equation*}
    By \eqref{eqn:optimal pi}, $\widehat\pi_1(x)\geq \widetilde M$ if and only if $\widehat\pi_1(x)\geq M_1$ and $\widehat\pi_2(x)\geq M_2$; that is, both lines retain all of their risk.}

	We define $w_0$ such that it satisfies
	\begin{equation}\label{eqn:defn of w0}
		\widehat\pi_1(w_0)={\widetilde M}.
	\end{equation} 
The existence of $w_0$ will be studied in Theorems \ref{thm:bdd w0<u1<u2}, \ref{thm:bdd u1<w0<u2}, and \ref{thm: bdd u1<u2<w0 or M infty}. By definition, we can interpret $w_0$ as the aggregate reserve level at which the insurer chooses \emph{zero reinsurance} for {both lines}. We refer to $w_0$ as the \emph{reinsurance threshold level}. {If $\frac{M_1}{\kappa_1}<\frac{M_2}{\kappa_2}$, then Line 1 retains all of its risk once $\widehat\pi_1(x)\geq M_1$, while Line 2 continues to purchase reinsurance until the aggregate reserve level reaches $w_0$. The case $\frac{M_1}{\kappa_1}>\frac{M_2}{\kappa_2}$ is similar, with the roles of the lines exchanged, and if $\frac{M_1}{\kappa_1}=\frac{M_2}{\kappa_2}$, both lines retain all of their risk from $w_0$ onward. Hence, no assumption on the ordering of $\frac{M_1}{\kappa_1}$ and $\frac{M_2}{\kappa_2}$ is required.}
	
We introduce the following notations that will be used in the discussion:
	\begin{equation}\label{eqn:notations}
		\begin{aligned}
			\overline N_1(y)&:=\kappa_1\mu_1(y)+\kappa_2\mu_2\left(\frac{\kappa_2}{\kappa_1}y\right),\qquad\qquad\qquad \qquad\overline N_2(y):=\sigma_1^2(y)+\sigma_2^2\left(\frac{\kappa_2}{\kappa_1}y\right),\\
			\overline\gamma_{2\pm}(y)&:=\frac{-\overline N_1(y)\pm \sqrt{\overline N_1(y)^2+2\delta\overline N_2(y) }}{\overline N_2(y)},\\
			\overline\gamma_{3\pm}(y)&:=\frac{-(\overline N_1(y)-\overline c_2)\pm \sqrt{(\overline N_1(y)-\overline c_2)^2+2\delta\overline N_2(y) }}{\overline N_2(y)},\\
			\overline\gamma_{4-}(y)&:=\frac{-(\overline N_1(y)-\overline c_1-\overline c_2)- \sqrt{(\overline N_1(y)-\overline c_1-\overline c_2)^2+2\delta\overline N_2(y) }}{\overline N_2(y)}.
		\end{aligned}
	\end{equation} 
	Write $N_i:=\overline N_i({\widetilde M})$, $\gamma_{2\pm}:=\overline \gamma_{2\pm}({\widetilde M})$, $\gamma_{3\pm}:=\overline \gamma_{3\pm}({\widetilde M})$, and $\gamma_{4-}:=\overline \gamma_{4-}({\widetilde M})$. {Note that $\overline N_1$ and $\overline N_2$ are constant on $[\widetilde M,\infty)$, so $N_1=\kappa_1\widetilde\mu_1+\kappa_2\widetilde\mu_2$ and $N_2=\widetilde\sigma_1^2+\widetilde\sigma_2^2$ whenever $\widetilde M<\infty$.} In addition, we define two functions $\psi,\zeta:(\frac{1-a}{\gamma_{3-}},0)\mapsto\mathbb R$ by
	\begin{equation}\label{eqn:psi}
		\begin{aligned}
			\psi(z)&:=(1-a-\gamma_{3-}z)e^{\gamma_{3+} \, \zeta(z)}+\gamma_{3-}ze^{\gamma_{3-} \, \zeta(z)}-a,\\
			\zeta(z)&:=\frac{1}{\gamma_{3+}-\gamma_{3-}}\ln\left[\frac{\gamma_{3-}(\gamma_{4-}-\gamma_{3-})z}{(1-a-\gamma_{3-}z)(\gamma_{3+}-\gamma_{4-})}\right].
		\end{aligned}    
	\end{equation}
	
	We can now present the main results. The reinsurance threshold level $w_0$ plays an important role in the results. Its position relative to the dividend threshold levels $u_1$ and $u_2$ yields three mutually exclusive cases: (i) $w_0< u_1\leq u_2$ (Theorem \ref{thm:bdd w0<u1<u2}), (ii) $u_1\leq w_0\leq u_2$ (Theorem \ref{thm:bdd u1<w0<u2}), and (iii) $u_1\leq u_2 <w_0$ (Theorem \ref{thm: bdd u1<u2<w0 or M infty}). Each case has its corresponding analytical form of the value function and a dividend-reinsurance strategy $(\pi_1^*,\pi_2^*,c_1^*,c_2^*)$. Since the capital-injection strategies $L_1^*$ and $L_2^*$ are modeled as singular controls, we obtain a uniform optimal strategy as discussed in Theorem \ref{thm:op_L}. The main results also cover whether {both claim-size distributions have bounded support (i.e., $\widetilde M<\infty$) or not (i.e., $\widetilde M=\infty$).}

	\begin{theorem}\label{thm:bdd w0<u1<u2}
		Suppose (i) ${\widetilde M}<\infty$, (ii) $\overline c_1+\overline c_2\geq N_1-\frac{\kappa_1 N_2}{2{\widetilde M}}+\frac{\delta {\widetilde M}}{\kappa_1}$ and (iii) $\psi(\alpha_{LB}+)< 0$, where $\psi$ is defined by \eqref{eqn:psi} {(equivalently, $\gamma_{3-}+\frac{\kappa_1}{\widetilde M}\geq 0$ or $K_{2+}>1-a$, where $K_{2+}$ is defined in \eqref{eqn:K2+ u1<w0<u2})} and 
		\begin{equation}\label{eqn:alpha w0<u1<u2}
			\begin{aligned}
            \alpha_{LB}&:=\max\left\{\frac{1-a}{\gamma_{3-}},\alpha_{0}\right\},\\
				\alpha_0&:=\frac{(1-a)\gamma_{3+}}{\gamma_{3+}-\gamma_{3-}}\left[\frac{N_1}{\delta}-\frac{\kappa_1 N_2}{2\delta {\widetilde M}}-\frac{1}{\gamma_{3+}}-\frac{\overline c_2}{\delta}\right].
			\end{aligned}
		\end{equation}
		We have the following results:
		\begin{enumerate}
			\item $w_0$ defined in \eqref{eqn:defn of w0} exists and is given by
			\begin{equation*}
				w_0=G({\widetilde M}),
			\end{equation*}
			where
			\begin{equation}\label{eqn:G}
				G(y):=\int_0^y\frac{\sigma_1^2(z)+\sigma_2^2\left(\frac{\kappa_2}{\kappa_1}z\right)}{2\kappa_1z\mu_1(z)+2\kappa_2z\mu_2\left(\frac{\kappa_2}{\kappa_1}z\right)+\frac{2 \delta}{\kappa_1}z^2-\kappa_1\left(\sigma_1^2(z)+\sigma_2^2\left(\frac{\kappa_2}{\kappa_1}z\right)\right)}dz,
			\end{equation}
			and $u_1$ and $u_2$, defined in \eqref{eqn:defn of u1 & u2}, are explicitly given by 
			\begin{equation}\label{eqn:u1 and u2 w0<u1<u2}
				\begin{aligned}
					u_1&=w_0+\frac{1}{\gamma_{2+}-\gamma_{2-}}\ln\left(\frac{\alpha_{2-}(\gamma_{2-}\alpha_3-1)}{\alpha_{2+}(1-\gamma_{2+}\alpha_3)}\right),\\
					\text{and} \quad u_2&=u_1+\frac{1}{\gamma_{3+}-\gamma_{3-}}\ln\left(\frac{K_{3-}\gamma_{3-}(\gamma_{4-}-\gamma_{3-})}{K_{3+}\gamma_{3+}(\gamma_{3+}-\gamma_{4-})}\right),
				\end{aligned}
			\end{equation}
			where
			\begin{equation}\label{eqn:alpha2pm}
				\alpha_{2+}:= \frac{-\gamma_{2-}-\frac{\kappa_1}{{\widetilde M}}}{\gamma_{2+}(\gamma_{2+}-\gamma_{2-})}, 
				\quad 
				\alpha_{2-} := \frac{\gamma_{2+}+\frac{\kappa_1}{{\widetilde M}}}{\gamma_{2-}(\gamma_{2+}-\gamma_{2-})},
			\end{equation}
			\begin{equation}
				\label{eqn:K3+}
				\alpha_3:=\frac{1}{\gamma_{3+}}+\frac{\overline c_2}{\delta}+\left(1-\frac{\gamma_{3-}}{\gamma_{3+}}\right)\frac{K_{3-}}{1-a},  \quad K_{3+}:=\frac{1}{\gamma_{3+}}\left(1-a-K_{3-}\gamma_{3-}\right),
			\end{equation}
			and  $K_{3-}$ is the unique solution to $\psi(z)=0$ on $\left(\alpha_{LB},\alpha_{UB}\right]$ with
			\begin{align}\label{eqn: alphaLB & UB}
	            \alpha_{UB}:=\frac{(1-a)(\gamma_{3+}-\gamma_{4-})}{\gamma_{3-}(\gamma_{3+}-\gamma_{3-})}.
			\end{align}
			The relation $w_0< u_1\leq u_2$ holds, {with $u_1=u_2$ if and only if $a=\frac{1}{2}$.}
			
			\item The function $g$, defined by
			\begin{equation}\label{eqn:g w0<u1<u2}
				g(x)=
				\begin{cases}
					K_1\int_0^x \exp\left[\int^{w_0}_z\frac{\kappa_1}{G^{-1}(y)}dy\right]dz&\mbox{if $x<w_0$,}\\
					K_1\left[\alpha_{2+}e^{\gamma_{2+}(x-w_0)}+\alpha_{2-}e^{\gamma_{2-}(x-w_0)}\right]&\mbox{if $w_0\leq x<u_1$,}\\
					K_{3+}e^{\gamma_{3+}(x-u_1)}+K_{3-}e^{\gamma_{3-}(x-u_1)}+\frac{(1-a)\overline c_2}{\delta}&\mbox{if $u_1\leq x<u_2$,}\\
					\frac{a}{\gamma_{4-}}e^{\gamma_{4-}(x-u_2)}+\frac{a\overline c_1+(1-a)\overline c_2}{\delta}&\mbox{if $x\geq u_2$,}
				\end{cases}
			\end{equation} 
			where $G^{-1}$ is the inverse of $G$ defined in \eqref{eqn:G}
			and
			\begin{equation}\label{eqn:K1 w0<u1<u2}
				K_1=(1-a)\left[\alpha_{2+}\gamma_{2+}e^{\gamma_{2+}(u_1-w_0)}+\alpha_{2-}\gamma_{2-}e^{\gamma_{2-}(u_1-w_0)}\right]^{-1},
			\end{equation}
			is a classical solution to the HJB equation in \eqref{eqn:hjb bdd} and thus equals the value function $V$ of the bounded-dividend optimization problem in \eqref{eqn:V bdd}. Moreover, $g$ is strictly concave. 
			
			\item The optimal reinsurance and dividend strategies $(\pi_1^*, \pi_2^*, c_1^*, c_2^*)$ {are given in feedback form by $\pi_i^*(t)=\Pi_i(X_1^*(t)+X_2^*(t))$ and $c_i^*(t)=D_i(X_1^*(t)+X_2^*(t))$, $i=1,2,$ where}
			\begin{equation*}
				{(\Pi_1, \Pi_2, D_1, D_2)} (x)=
				\begin{cases}
					\left({\min\{M_1,G^{-1}(x)\},\min\{M_2,\frac{\kappa_2}{\kappa_1}G^{-1}(x)\}},0,0\right)&\mbox{if $x < w_0$,}\\
					\left(M_1,{M_2},0,0\right)&\mbox{if $w_0\leq x<u_1$,}\\
					\left(M_1,{M_2},0,\overline{c}_2\right)&\mbox{if $u_1\leq x<u_2$,}\\
					\left(M_1,{M_2},\overline{c}_1,\overline{c}_2\right)&\mbox{if $x\geq u_2$.}
				\end{cases}
			\end{equation*}
		\end{enumerate}		
	\end{theorem}

        We first discuss the optimal dividend strategies $(c_1^*, c_2^*)$. When the aggregate reserve level is low (i.e., $x<u_1$), the lines do not distribute dividends. Recall that we assume greater importance assigned to Line 2 (i.e., $a\leq \frac{1}{2}$). As soon as the first threshold $u_1$ is crossed, Line 2 pays dividends at the maximum rate and Line 1 pays dividends only after the second threshold $u_2$ is crossed.

        For the optimal reinsurance strategy $(\pi_1^*, \pi_2^*)$, {both lines purchase excess-of-loss reinsurance when the aggregate reserve level is low, and their retention levels increase with the aggregate reserve level. If $\frac{M_1}{\kappa_1}<\frac{M_2}{\kappa_2}$, Line 1 retains all of its risk once the aggregate reserve level exceeds $G(M_1)$, while the retention level of Line 2 continues to increase until it reaches $M_2$ at $x=w_0$. The case $\frac{M_1}{\kappa_1}>\frac{M_2}{\kappa_2}$ is similar, with Line 2 retaining all of its risk first, at $x=G(\frac{\kappa_1}{\kappa_2}M_2)$. Hence, reinsurance is phased out line by line, and both lines retain all of their risk when the aggregate reserve level exceeds $w_0$.}
        
		Under condition (i), {both claim-size distributions have bounded support, which implies that both lines have bounded claim sizes}. A sufficient condition for conditions (ii) and (iii) is $\overline c_2>N_1-\frac{\kappa_1 N_2}{2{\widetilde M}}+\frac{\delta {\widetilde M}}{\kappa_1}$. This suggests that the maximum dividend rate of Line 1 can be relatively low while still satisfying the total maximum dividend rate condition. We also remark that the unique root of $\psi$ ensures that $g$ is differentiable at $x=u_2$.

        The next theorem discusses the case when conditions (i) and (ii) of Theorem \ref{thm:bdd w0<u1<u2} are still satisfied, while condition (iii) is not.

\begin{theorem}\label{thm:bdd u1<w0<u2}
		Suppose (i) ${\widetilde M}<\infty$, (ii) $\overline c_1+\overline c_2\geq N_1-\frac{\kappa_1 N_2}{2{\widetilde M}}+\frac{\delta {\widetilde M}}{\kappa_1}$ and (iii) {$\gamma_{3-}+\frac{\kappa_1}{\widetilde M}<0$ and $K_{2+}\leq 1-a$, where $K_{2+}$ is defined in \eqref{eqn:K2+ u1<w0<u2} below}. We then have the following results:
		\begin{enumerate}
			\item $w_0$ defined in \eqref{eqn:defn of w0} exists and is given by $H(w_0)={\widetilde M}$, where $H$ satisfies the following differential equation {on $[u_1,w_0]$ with matching conditions $H(u_1)=G^{-1}(u_1)$ and $H(w_0)=\widetilde M$:}
            \begin{equation}\label{eqn:H}
                \frac{\dd}{\dd x}H(x)=\frac{2\kappa_1H(x)\mu_1(H(x))+2\kappa_2H(x)\mu_2\left(\frac{\kappa_2}{\kappa_1}H(x)\right)+\frac{2 \delta}{\kappa_1}\left[H(x)\right]^2-2\overline c_2 H(x)}{\sigma_1^2(H(x))+\sigma_2^2\left(\frac{\kappa_2}{\kappa_1}H(x)\right)}-\kappa_1.
            \end{equation} {Moreover, $u_1$ defined in \eqref{eqn:defn of u1 & u2} is characterized as the unique solution to
            \begin{equation}\label{eqn:u1 u1<w0<u2}
                \int_{u_1}^{w_0}\frac{\kappa_1}{H(y)}\,dy=\ln\left(\frac{1-a}{K_{2+}}\right)
            \end{equation}
            in $(0,w_0]$, and $u_2$ defined in \eqref{eqn:defn of u1 & u2} is explicitly given by}
            \begin{equation}\label{eqn:u1 and u2 u1<w0<u2}
				u_2=w_0+\frac{1}{\gamma_{3+}-\gamma_{3-}}\ln\left(\frac{(\gamma_{4-}-\gamma_{3-})\left(-\frac{\kappa_1}{{\widetilde M}}-\gamma_{3+}\right)}{(\gamma_{3+}-\gamma_{4-})\left(\gamma_{3-}+\frac{\kappa_1}{{\widetilde M}}\right)}\right),
			\end{equation}
			{where
            \begin{equation}\label{eqn:K2+ u1<w0<u2}
                K_{2+}:=\frac{a}{\gamma_{3+}-\gamma_{3-}}\left[(\gamma_{4-}-\gamma_{3-})e^{\gamma_{3+}(w_0-u_2)}+(\gamma_{3+}-\gamma_{4-})e^{\gamma_{3-}(w_0-u_2)}\right].
            \end{equation}
            Since \eqref{eqn:H} does not depend on $x$ explicitly, \eqref{eqn:u1 u1<w0<u2} determines $w_0-u_1$ and $H(u_1)$, and then $u_1=G(H(u_1))$.} Moreover, the relation $u_1\leq w_0\leq u_2$ holds, {with $u_1=w_0$ if and only if $K_{2+}=1-a$}.
			
			\item The function $g$, defined by
			\begin{equation}\label{eqn:g u1<w0<u2}
				g(x)=
				\begin{cases}
					(1-a)\int_0^x \exp\left[\int^{u_1}_z\frac{\kappa_1}{G^{-1}(y)}dy\right]dz&\mbox{if $x<u_1$,}\\
					(1-a)\left[\int_{u_1}^x \exp\left[\int^{u_1}_z\frac{\kappa_1}{H(y)}dy\right]dz+\int_0^{u_1} \exp\left[\int^{u_1}_z\frac{\kappa_1}{G^{-1}(y)}dy\right]dz\right]&\mbox{if $u_1\leq x<w_0$,}\\
					{\frac{a}{\gamma_{3+}-\gamma_{3-}}\left[\frac{\gamma_{4-}-\gamma_{3-}}{\gamma_{3+}}e^{\gamma_{3+}(x-u_2)}+\frac{\gamma_{3+}-\gamma_{4-}}{\gamma_{3-}}e^{\gamma_{3-}(x-u_2)}\right]+\frac{(1-a)\overline c_2}{\delta}}&\mbox{if $w_0\leq x<u_2$,}\\
					\frac{a}{\gamma_{4-}}e^{\gamma_{4-}(x-u_2)}+\frac{a\overline c_1+(1-a)\overline c_2}{\delta}&\mbox{if $x\geq u_2$,}
				\end{cases}
			\end{equation} 
			where $G^{-1}$ is the inverse of $G$ defined in \eqref{eqn:G}, is a classical solution to the HJB equation in \eqref{eqn:hjb bdd} and thus equals the value function $V$ of the bounded-dividend optimization problem in \eqref{eqn:V bdd}. In addition, $g$ is strictly concave. 
			
			\item The optimal reinsurance and dividend strategies $(\pi_1^*, \pi_2^*, c_1^*, c_2^*)$ {are given in feedback form by $\pi_i^*(t)=\Pi_i(X_1^*(t)+X_2^*(t))$ and $c_i^*(t)=D_i(X_1^*(t)+X_2^*(t))$, $i=1,2,$ where}
			\begin{equation*}
				{(\Pi_1, \Pi_2, D_1, D_2)} (x)=
				\begin{cases}
					\left({\min\{M_1,G^{-1}(x)\},\min\{M_2,\frac{\kappa_2}{\kappa_1}G^{-1}(x)\}},0,0\right)&\mbox{if $x < u_1$,}\\
					\left({\min\{M_1,H(x)\},\min\{M_2,\frac{\kappa_2}{\kappa_1}H(x)\}},0,{\overline c_2}\right)&\mbox{if $u_1\leq x<w_0$,}\\
					\left(M_1,{M_2},0,\overline{c}_2\right)&\mbox{if $w_0\leq x<u_2$,}\\
					\left(M_1,{M_2},\overline{c}_1,\overline{c}_2\right)&\mbox{if $x\geq u_2$.}
				\end{cases}
			\end{equation*}
		\end{enumerate}		
	\end{theorem}

		By comparing Theorems \ref{thm:bdd w0<u1<u2} and \ref{thm:bdd u1<w0<u2}, the two theorems exhibit notable similarities in the optimal reinsurance and dividend strategies in both cases. The main distinction is the form of the retention level when the aggregate reserve level is between $w_0$ and $u_1$. Moreover, Theorem \ref{thm:bdd w0<u1<u2} provides an explicit expression for $w_0$, whereas in Theorem \ref{thm:bdd u1<w0<u2}, $w_0$ is presented implicitly. 
		
		In this scenario, a necessary condition is $0<\overline c_2<N_1-\frac{\kappa_1 N_2}{2{\widetilde M}}+\frac{\delta {\widetilde M}}{\kappa_1}$, which means that the maximum dividend rate of Line 1, $\overline c_1$, must be sufficiently large to make the total maximum dividend rate exceed the required threshold to achieve a minimum total of $N_1-\frac{\kappa_1 N_2}{2{\widetilde M}}+\frac{\delta {\widetilde M}}{\kappa_1}$. This suggests that Line 1 may have a larger maximum dividend rate than Line 2, which could be strategically advantageous in improving overall financial performance by setting a higher maximum dividend rate.

        {Condition (iii) of Theorem \ref{thm:bdd u1<w0<u2} is explicit in terms of the model parameters, since $w_0-u_2$ is given in \eqref{eqn:u1 and u2 u1<w0<u2}. It has the following interpretation. The constant $K_{2+}$ is the marginal value of the aggregate reserve level $g'(w_0)$ obtained from the regions $\{x\geq w_0\}$, in which both lines retain all of their risk. If $K_{2+}<1-a$, then Line 2 already pays dividends at the aggregate reserve level $w_0$, which yields $u_1<w_0$; in the boundary case $K_{2+}=1-a$, Line 2 starts paying dividends exactly when both lines retain all of their risk (i.e., $u_1=w_0$). Moreover, under conditions (i) and (ii), exactly one of Theorems \ref{thm:bdd w0<u1<u2} and \ref{thm:bdd u1<w0<u2} applies (see Section \ref{sec:proof}), and condition (iii) of Theorem \ref{thm:bdd w0<u1<u2} is precisely the negation of condition (iii) of Theorem \ref{thm:bdd u1<w0<u2}. Hence, condition (iii) of Theorem \ref{thm:bdd w0<u1<u2} can also be verified explicitly.}

        It is worth noting that \cite{asmussen2000} show that, in the univariate case, the reinsurance threshold cannot exceed the dividend threshold. In the present bivariate setting, however, this configuration can occur.

        The next theorem covers the remaining case, in which at least one of conditions (i) and (ii) in Theorems \ref{thm:bdd w0<u1<u2} and \ref{thm:bdd u1<w0<u2} fails.

	\begin{theorem}\label{thm: bdd u1<u2<w0 or M infty}
		Suppose that either (i) ${\widetilde M}=\infty$ or (ii) ${\widetilde M}<\infty$ and $\overline c_1+\overline c_2< N_1-\frac{\kappa_1 N_2}{2{\widetilde M}}+\frac{\delta {\widetilde M}}{\kappa_1}$ holds. We have the following results:
		\begin{enumerate}
			\item $w_0$ defined in \eqref{eqn:defn of w0} is infinite ($w_0=\infty$), $u_1$ is characterized as the unique solution to
			\begin{equation}\label{eqn:formula for u1 case 3}
					\int_{u_2}^{u_1}\frac{\kappa_1}{H(y)}\,dy
	=
	\ln\left(\frac{a}{1-a}\right)
			\end{equation}
			in $(0,u_2]$, and $u_2$ defined in \eqref{eqn:defn of u1 & u2} satisfies 
			\begin{equation}
				H(u_2)=M_0,
			\end{equation}
            where $H$ satisfies \eqref{eqn:H} on $[u_1,u_2]$ with matching conditions $H(u_1)=G^{-1}(u_1)$ and $H(u_2)=M_0$,  and $M_0$ is the unique solution of
			\begin{equation}
				-\frac{\kappa_1}{y}=\overline\gamma_{4-}(y)
			\end{equation}
			in $(0,{\widetilde M})$. Moreover, the relation $u_1\leq u_2<w_0$ holds.
			
			\item The function $g$, defined by
			\begin{equation}
				g(x)=
				\begin{cases}
					(1-a)\int_0^x \exp\left[\int^{u_1}_z\frac{\kappa_1}{G^{-1}(y)}dy\right]dz&\mbox{if $x<u_1$,}\\
					(1-a)\left[\int_{u_1}^x \exp\left[\int^{u_1}_z\frac{\kappa_1}{H(y)}dy\right]dz+\int_0^{u_1} \exp\left[\int^{u_1}_z\frac{\kappa_1}{G^{-1}(y)}dy\right]dz\right]&\mbox{if $u_1\leq x<u_2$,}\\
					\frac{a}{\overline\gamma_{4-}(M_0)}e^{\overline\gamma_{4-}(M_0)(x-u_2)}+\frac{a\overline c_1+(1-a)\overline c_2}{\delta}&\mbox{if $x\geq u_2$,}
				\end{cases}
			\end{equation} 
			where $G^{-1}$ is the inverse of $G$ defined in \eqref{eqn:G}, is a classical solution to the HJB equation in \eqref{eqn:hjb bdd} and thus equals the value function $V$ of the bounded-dividend optimization problem in \eqref{eqn:V bdd}. In addition, $g$ is strictly concave. 
			
			\item The optimal reinsurance and dividend strategies $(\pi_1^*, \pi_2^*, c_1^*, c_2^*)$ {are given in feedback form by $\pi_i^*(t)=\Pi_i(X_1^*(t)+X_2^*(t))$ and $c_i^*(t)=D_i(X_1^*(t)+X_2^*(t))$, $i=1,2,$ where}
			\begin{equation*}
				{(\Pi_1, \Pi_2, D_1, D_2)} (x)=
				\begin{cases}
					\left({\min\{M_1,G^{-1}(x)\},\min\{M_2,\frac{\kappa_2}{\kappa_1}G^{-1}(x)\}},0,0\right)&\mbox{if $x < u_1$,}\\
					\left({\min\{M_1,H(x)\},\min\{M_2,\frac{\kappa_2}{\kappa_1}H(x)\}},0,\overline c_2\right)&\mbox{if $u_1\leq x<u_2$,}\\
					\left({\min\{M_1,M_0\},\min\{M_2,\frac{\kappa_2}{\kappa_1}M_0\}},\overline{c}_1,\overline{c}_2\right)&\mbox{if $x\geq u_2$.}
				\end{cases}
			\end{equation*}
		\end{enumerate}		
	\end{theorem}

		Under condition (i), at least one claim-size distribution has unbounded support, so that at least one line must always cede a portion of its risk to the reinsurer, regardless of the aggregate reserve level. This implies that $w_0$ defined in \eqref{eqn:defn of w0} does not exist. If only one of the supports is unbounded, the other line may still retain all of its risk when the aggregate reserve level is sufficiently large. Under condition (ii), the total maximum dividend rate is below the level
$N_1-\frac{\kappa_1 N_2}{2{\widetilde M}}+\frac{\delta {\widetilde M}}{\kappa_1}$.
This leaves more reserves inside the company, allowing the insurer to purchase more reinsurance.
        Similar to the results in \cite{asmussen2000}, a finite reinsurance threshold strictly exceeding all dividend thresholds does not arise here; instead, in this case $w_0=\infty$.

    We now discuss the optimal capital-injection strategy. Since the value function obtained above has the form $V(x_1,x_2)=g(x_1+x_2)$,
 any instantaneous transfer of one unit of reserve from one line to the other leaves the value unchanged. The role of the capital-transfer processes $L_1,L_2$ is thus only to keep the reserve pair inside the nonnegative quadrant until the aggregate reserve reaches zero, and
to implement the dividend policy when the line that should pay dividends has insufficient reserves. The transfer policies stated below are one convenient implementation; they are not unique.
    
    We partition the domain of the reserve level pair $(x_1,x_2)\in \mathbb R_+^2$ into seven regions (see Figure \ref{fig:capitaltransfer}). Define the constants $d_i$, $i=0,1,2$, corresponding to each of the three cases in Theorems \ref{thm:bdd w0<u1<u2}, \ref{thm:bdd u1<w0<u2}, and \ref{thm: bdd u1<u2<w0 or M infty}, by 
	\begin{equation}
		(d_0,d_1,d_2)=
		\begin{cases}
			(w_0,u_1,u_2)&\mbox{if $w_0<u_1\leq u_2$,}\\
			(u_1,w_0,u_2)&\mbox{if $u_1\leq w_0\leq u_2$,}\\
			(u_1,u_1,u_2)&\mbox{if $u_1\leq u_2<w_0$}.
		\end{cases}
	\end{equation}
	The seven regions $A_i$, $i=1,2,\cdots, 7$, are defined as follows (see Figure \ref{fig:capitaltransfer}):
	\begin{itemize}
		\item $A_1=\{(x_1,x_2):x_1\geq 0, x_2>d_2\},$
		\item $A_2=\{(x_1,x_2):x_1>0, x_2\in [0,d_2], x_1+x_2> d_2\},$
		\item $A_3=\{(x_1,x_2):x_1\geq0, x_2\in (d_1,d_2], x_1+x_2\leq d_2\},$
		\item $A_4=\{(x_1,x_2):x_1> 0, x_2\in [0,d_1], x_1+x_2\in (d_1,d_2]\},$
		\item $A_5=\{(x_1,x_2):x_1\geq 0, x_2\in (d_0,d_1], x_1+x_2\leq d_1\},$
		\item $A_6=\{(x_1,x_2):x_1> 0, x_2\in [0,d_0], x_1+x_2\in (d_0,d_1]\},$
		\item $A_7=\{(x_1,x_2):x_1\geq 0, x_2\geq 0, x_1+x_2\leq d_0\}.$
	\end{itemize}
	
	\begin{figure}[htb]
		\centering
		\begin{tikzpicture}[scale=1] % Default scale
			% Axes
			\draw[->] (-1,0) -- (7,0) node[right] {$x_1$};
			\draw[->] (0,-1) -- (0,7) node[above] {$x_2$};
			
			% Lines
			\draw (0,6) node[left] {$(0,d_2)$} -- (6,0) node[below] {$(d_2,0)$}; 
			\draw (0,4) node[left] {$(0,d_1)$} -- (4,0) node[below] {$(d_1,0)$};
			\draw (0,2) node[left] {$(0,d_0)$} -- (2,0) node[below] {$(d_0,0)$};
			\draw (0,6) -- (7,6) node[right] {$x_2 = d_2$}; % horizontal line
			\draw (0,4) -- (2,4) node[right] {}; % horizontal line
			\draw (0,2) -- (2,2) node[right] {}; % horizontal line
			
			% Points
			\filldraw[black] (0,4) circle (2pt); 
			\filldraw[black] (0,6) circle (2pt); 
			\filldraw[black] (6,0) circle (2pt);
			\filldraw[black] (0,2) circle (2pt);
			\filldraw[black] (2,0) circle (2pt);
			\filldraw[black] (4,0) circle (2pt);
			
			% Region Labels
			\node at (0.5, 0.9) {$A_7$};
			\node at (2.1, 0.9) {$A_6$};
			\node at (3, 2) {$A_4$};
			\node at (0.5, 2.9) {$A_5$};
			\node at (4.1, 2.9) {$A_2$};
			\node at (0.5, 4.9) {$A_3$};
			\node at (0.5, 6.5) {$A_1$};

			\node[rotate=-45] at (1.1, 1.1) {\footnotesize$x_1+x_2 = d_0$};
			\node[rotate=-45] at (3.1, 1.1) {\footnotesize$x_1+x_2 = d_1$};
			\node[rotate=-45] at (5.1, 1.1) {\footnotesize$x_1+x_2 = d_2$};
			\node at (1, 4.2) {\footnotesize$x_2 = d_1$};
			\node at (1, 2.2) {\footnotesize$x_2 = d_0$};
		\end{tikzpicture}
		\caption{Regions for Capital-Injection Decisions.}
		\label{fig:capitaltransfer}
	\end{figure}
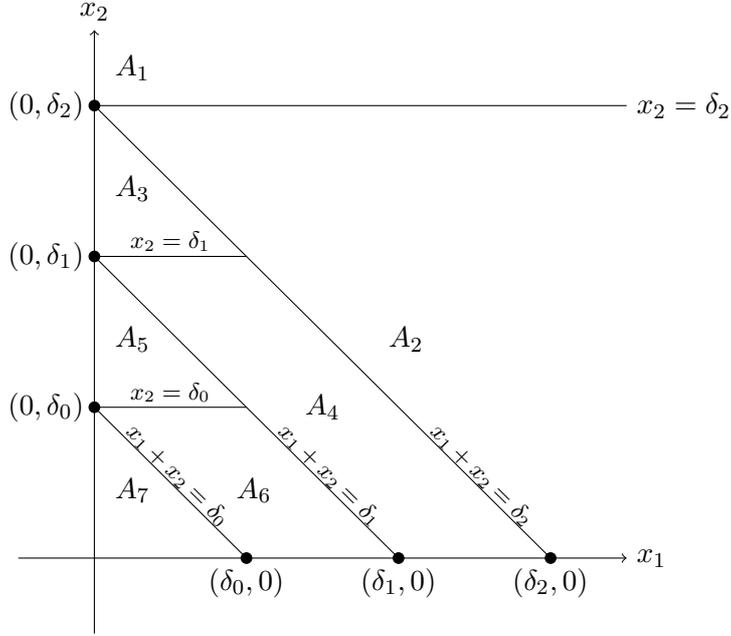
	
	\begin{theorem}
		\label{thm:op_L}
		The optimal capital-injection strategies are given by the following cases:
		\begin{enumerate}

			\item If $x\in A_1$ and Line 1 reaches zero, Line 2 transfers an amount of $x_2-d_2$ to Line 1, and the resulting reserve pair $(x_2-d_2,d_2)$ belongs to region $A_{2}$. If Line 1 does not reach zero, the reserve pair remains in $A_1$ until it moves to region $A_2$ or $A_3$. 
			
			\item If $x\in A_2$ and Line 2 reaches zero, Line 1 transfers an amount of $x_1-d_2$ to Line 2, resulting in the reserve pair $(d_2,x_1-d_2)$. This pair remains in $A_2$ if $x_1\leq 2d_2$. Otherwise, it moves to $A_1$. If Line 2 does not reach zero, the reserve pair remains in $A_2$ until it moves to region $A_1$, $A_3$, or $A_4$.
            
			\item If $x\in A_3$ and Line 1 reaches zero, Line 2 transfers an amount of $x_2-d_1$ to Line 1, and the resulting reserve pair $(x_2-d_1,d_1)$ belongs to region $A_4$. If Line 1 does not reach zero, the reserve pair remains in $A_3$ until it moves to region $A_2$, $A_4$, or $A_5$.
			
			\item If $x\in A_4$ and Line 2 reaches zero, Line 1 transfers an amount of $x_1-d_1$ to Line 2, resulting in the reserve pair $(d_1,x_1-d_1)$. This pair remains in $A_4$ if $x_1\leq 2d_1$. Otherwise, it moves to $A_3$. If Line 2 does not reach zero, the reserve pair remains in $A_4$ until it moves to region $A_2$, $A_3$, $A_5$, or $A_6$.
			
			\item If $x\in A_5$ and Line 1 reaches zero, Line 2 transfers an amount of $x_2-d_0$ to Line 1, and the resulting reserve pair $(x_2-d_0,d_0)$ belongs to region $A_6$.  If Line 1 does not reach zero, the reserve pair remains in $A_5$ until it moves to region $A_4$, $A_6$, or $A_7$.
			
			\item If $x\in A_6$ and Line 2 reaches zero, Line 1 transfers an amount of $x_1-d_0$ to Line 2, resulting in the reserve pair $(d_0,x_1-d_0)$. This pair remains in $A_6$ if $x_1\leq 2d_0$. Otherwise, it moves to $A_5.$ If Line 2 does not reach zero, the reserve pair remains in $A_6$ until it moves to region $A_4$, $A_5$, or $A_7$.
			
			\item If $x\in A_7$, the reserve pair remains in $A_7$ until it moves to region
$A_5$ or $A_6$, depending on where it crosses the boundary $x_1+x_2=d_0$. When the reserve pair reaches either axis, the insurer applies capital transfers from the other line to keep both reserve levels nonnegative. The problem ends when the reserve pair reaches $(0,0)$.
		\end{enumerate}	
	\end{theorem}

    {Since $V(x_1,x_2)=g(x_1+x_2)$ and capital transfers leave the aggregate reserve level unchanged, any capital-injection strategy that keeps both reserve levels nonnegative as long as the aggregate reserve level is positive is optimal by Theorem \ref{thm:verification}. The strategy in Theorem \ref{thm:op_L} is one such strategy: each lump-sum transfer is positive and does not exceed the reserve level of the transferring line, and at the axes the transfers act as reflections, which also fund the dividend payments of a line whose reserve level is zero. Moreover, $L_1$ and $L_2$ do not increase at the same time, since both reserve levels are zero only when the aggregate reserve level is zero.}
	
	\section{Unbounded Dividend Rates}\label{sec:ubdd}
	
	In this section, we consider the case in which there are no restrictions on the dividend rates of both lines. {Since every strategy with bounded dividend rates is also admissible in this case, the value function with unbounded dividend rates is at least as large as the value function with bounded dividend rates. We still treat the two cases separately, since bounded dividend rates are a practically relevant restriction and lead to a different structure of the optimal strategies. In the bounded case, the optimal dividend rule is a threshold strategy with two dividend threshold levels rather than a barrier strategy, and the reinsurance threshold level can exceed one of the dividend threshold levels.} The corresponding HJB equation is given by:	
	\begin{equation}\label{eqn:hjb unbdd}
			\sup\left\{\sup_{\pi_i \in[0,M_i]}\Lc^{\pi}(V), \quad \frac{\partial V}{\partial x_1}-\frac{\partial V}{\partial x_2}, \quad \frac{\partial V}{\partial x_2}-\frac{\partial V}{\partial x_1},a-\frac{\partial V}{\partial x_1},1-a-\frac{\partial V}{\partial x_2}\right\} = 0,
		\end{equation}
		with boundary condition $V(0,0)=0$ and
		\begin{equation*}
			\mathcal L^{\pi}(\phi):=\sum_{i=1}^2\left[\kappa_i\mu_i(\pi_i)\frac{\partial \phi}{\partial x_i}+\frac{1}{2}\sigma_i^2(\pi_i)\frac{\partial^2 \phi}{\partial x_i^2}\right]
-\delta \phi,\quad \phi\in C^{2,2}.
		\end{equation*}
        The first, fourth, and fifth terms in the outer supremum can be handled using arguments similar to those in \citet[Proposition 2.54]{schmidli2007book}. The second and third terms arise from the capital-injection strategies modeled as singular controls.
        
    {By Theorem \ref{thm:verification}(ii), a classical solution of the form $V(x_1,x_2)=g(x_1+x_2)$}, if it exists, is the value function that solves the associated optimization problem.

    We now present the main results for the unbounded-dividend-rate case. Two cases arise depending on whether ${\widetilde M}$ is finite (Theorem \ref{thm:ubdd M1 finite}) or infinite (Theorem \ref{thm: unbdd M1 infty}). As in the previous section, each case has its corresponding analytical form of the value function and a reinsurance strategy $(\pi_1^*,\pi_2^*)$. Since the dividend payout strategies $C_1$ and $C_2$ are also modeled as singular controls, similar to the capital-injection strategies, we obtain a uniform optimal strategy $(C_1^*,C_2^*,L_1^*,L_2^*)$ in Theorem \ref{thm: ubdd uniform strat}.
	
	\begin{theorem}\label{thm:ubdd M1 finite}
		Suppose ${\widetilde M}<\infty$. We have the following results:
		\begin{enumerate}
			\item $w_0$ defined in \eqref{eqn:defn of w0} exists and is given by
			\begin{equation*}
				w_0=G({\widetilde M}),
			\end{equation*}
			where $G$ is defined in \eqref{eqn:G}, and $u_1$ defined in \eqref{eqn:defn of u1 & u2} is explicitly given by 
			\begin{equation}\label{eqn:u1 unbdd M1 finite}
				u_1=w_0+\frac{1}{\gamma_{2+}-\gamma_{2-}}\ln\left(\frac{\gamma_{2-}(\kappa_1+\gamma_{2+}{\widetilde M})}{\gamma_{2+}(\kappa_1+\gamma_{2-}{\widetilde M})}\right),
			\end{equation}
			where $\gamma_{2\pm}:=\overline\gamma_{2\pm}({\widetilde M})$, with
$\overline\gamma_{2\pm}$ defined in \eqref{eqn:notations}. The relation $w_0< u_1$ holds. {Moreover, $u_1=G(\infty):=\lim_{y\to\infty} G(y)$.}
			
			\item The function $g$, defined by
			\begin{equation}\label{g:unbdd M1 finite}
				g(x)=
				\begin{cases}
					K_1\int_0^x \exp\left[\int^{w_0}_z\frac{\kappa_1}{G^{-1}(y)}dy\right]dz&\mbox{if $x<w_0$,}\\
					\frac{1-a}{\gamma_{2+}\gamma_{2-}(\gamma_{2+}-\gamma_{2-})}\left[\gamma_{2+}^2e^{\gamma_{2-}(x-u_1)}-\gamma_{2-}^2e^{\gamma_{2+}(x-u_1)}\right]&\mbox{if $w_0\leq x<u_1$,}\\
					(1-a)\left[x-u_1+\frac{N_1}{\delta}\right]&\mbox{if $x\geq u_1$,}
				\end{cases}
			\end{equation} 
			where $G^{-1}$ is the inverse of $G$ defined in \eqref{eqn:G}, $N_1$ equals $\overline N_1({\widetilde M})$ defined in \eqref{eqn:notations},
			and
			\begin{equation*}
				K_1=\frac{1-a}{\gamma_{2+}-\gamma_{2-}}\left[\gamma_{2+}e^{\gamma_{2-}(w_0-u_1)}-\gamma_{2-}e^{\gamma_{2+}(w_0-u_1)}\right],
			\end{equation*}
			is a classical solution to the HJB equation in \eqref{eqn:hjb unbdd} and thus equals the value function $V$ of the optimization problem in \eqref{eqn:V}. Moreover, $g$ is concave. 
			
			\item The optimal reinsurance strategy $(\pi_1^*, \pi_2^*)$ is given {in feedback form by $\pi_i^*(t)=\Pi_i(X_1^*(t)+X_2^*(t))$, $i=1,2,$ where}
			\begin{equation*}
				({\Pi_1,\Pi_2}) (x)=
				\begin{cases}
					\left({\min\{M_1,G^{-1}(x)\},\min\{M_2,\frac{\kappa_2}{\kappa_1}G^{-1}(x)\}}\right)&\mbox{if $x < w_0$,}\\
					\left({M_1,M_2}\right)&\mbox{if $x\geq w_0$.}
				\end{cases}
			\end{equation*}
		\end{enumerate}		
	\end{theorem}
		Similar to {Theorem \ref{thm:bdd w0<u1<u2}, reinsurance is phased out line by line in Theorem \ref{thm:ubdd M1 finite}, and both lines retain all of their risk when the aggregate reserve level exceeds $w_0$. Moreover, the barrier $u_1=G(\infty)$ has the same form as in Theorem \ref{thm: unbdd M1 infty} below, so that $u_1$ can be computed directly from $G$ in both cases, regardless of whether the claim-size distributions have bounded support.}
	
	\begin{theorem}\label{thm: unbdd M1 infty}
	Suppose ${\widetilde M}=\infty$. Define $N_1^\infty:=\lim_{y\to\infty}\overline N_1(y)		=\kappa_1\widetilde\mu_1+\kappa_2\widetilde\mu_2$.
	We have the following results:
	\begin{enumerate}
		\item $w_0$ defined in \eqref{eqn:defn of w0} is infinite, i.e., $w_0=\infty$, and $u_1$ defined in \eqref{eqn:defn of u1 & u2} is explicitly given by 
		\begin{equation}
			u_1=G(\infty):=\lim _{y\to\infty}G(y),
		\end{equation}
		where $G$ is defined by \eqref{eqn:G}. The relation $u_1<w_0$ holds.
		
		\item The function $g$, defined by
		\begin{equation}\label{eqn:g unbdd M infty}
			g(x)=
			\begin{cases}
				(1-a)\displaystyle\int_0^x \exp\left[\int^{u_1}_z\frac{\kappa_1}{G^{-1}(y)}dy\right]dz,
				&\mbox{if $x<u_1$,}\\[2ex]
				(1-a)\left[x-u_1+\dfrac{N_1^\infty}{\delta}\right],
				&\mbox{if $x\geq u_1$,}
			\end{cases}
		\end{equation} 
		where $G^{-1}$ is the inverse of $G$ defined in \eqref{eqn:G}, is a classical solution to the HJB equation in \eqref{eqn:hjb unbdd} and thus equals the value function $V$ of the optimization problem in \eqref{eqn:V}. Moreover, $g$ is concave.
		
		\item The optimal reinsurance strategy on the continuation region is given {in feedback form by $\pi_i^*(t)=\Pi_i(X_1^*(t)+X_2^*(t))$, $i=1,2,$ where}
			\begin{equation*}
				({\Pi_1,\Pi_2})  (x)=
			\left({\min\left\{M_1,G^{-1}(x)\right\},\min\left\{M_2,\frac{\kappa_2}{\kappa_1}G^{-1}(x)\right\}}\right),
			\quad x < u_1.
		\end{equation*}
		For $x\geq u_1$, the surplus above the barrier is immediately removed by the singular dividend and capital-transfer strategy, so the reinsurance control in this region is irrelevant.
	\end{enumerate}		
\end{theorem}
		Theorem \ref{thm: unbdd M1 infty}, in conjunction with Theorem \ref{thm: ubdd uniform strat} below, implies that the aggregate reserve level must remain below $x=u_1$ since the aggregate reserves above $u_1$ are transferred between lines and are immediately paid as dividends.

    For the optimal dividend payout and capital-injection strategy, we partition the domain of the reserve level pair $(x_1,x_2)\in\mathbb R_+^2$ into 3 regions (see Figure \ref{fig:capitaltransfer unbdd}). The three regions $A_i$, $i=1,2,3$, are defined as follows:
	\begin{itemize}
		\item $A_1=\{(x_1,x_2):x_2\geq 0, x_1>u_1\},$
		\item $A_2=\{(x_1,x_2):x_1\in[0,u_1], x_2>0, x_1+x_2>u_1\},$
		\item $A_3=\{(x_1,x_2):x_1\geq 0, x_2\geq 0, x_1+x_2\leq u_1\}.$
	\end{itemize}
	
	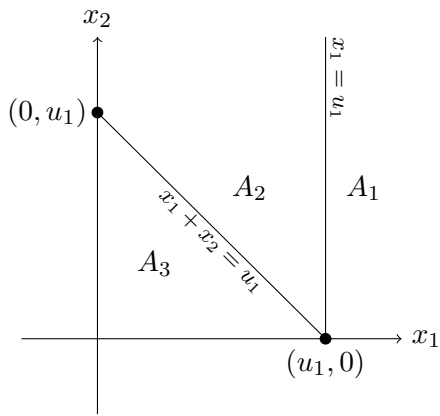
\begin{figure}[htb]
		\centering
		\begin{tikzpicture}[scale=1] % Default scale
			% Axes
			\draw[->] (-1,0) -- (4,0) node[right] {$x_1$};
			\draw[->] (0,-1) -- (0,4) node[above] {$x_2$};
			
			% Lines
			\draw (0,3) node[left] {$(0,u_1)$} -- (3,0) node[below] {$(u_1,0)$};
			\draw (3,0) -- (3,4) node[right] {}; % horizontal line
			
			% Points
			\filldraw[black] (0,3) circle (2pt); 
			\filldraw[black] (3,0) circle (2pt);
			
			% Region Labels
			\node at (3.5, 2) {$A_1$};
			\node at (2, 2) {$A_2$};
			\node at (0.75, 1) {$A_3$};

			\node[rotate=-45] at (1.5, 1.3) {\footnotesize$x_1+x_2 = u_1$};
			\node[rotate=-90] at (3.15, 3.45) {\footnotesize$x_1 = u_1$};
		\end{tikzpicture}
		\caption{Regions for Dividend Payout and Capital-Injection Decisions.}
		\label{fig:capitaltransfer unbdd}
	\end{figure}
	
	\begin{theorem}\label{thm: ubdd uniform strat}
		The optimal dividend-payout and capital-injection strategies are given by the following cases:
		\begin{enumerate}
			\item If $x\in A_1$, Line 1 transfers an amount $x_1-u_1$ to Line 2, and the reserve pair moves to region $A_2$.
			\item If $x\in A_2$, Line 2 pays $x_1+x_2-u_1$ directly as dividends, and the reserve pair moves to region $A_3$.
			\item If $x\in A_3$, no dividends are paid and the reserve pair remains in $A_3$ until it moves to region $A_2$. When the reserve pair reaches either axis, the insurer applies capital transfers from the other line to keep both reserve levels nonnegative. The problem ends when the reserve pair reaches $(0,0)$.
		\end{enumerate}
	\end{theorem}

    {The same argument as for Theorem \ref{thm:op_L} shows that the strategy in Theorem \ref{thm: ubdd uniform strat} is optimal: the aggregate reserve level follows the barrier strategy with barrier $u_1$ and if the reserve level of Line 2 is zero at the barrier, Line 1 transfers the amount paid as dividends to Line 2.}

    The optimal capital-injection strategies in the bounded-dividend-rate case and the unbounded-dividend-rate case operate according to the same principle: when one business line's reserve level reaches zero, capital is transferred from the other line in order to maintain nonnegative reserve levels, provided that the transfer does not result in insolvency. However, the two cases require different capital-injection strategies precisely because of the nature of the dividend rates. In the unbounded-dividend-rate case, dividends can be paid immediately. On the other hand, in the bounded-dividend-rate case, dividends can only be paid up to a maximum rate, so excess reserves cannot, in general, be paid out as dividends immediately. Hence, the insurer must retain excess reserves and allow them to be depleted over time via dividend payouts.

    This distinction also affects the allocation of the reserves between the two business lines, particularly the dividend payout and capital-injection strategies. In the unbounded-dividend-rate case, the business line with the higher weight in the objective function is the one that distributes dividends. Thus, when a capital transfer is needed, capital is transferred from the
    lower-weighted business line to the higher-weighted business line. In the bounded-dividend-rate case, both business lines can distribute dividends. The higher-weighted business line has a lower reserve threshold at which dividends are paid at a maximum rate, while the lower-weighted business line has a higher reserve threshold.

    {\begin{remark}\label{remark: g near zero}
        Suppose that $F_1(0)=F_2(0)=0$. For all of the solutions obtained in Sections \ref{sec:bdd} and \ref{sec:ubdd}, $g$ is given by $g(x)=K\int_0^x\exp\left[\int_z^{x_0}\frac{\kappa_1}{G^{-1}(y)}dy\right]dz$ in a neighborhood of 0, for some constants $K,x_0>0$. Since $\mu_i(s)=s+o(s)$ and $\sigma_i^2(s)=s^2+o(s^2)$ as $s\downarrow 0$, the integrand in \eqref{eqn:G} converges to $\frac{\kappa_1^2+\kappa_2^2}{\kappa_1(\kappa_1^2+\kappa_2^2+2\delta)}$ as $z\downarrow 0$, so that $\kappa_1 G^{-1}(y)^{-1}\sim \frac{1-\gamma_0}{y}$ with
        \begin{equation*}
            \gamma_0:=\frac{2\delta}{\kappa_1^2+\kappa_2^2+2\delta}\in(0,1).
        \end{equation*}
        Hence, $g'(x)$ behaves like $x^{\gamma_0-1}$ and $g(x)$ behaves like $x^{\gamma_0}$ as $x\downarrow 0$. In particular, $g'(0+)=\infty$, and $g$ increases steeply near zero when $\gamma_0$ is small. This is a consequence of the assumption that reinsurance is cheap: when the aggregate reserve level is small, the insurer can cede almost all of its risk at no additional cost.
    \end{remark}}

	\section{Numerical Examples}\label{sec:numerical}

    In this section, we present several numerical examples to gain further insight into the main results in Sections \ref{sec:bdd} and \ref{sec:ubdd}. Since the capital-injection strategies were presented in Theorems \ref{thm:op_L} and \ref{thm: ubdd uniform strat} and the dividend payout strategy is a threshold strategy in the bounded-dividend case and a singular control in the unbounded-dividend case, we focus on illustrating the optimal (excess-of-loss) reinsurance strategies for each of the main results, along with their corresponding value functions.

     We fix the following parameters: $\kappa_1=4$, $\kappa_2=2$, $\delta=0.5$, and $a=0.3$. For the claim size distributions $F_1$ and $F_2$ in the bounded support case, we use uniform distributions on $[0,1]$ and ${[0,1.5]}$, respectively, which correspond to $M_1=1$ and ${M_2=1.5}$. In the unbounded support case $(M_1=M_2=\infty)$, we use exponential distributions with {rate} parameters $1$ and ${1.5}$. Figures \ref{fig:row1} to \ref{fig:row5} display the value functions on the left and the reinsurance strategies on the right. The vertical dotted lines represent the threshold levels.

    \begin{figure}[htb]
		\centering
		\begin{subfigure}{0.45\textwidth}
			\centering
			\includegraphics[width=\linewidth, trim = 0cm 0.5cm 1cm 2cm, clip = true]{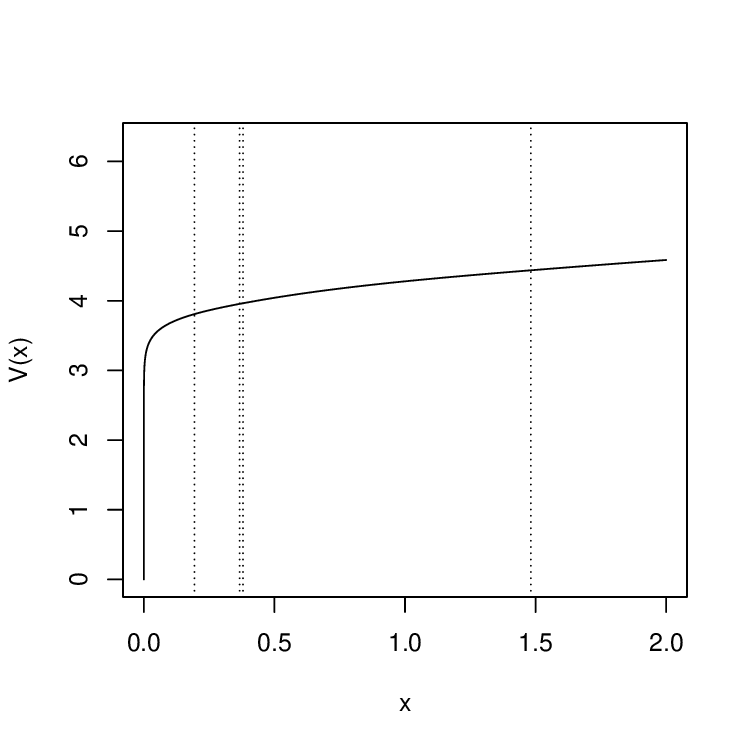} 
			%\caption{Caption for the first plot.}
			%\label{fig:plot1}
		\end{subfigure}
		\hfill
		\begin{subfigure}{0.45\textwidth}
			\centering
			\includegraphics[width=\linewidth, trim = 0cm 0.5cm 1cm 2cm, clip = true]{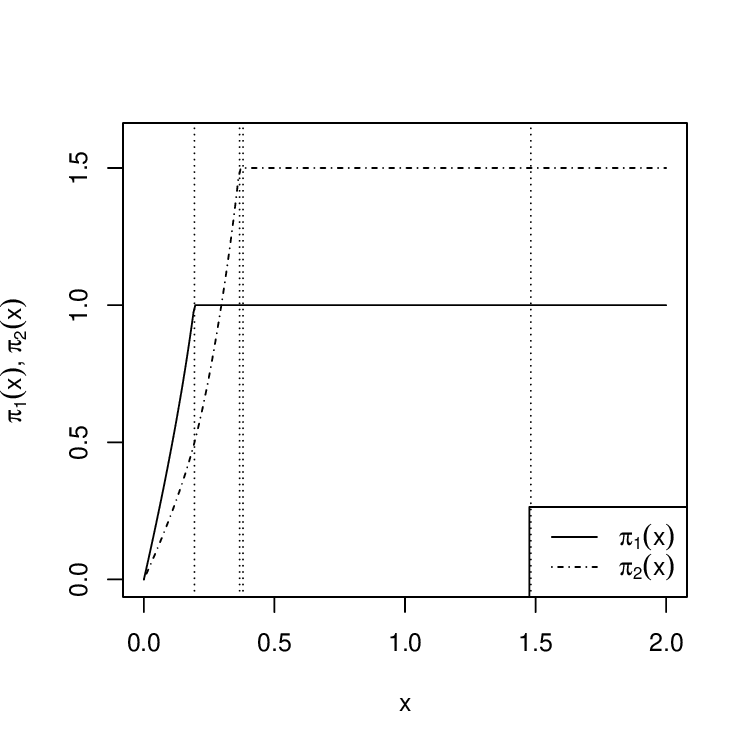}
			%\caption{Caption for the second plot.}
			%\label{fig:plot2}
		\end{subfigure}
		\\[-2ex]
		\caption{\centering Uniformly distributed claims: $\kappa_1=4$, $\kappa_2=2$, $\delta =0.5$, $a=0.3$, $\overline c_1=3$, ${\overline c_2=3}$ \\ ({$M_1=1$, $M_2=1.5$, Line 1 retains all of its risk from $x=0.19$ onward,  $w_0=0.37 < u_1=0.38<u_2=1.48$}, corresponding to Theorem \ref{thm:bdd w0<u1<u2}).}
		\label{fig:row1}
	\end{figure}

    \begin{figure}[htb]
		\centering
		\begin{subfigure}{0.45\textwidth}
			\centering
			\includegraphics[width=\linewidth, trim = 0cm 0.5cm 1cm 2cm, clip = true]{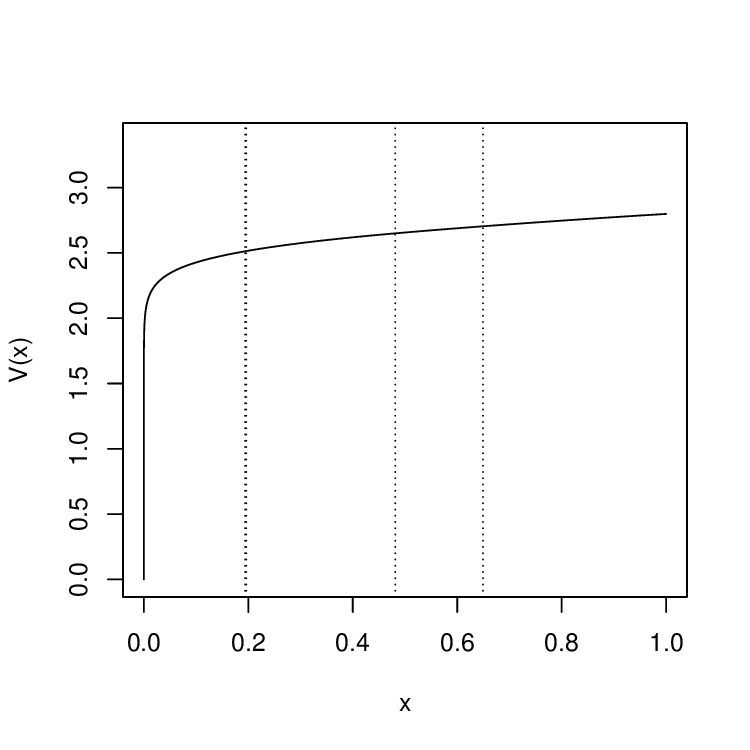} 
			%\caption{Caption for the first plot.}
			%\label{fig:plot1}
		\end{subfigure}
		\hfill
		\begin{subfigure}{0.45\textwidth}
			\centering
			\includegraphics[width=\linewidth, trim = 0cm 0.5cm 1cm 2cm, clip = true]{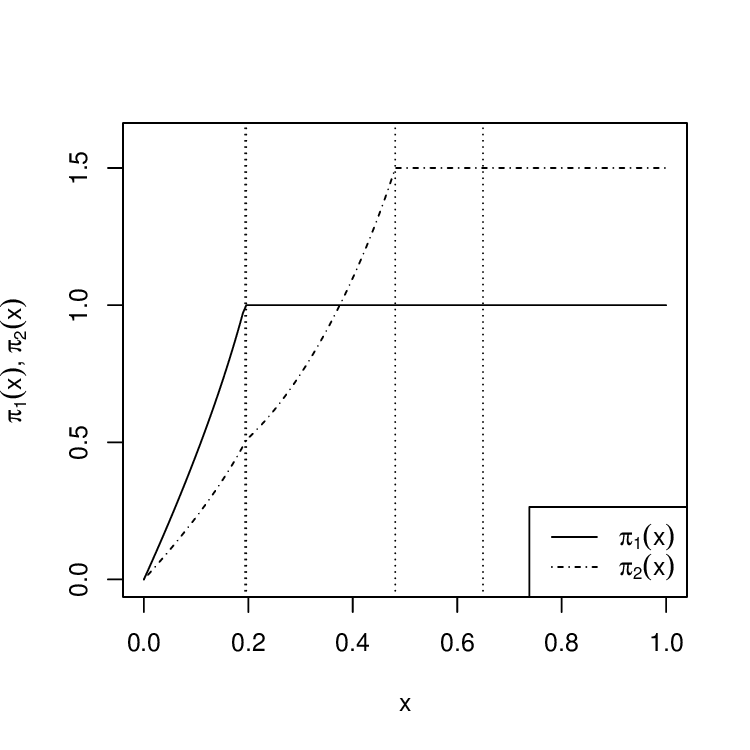}
			%\caption{Caption for the second plot.}
			%\label{fig:plot2}
		\end{subfigure}
		\\[-2ex]
		\caption{\centering Uniformly distributed claims: $\kappa_1=4$, $\kappa_2=2$, $\delta =0.5$, $a=0.3$, $\overline c_1=3$, $\overline c_2=1$ \\ ({$M_1=1$, $M_2=1.5$, Line 1 retains all of its risk from $x=0.19$ onward,  $u_1=0.20 < w_0=0.48<u_2=0.65$}, corresponding to Theorem \ref{thm:bdd u1<w0<u2}).}
		\label{fig:row2}
	\end{figure}

    \begin{figure}[htb]
		\centering
		\begin{subfigure}{0.45\textwidth}
			\centering
			\includegraphics[width=\linewidth, trim = 0cm 0.5cm 1cm 2cm, clip = true]{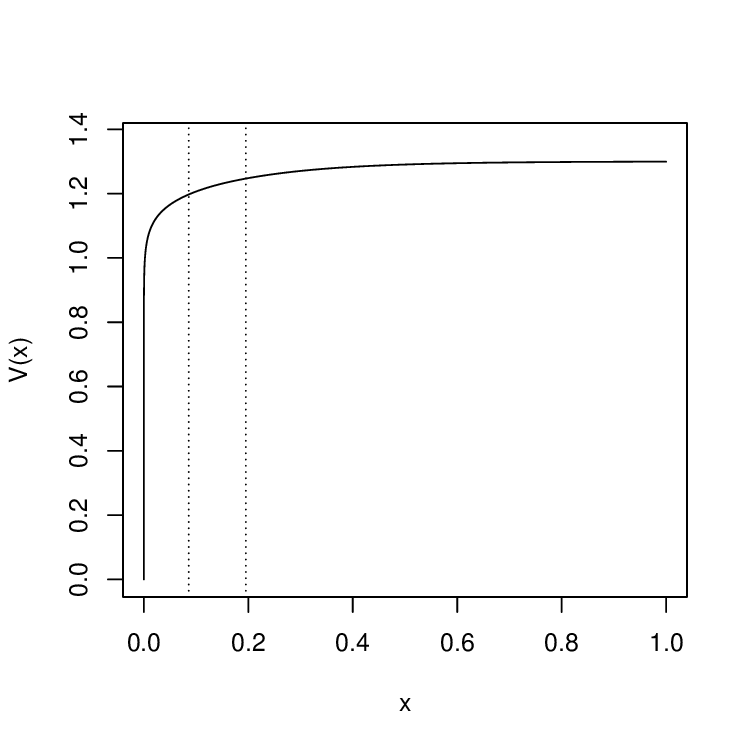} 
			%\caption{Caption for the first plot.}
			%\label{fig:plot1}
		\end{subfigure}
		\hfill
		\begin{subfigure}{0.45\textwidth}
			\centering
			\includegraphics[width=\linewidth, trim = 0cm 0.5cm 1cm 2cm, clip = true]{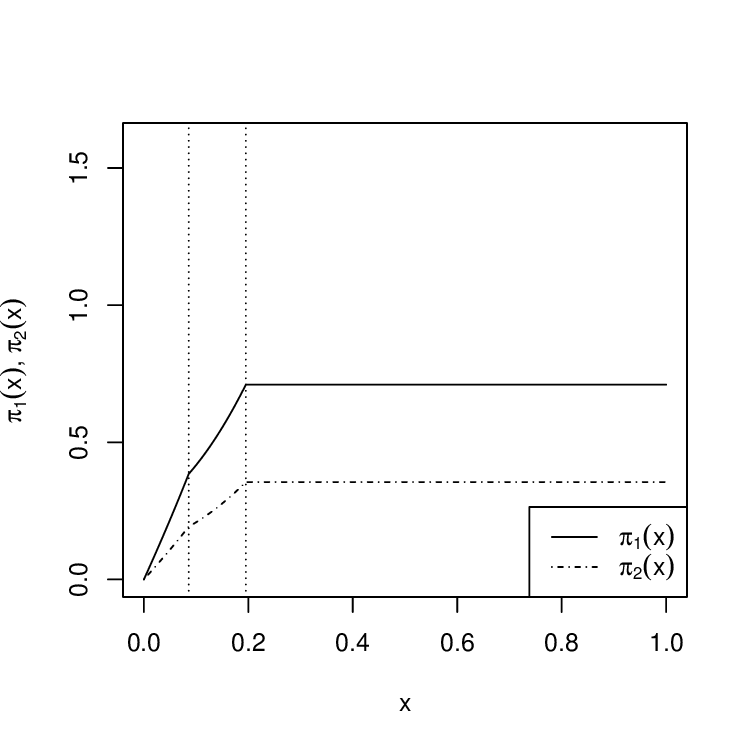}
			%\caption{Caption for the second plot.}
			%\label{fig:plot2}
		\end{subfigure}
		\\[-2ex]
		\caption{\centering Uniformly distributed claims: $\kappa_1=4$, $\kappa_2=2$, $\delta =0.5$, $a=0.3$, $\overline c_1=1$, $\overline c_2=0.5$ \\ ($M_0=0.71$, $M_1=1$, {$M_2=1.5$,  $u_1=0.09 <u_2=0.20$}, corresponding to Theorem \ref{thm: bdd u1<u2<w0 or M infty}).}
		\label{fig:row3}
	\end{figure}
    \begin{figure}[htb]
		\centering
		\begin{subfigure}{0.45\textwidth}
			\centering
			\includegraphics[width=\linewidth, trim = 0cm 0.5cm 1cm 2cm, clip = true]{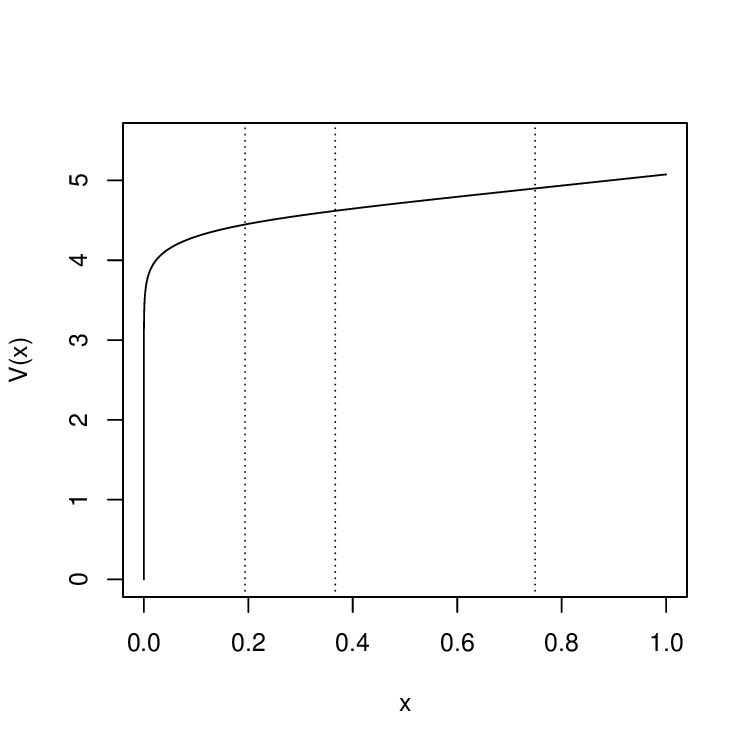} 
			%\caption{Caption for the first plot.}
			%\label{fig:plot1}
		\end{subfigure}
		\hfill
		\begin{subfigure}{0.45\textwidth}
			\centering
			\includegraphics[width=\linewidth, trim = 0cm 0.5cm 1cm 2cm, clip = true]{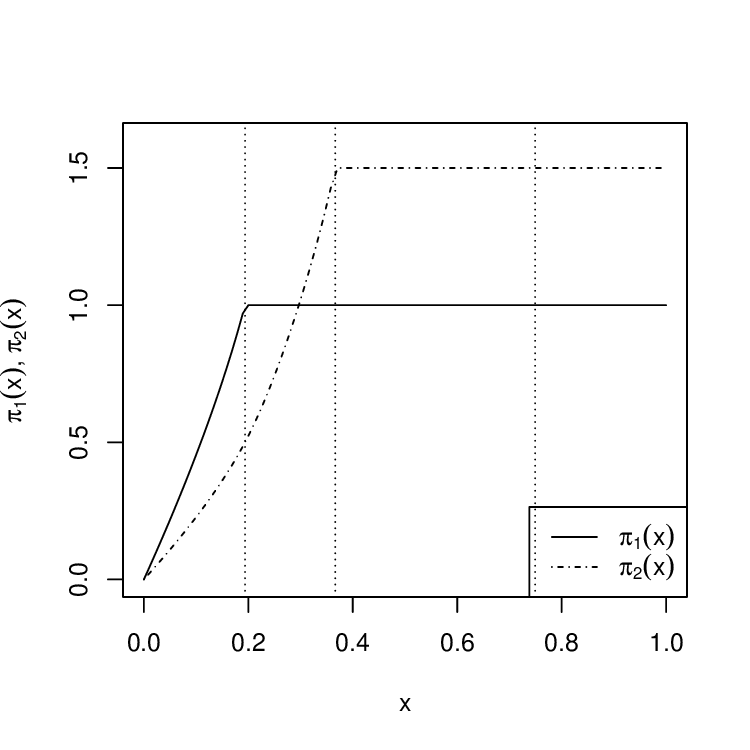}
			%\caption{Caption for the second plot.}
			%\label{fig:plot2}
		\end{subfigure}
		\\[-2ex]
		\caption{\centering Uniformly distributed claims: $\kappa_1=4$, $\kappa_2=2$, $\delta =0.5$, $a=0.3$\\ ($M_1=1$, {$M_2=1.5$, Line 1 retains all of its risk from $x=0.19$ onward, $w_0=0.37 < u_1=0.75$}, corresponding to Theorem \ref{thm:ubdd M1 finite}).}
		\label{fig:row4}
	\end{figure}
	
	\begin{figure}[htb]
		\centering
		\begin{subfigure}{0.45\textwidth}
			\centering
			\includegraphics[width=\linewidth, trim = 0cm 0.5cm 1cm 2cm, clip = true]{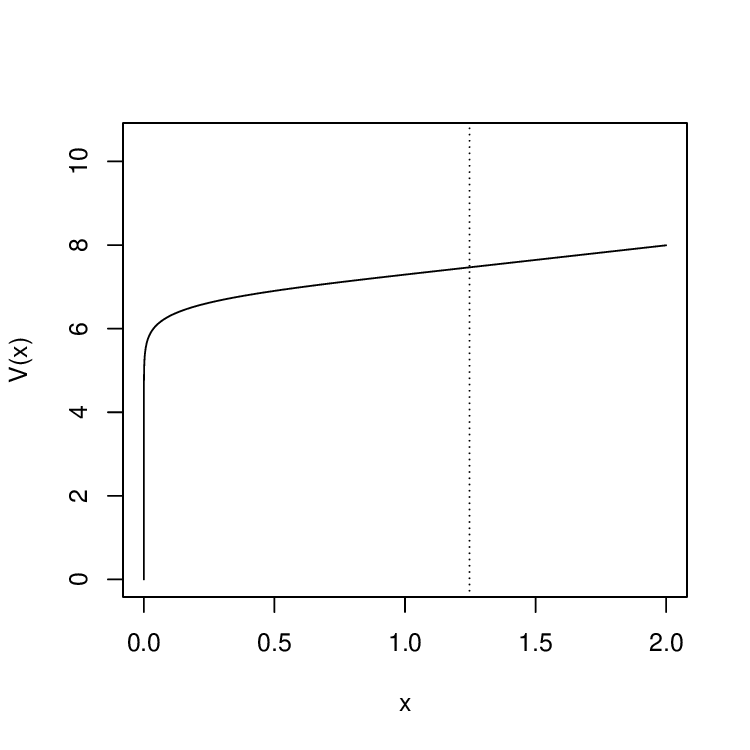} 
			%\caption{Caption for the first plot.}
			%\label{fig:plot1}
		\end{subfigure}
		\hfill
		\begin{subfigure}{0.45\textwidth}
			\centering
			\includegraphics[width=\linewidth, trim = 0cm 0.5cm 1cm 2cm, clip = true]{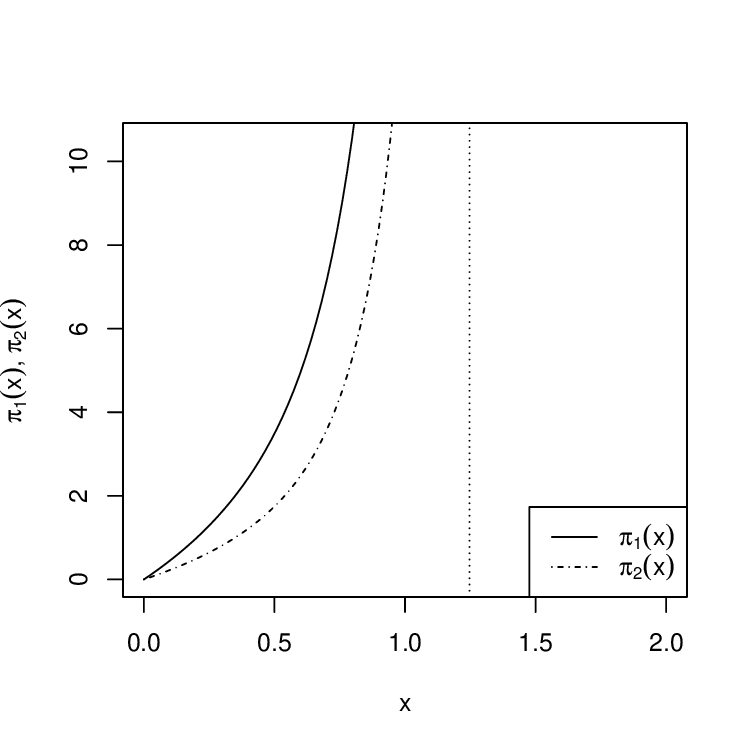}
			%\caption{Caption for the second plot.}
			%\label{fig:plot2}
		\end{subfigure}
		\\[-2ex]
	\caption{\centering Exponentially distributed claims: $\kappa_1=4$, $\kappa_2=2$, $\delta =0.5$, $a=0.3$\\ ($M_1=\infty$, $M_2=\infty$, $u_1=1.25$, corresponding to Theorem \ref{thm: unbdd M1 infty}).}
		\label{fig:row5}
	\end{figure}

   With these parameters, {as long as both lines purchase reinsurance, the optimal retention} level for Line 2 is exactly half that of Line 1: $\frac{\kappa_2}{\kappa_1}=\frac{1}{2}$. For uniformly distributed claim sizes, those in Lines 1 and 2 are capped at $M_1=1$ and ${M_2=1.5}$, respectively. This means that if the optimal excess-of-loss retention {level of Line $i$ reaches $M_i$}, then Line $i$ retains all of its risk. {Since $\frac{M_1}{\kappa_1}=0.25<\frac{M_2}{\kappa_2}=0.75$, Line 1 retains all of its risk first, and $\widetilde M=3$.} For an exponential distribution, claim sizes are unbounded and have support $[0,\infty)$.

    Figures \ref{fig:row1} to \ref{fig:row3} correspond to the main results for the bounded-dividend-rate case in Section \ref{sec:bdd}.  The maximum dividend rates of the business lines vary, and their sum decreases noticeably across the three figures. Figures \ref{fig:row4} and \ref{fig:row5} correspond to the main results for the unbounded dividend rates in Section \ref{sec:ubdd}. The optimal excess-of-loss retention level increases as the aggregate reserve level increases across all examples. Equivalently, the amount of risk ceded to the reinsurer decreases as the aggregate reserve level increases.
    
    By definition, the reinsurance threshold level $w_0$ indicates that the insurer should retain all risk of {both lines} once the aggregate reserve level exceeds it. In Figures \ref{fig:row1}, \ref{fig:row2}, and \ref{fig:row4}, {Line 1 retains all of its risk once the aggregate reserve level exceeds 0.19, while the retention level of Line 2 continues to increase until it reaches $M_2=1.5$ at $x=w_0$. Hence, reinsurance is phased out line by line. In Figure \ref{fig:row2}, the reinsurance threshold level $w_0$ lies between the dividend threshold levels $u_1$ and $u_2$, so that Line 2 still purchases reinsurance after it starts paying dividends.} In the scenario where $w_0$ does not exist, as stated in Theorem \ref{thm: bdd u1<u2<w0 or M infty}, at least one line must always cede a portion of its risk to the reinsurer, regardless of the aggregate reserve level. Figure \ref{fig:row3} illustrates this scenario, where the retention level of Line 1 is capped at $M_0=0.71$ {and that of Line 2 at $\frac{\kappa_2}{\kappa_1}M_0=0.36$, so that both lines purchase reinsurance at every aggregate reserve level}.
    
    In Figure \ref{fig:row5}, the optimal retention level increases without bound,
with $x=u_1=1.25$ acting as a vertical asymptote for the retention level. As stated in Theorem \ref{thm: ubdd uniform strat}, 
    the aggregate reserve level must remain below $u_1$ since any reserve above $u_1$ is allocated to dividends and capital transfers between lines.

    The value functions are all increasing and concave. {They satisfy $V(0,0)=0$, but increase steeply near zero since $g(x)$ behaves like $x^{\gamma_0}$ with $\gamma_0=\frac{2\delta}{\kappa_1^2+\kappa_2^2+2\delta}=\frac{1}{21}$ (see Remark \ref{remark: g near zero}); for instance, $g(10^{-6})\approx0.58 g(0.1)$ in Figure \ref{fig:row1}.} For the bounded dividend rates (Figures \ref{fig:row1} to \ref{fig:row3}), they approach the limit $\frac{a\overline c_1+(1-a)\overline c_2}{\delta}$. In particular, the value function in Figure \ref{fig:row3} approaches this limit more quickly than in the other cases, which can be attributed to the smaller values of $u_1$ and $u_2$.

	\section{Proof of Main Results}\label{sec:proof}
	
	In this section, we provide the proofs of Theorems \ref{thm:bdd w0<u1<u2}, \ref{thm:bdd u1<w0<u2}, \ref{thm: bdd u1<u2<w0 or M infty}, \ref{thm:ubdd M1 finite}, and \ref{thm: unbdd M1 infty}.
	
	\subsection{Proof of Theorem \ref{thm:bdd w0<u1<u2}}
	
	\subsubsection*{Deriving the analytical solution}
	
	Suppose for now that $w_0<u_1\leq u_2$ and that $w_0$ exists. In the region $\{x<w_0\}$, {the optimal retention levels are given by \eqref{eqn:optimal pi constrained}. Since $\mu_i$ and $\sigma_i^2$ are constant on $[M_i,\infty)$, we can replace $\Pi_i(x)$ by $\widehat\pi_i(x)=\frac{\kappa_i}{\kappa_1}\widehat\pi_1(x)$ in the HJB equation, which then becomes}
	\begin{equation}\label{eqn:hjb x<w0}
		\begin{aligned}
			0
			&=\left[\kappa_1\mu_1({\widehat\pi_1(x)})+\kappa_2\mu_2\left(\frac{\kappa_2}{\kappa_1}{\widehat\pi_1(x)}\right)-\frac{\kappa_1}{2{\widehat\pi_1(x)}}\left(\sigma_1^2({\widehat\pi_1(x)})+\sigma_2^2\left(\frac{\kappa_2}{\kappa_1}{\widehat\pi_1(x)}\right)\right)\right]g'(x)-\delta g(x).
		\end{aligned}
	\end{equation}
	Differentiating with respect to $x$ leads to
	\begin{equation}
		\begin{aligned}
			0
			&=\left[\frac{\kappa_1\frac{\dd {\widehat\pi_1(x)}}{\dd x}}{2 ({\widehat\pi_1(x)})^2}\left(\sigma_1^2( {\widehat\pi_1(x)})+\sigma_2^2\left(\frac{\kappa_2}{\kappa_1}{\widehat\pi_1(x)}\right)\right)-\delta\right]g'(x)\\
			&\quad + \left[\kappa_1\mu_1({\widehat\pi_1(x)})+\kappa_2\mu_2\left(\frac{\kappa_2}{\kappa_1}{\widehat\pi_1(x)}\right)-\frac{\kappa_1}{2{\widehat\pi_1(x)}}\left(\sigma_1^2({\widehat\pi_1(x)})+\sigma_2^2\left(\frac{\kappa_2}{\kappa_1}{\widehat\pi_1(x)}\right)\right)\right]g''(x).
		\end{aligned}
	\end{equation}
	Using \eqref{eqn:optimal pi} once more yields
	\begin{equation}
		0=\left[\frac{\kappa_1^2+\kappa_1\frac{\dd {\widehat\pi_1(x)}}{\dd x}}{2 ({\widehat\pi_1(x)})^2}\left(\sigma_1^2({\widehat\pi_1(x)})+\sigma_2^2\left(\frac{\kappa_2}{\kappa_1}{\widehat\pi_1(x)}\right)\right)-\frac{\kappa_1^2}{{\widehat\pi_1(x)}}\mu_1({\widehat\pi_1(x)})-\frac{\kappa_1\kappa_2}{{\widehat\pi_1(x)}}\mu_2\left(\frac{\kappa_2}{\kappa_1}{\widehat\pi_1(x)}\right)-\delta\right]g'(x).
	\end{equation}
	Since $g$ is assumed to be strictly increasing and concave, we obtain the following differential equation for ${\widehat\pi_1(x)}$:
	\begin{equation*}
		\frac{\dd {\widehat\pi_1(x)}}{\dd x}
=
\left(\frac{\dd G(y)}{\dd y}\Bigg|_{y={\widehat\pi_1(x)}}\right)^{-1},
	\end{equation*}
	where $G$ is defined by \eqref{eqn:G}. {Since $\sigma_i^2(s)\leq s\mu_i(s)$ for all $s\geq 0$ (see the proof of Lemma \ref{lemma:h is inc}), we have $\kappa_1\overline N_2(z)\leq z\overline N_1(z)$, so that the denominator of the integrand in \eqref{eqn:G} is bounded below by $z\overline N_1(z)+\frac{2\delta}{\kappa_1}z^2>0$ for $z>0$. Hence,} $G$ is continuous and strictly increasing, and its inverse, denoted by $G^{-1}$, exists. Since ${\widehat \pi_1}(0)=0$, we then obtain 
	\begin{equation}\label{eqn:m=G^-1}
		{\widehat\pi_1(x)}=G^{-1}(x).
	\end{equation}
	Moreover, since $\frac{\dd {\widehat \pi_1}(x)}{\dd x}>0$, ${\widehat \pi_1}(0)=0$, {and $G(\widetilde M)<\infty$, choosing $w_0=G(\widetilde M)$} proves the existence of $w_0$ defined in \eqref{eqn:defn of w0}. From \eqref{eqn:optimal pi}, we can see that $-\frac{\kappa_1}{{\widehat \pi_1}(x)}=\frac{g''(x)}{g'(x)}=\frac{\dd}{\dd x}\ln g'(x)$. Hence, a solution to the HJB equation \eqref{eqn:hjb x<w0} that satisfies $g(0)=0$ is given by
	\begin{equation*}
		g(x)=K_1\int_0^x \exp\left[\int_{z}^{w_0}\frac{\kappa_1}{G^{-1}(y)}dy\right]dz,
	\end{equation*}
	where $K_1>0$ is an unknown constant.	
	
	{For $x\geq w_0$, we conjecture that $\widehat\pi_1(x)\geq \widetilde M$, so that $(\Pi_1,\Pi_2)(x)=(M_1,M_2)$ by \eqref{eqn:optimal pi constrained}; that is, both lines retain all of their risk. This conjecture is verified in Lemma \ref{lemma: both lines retain} below.} Hence, in the region $\{w_0<x<u_1\}$, the HJB equation becomes
	\begin{equation}\label{eqn:hjb w0<x<u1}
		0=\frac{1}{2}N_2g''(x)+N_1g'(x)-\delta g(x),
	\end{equation}
	where $N_1=\overline N_1({\widetilde M})$ and $N_2=\overline N_2({\widetilde M})$ are defined in \eqref{eqn:notations}.	The solution is given by
	\begin{equation}\label{eqn:g_2}
		g_2(x)=K_{2+}e^{\gamma_{2+}(x-w_0)}+K_{2-}e^{\gamma_{2-}(x-w_0)},
	\end{equation}
	where $\gamma_{2\pm}:=\overline \gamma_{2\pm}({\widetilde M})$ are given by \eqref{eqn:notations} and $K_{2\pm}$ are unknown constants.
	
	In the region $\{u_1<x<u_2\}$, the HJB equation becomes
	\begin{equation}\label{eqn:hjb 3}
		0=\frac{1}{2}N_2g''(x)+(N_1-\overline c_2)g'(x)-\delta g(x)+(1-a)\overline c_2,
	\end{equation}
	whose solution is given by
	\begin{equation}\label{eqn:g_3}
		g_3(x)=K_{3+}e^{\gamma_{3+}(x-u_1)}+K_{3-}e^{\gamma_{3-}(x-u_1)}+\frac{(1-a)\overline c_2}{\delta},
	\end{equation}
	where $\gamma_{3\pm}:=\overline \gamma_{3\pm}({\widetilde M})$ are given by \eqref{eqn:notations} and $K_{3\pm}$ are unknown constants.

	In the region $\{x>u_2\}$, the HJB equation becomes
	\begin{equation}\label{eqn:hjb 4}
		0=\frac{1}{2}N_2g''(x)+(N_1-\overline c_1-\overline c_2)g'(x)-\delta g(x)+a\overline c_1+(1-a)\overline c_2,
	\end{equation}
	whose solution is given by
	\begin{equation}\label{eqn:g_4}
		g_4(x)=K_{4-}e^{\gamma_{4-}(x-u_2)}+\frac{a\overline c_1+(1-a)\overline c_2}{\delta},
	\end{equation}
	where $\gamma_{4-}:=\overline \gamma_{4-}({\widetilde M})$ is defined in \eqref{eqn:notations}. We conjecture the following solution:
	\begin{equation}\label{eqn:conjecture g one}
		g(x)=
		\begin{cases}
			K_1\int_0^x \exp\left[\int^{w_0}_z\frac{\kappa_1}{G^{-1}(y)}dy\right]dz&\mbox{if $x<w_0$,}\\
			K_{2+}e^{\gamma_{2+}(x-w_0)}+K_{2-}e^{\gamma_{2-}(x-w_0)}&\mbox{if $w_0<x<u_1$,}\\
			K_{3+}e^{\gamma_{3+}(x-u_1)}+K_{3-}e^{\gamma_{3-}(x-u_1)}+\frac{(1-a)\overline c_2}{\delta}&\mbox{if $u_1<x<u_2$,}\\
			K_{4-}e^{\gamma_{4-}(x-u_2)}+\frac{a\overline c_1+(1-a)\overline c_2}{\delta}&\mbox{if $x>u_2$,}
		\end{cases}
	\end{equation}
	where $w_0=G({\widetilde M})$ and $K_1,\, K_{2\pm},\, K_{3\pm}, \, K_{4-}, \, u_1,\, u_2$ are yet to be determined.
	
	We now solve for the unknowns using the principle of smooth fit. At $x=w_0$, we have the following system of equations:
	\begin{equation}\label{eqn:system at xM1 case 1}
		\begin{aligned}
			K_1&=K_{2+}\gamma_{2+}+K_{2-}\gamma_{2-}\\
			-K_1\frac{\kappa_1}{{\widetilde M}}&=K_{2+}\gamma_{2+}^2+K_{2-}\gamma_{2-}^2,
		\end{aligned}
	\end{equation}
	whose solution is given by
	\begin{equation}
		K_{2+}=K_1\alpha_{2+}\quad\mbox{and}\quad K_{2-}=K_1\alpha_{2-},
	\end{equation}
	where $\alpha_{2\pm}$ are given in \eqref{eqn:alpha2pm}. Recall that $\gamma_{2-}<0<\gamma_{2+}$. It is then easy to see that $\alpha_{2-}<0$. To prove $\alpha_{2+}>0$, it suffices to show that $-\gamma_{2-}-\frac{\kappa_1}{{\widetilde M}}>0$. This result is proved in the following lemma. 
	\begin{lemma}\label{lemma: k1/M1>-gamma2-}
		$			-\frac{\kappa_1}{\gamma_{2-}}<{\widetilde M}$.
	\end{lemma}
\begin{proof}
	Let $s_2:=\frac{\kappa_2}{\kappa_1}{\widetilde M}$. Since
	$\sigma_i^2(s)\leq s\mu_i(s)$ for all $s\geq0$, we have
	\begin{equation*}
		N_2
		=
		\sigma_1^2({\widetilde M})+\sigma_2^2(s_2)
		\leq
		{\widetilde M}\mu_1({\widetilde M})+s_2\mu_2(s_2).
	\end{equation*}
	Therefore,
	\begin{equation*}
		\kappa_1N_2
		\leq
		{\widetilde M}\left[
			\kappa_1\mu_1({\widetilde M})+\kappa_2\mu_2(s_2)
		\right]
		=
		{\widetilde M}N_1.
	\end{equation*}
	Consequently,
	\begin{equation*}
		-\frac{\kappa_1}{\gamma_{2-}}
		=
		\frac{\kappa_1N_2}
		{N_1+\sqrt{N_1^2+2\delta N_2}}
		<
		\frac{\kappa_1N_2}{2N_1}
		\leq
		\frac{{\widetilde M}}{2}
		<
		{\widetilde M}.
	\end{equation*}
	This proves the result.
\end{proof} We can then rewrite $g_2$ as
	\begin{equation*}
		g_2(x)=K_1\left[\alpha_{2+}e^{\gamma_{2+}(x-w_0)}+\alpha_{2-}e^{\gamma_{2-}(x-w_0)}\right],
	\end{equation*}
	where $K_1>0$ is still unknown.
	
	By the definition of $u_1$ in \eqref{eqn:defn of u1 & u2}, we have $g'(u_1)=1-a$. Hence,
	\begin{equation*}
		1-a=g'_3(u_1)=K_{3+}\gamma_{3+}+K_{3-}\gamma_{3-},
	\end{equation*}
	which is equivalent to $K_{3+}$ defined in \eqref{eqn:K3+}. It suffices to show that $g_2(u_1)=g_3(u_1)$ and $g'_2(u_1)=g'_3(u_1)$ to ensure that $g$ is twice continuously differentiable at $x=u_1$. We then have the following system of equations:
	\begin{equation}\label{eqn:system at u1 bb}
		\begin{aligned}
			K_1\left[\alpha_{2+}e^{\gamma_{2+}(u_1-w_0)}+\alpha_{2-}e^{\gamma_{2-}(u_1-w_0)}\right]&=(1-a)\alpha_3\\
			K_1\left[\alpha_{2+}\gamma_{2+}e^{\gamma_{2+}(u_1-w_0)}+\alpha_{2-}\gamma_{2-}e^{\gamma_{2-}(u_1-w_0)}\right]&=(1-a),
		\end{aligned}
	\end{equation}
	where $\alpha_3$ is defined in \eqref{eqn:K3+}. From the second equation in \eqref{eqn:system at u1 bb}, we obtain $K_1$ defined in \eqref{eqn:K1 w0<u1<u2}. Dividing the first equation in \eqref{eqn:system at u1 bb} by the second equation yields
	\begin{equation*}
		\frac{\alpha_{2+}e^{\gamma_{2+}(u_1-w_0)}+\alpha_{2-}e^{\gamma_{2-}(u_1-w_0)}}{\alpha_{2+}\gamma_{2+}e^{\gamma_{2+}(u_1-w_0)}+\alpha_{2-}\gamma_{2-}e^{\gamma_{2-}(u_1-w_0)}}=\alpha_3,
	\end{equation*}
	which is equivalent to $u_1$ defined in \eqref{eqn:u1 and u2 w0<u1<u2}. We point out that we have yet to establish that $u_1$ is well defined.
	
	By the definition of $u_2$, we have $g'(u_2)=a$. Hence,
	\begin{equation*}
		a=g'_4(u_2)=K_{4-}\gamma_{4-},
	\end{equation*}
	or, equivalently,
	\begin{equation}
		K_{4-}=\frac{a}{\gamma_{4-}}.
	\end{equation}
	It suffices to show that $g'_3(u_2)=g'_4(u_2)$ and $g''_3(u_2)=g''_4(u_2)$ to ensure that $g$ is twice continuously differentiable at $x=u_2$. We then have the following system of equations:
	\begin{equation}\label{eqn:system at x=u2}
		\begin{aligned}
			K_{3+}\gamma_{3+}e^{\gamma_{3+}(u_2-u_1)}+K_{3-}\gamma_{3-}e^{\gamma_{3-}(u_2-u_1)}&=a\\
			K_{3+}\gamma_{3+}^2e^{\gamma_{3+}(u_2-u_1)}+K_{3-}\gamma_{3-}^2e^{\gamma_{3-}(u_2-u_1)}&=a\gamma_{4-}.
		\end{aligned}
	\end{equation}
	Dividing the second equation by the first equation yields
	\begin{equation*}
		\frac{K_{3+}\gamma_{3+}^2e^{\gamma_{3+}(u_2-u_1)}+K_{3-}\gamma_{3-}^2e^{\gamma_{3-}(u_2-u_1)}}{K_{3+}\gamma_{3+}e^{\gamma_{3+}(u_2-u_1)}+K_{3-}\gamma_{3-}e^{\gamma_{3-}(u_2-u_1)}}=\gamma_{4-},
	\end{equation*}
	which is equivalent to $u_2$ defined in \eqref{eqn:u1 and u2 w0<u1<u2}, where $K_{3-}$ is still unknown. Thus, we have obtained the form of the value function in \eqref{eqn:g w0<u1<u2}. The next steps are (i) to establish that the formulas for $u_1$ and $u_2$ in \eqref{eqn:u1 and u2 w0<u1<u2} are well defined, (ii) to guarantee that $w_0< u_1\leq u_2$, and (iii) to show that $g$ is increasing and concave. 
	
	\subsubsection*{Establishing the bounds for $K_{3-}$}
	
	We now establish the bounds for $K_{3-}$. Since the candidate value function $g$ must be positive for $x>0$, we must have $\alpha_3>0$ from the first equation in \eqref{eqn:system at u1 bb}. Since $\alpha_{2-}<0<\alpha_{2+}$, we must have $1-\gamma_{2+}\alpha_3>0$ for $u_1$ in \eqref{eqn:u1 and u2 w0<u1<u2} to be well defined. Combining these inequalities for $\alpha_3$ yields $0<\alpha_3<\frac{1}{\gamma_{2+}}$, which is equivalent to
	\begin{equation}\label{eqn:alpha 1st ineq}
		\underline{\alpha}:=-\frac{(1-a)\gamma_{3+}}{\gamma_{3+}-\gamma_{3-}}\left(\frac{1}{\gamma_{3+}}+\frac{\overline{c}_2}{\delta}\right)<K_{3-}<\frac{(1-a)\gamma_{3+}}{\gamma_{3+}-\gamma_{3-}}\left(\frac{1}{\gamma_{2+}}-\frac{1}{\gamma_{3+}}-\frac{\overline{c}_2}{\delta}\right)=:\overline{\alpha}.
	\end{equation}
	It is clear that $\underline\alpha<\overline\alpha$ since $\frac{1}{\gamma_{2+}}>0$. 
	
We have now established conditions under which \eqref{eqn:u1 and u2 w0<u1<u2} is well defined; however, these conditions do not guarantee that $w_0< u_1$. The following lemma shows that $\alpha_0$, defined in \eqref{eqn:alpha w0<u1<u2}, is greater than $\underline{\alpha}$.
	\begin{lemma}\label{lemma:alpha0, ul alpha, ol alpha}
		$\alpha_0 > \underline{\alpha}$.
	\end{lemma}
	\begin{proof}
		From \eqref{eqn:alpha 1st ineq} and \eqref{eqn:alpha w0<u1<u2}, $\alpha_0>\underline \alpha$ if and only if $N_1-\frac{\kappa_1N_2}{2{\widetilde M}}>0$. From the proof of Lemma \ref{lemma: k1/M1>-gamma2-}, $\kappa_1N_2\leq{\widetilde M}N_1$, so that $N_1-\frac{\kappa_1N_2}{2{\widetilde M}}\geq\frac{N_1}{2}>0$.
	\end{proof}
	The following lemma proves that $\alpha_0$ does not exceed the upper bound $\overline\alpha$.
	\begin{lemma}\label{lemma:alpha0 < ol alpha}
		$\alpha_0 < \overline{\alpha}$.
	\end{lemma}
	\begin{proof}
		We can rewrite $\frac{N_1}{\delta}-\frac{\kappa_1 N_2}{2\delta {\widetilde M}}<\frac{1}{\gamma_{2+}}$ as
		\begin{equation*}
			\left(2N_1-\frac{\kappa_1}{{\widetilde M}}N_2\right)\left(\sqrt{N_1^2+2\delta N_2}-N_1\right)<2\delta N_2.
		\end{equation*}
		From the proof of Lemma \ref{lemma:alpha0, ul alpha, ol alpha}, $2N_1-\frac{\kappa_1}{{\widetilde M}}N_2>0$. We can then rewrite the above inequality as
		\begin{equation*}
			\left(2N_1-\frac{\kappa_1}{{\widetilde M}}N_2\right)\sqrt{N_1^2+2\delta N_2}<2\delta N_2+2N_1^2-\frac{\kappa_1}{{\widetilde M}}N_1N_2,
		\end{equation*}
		where both sides are positive. Squaring both sides yields $N_1- \frac{\kappa_1N_2}{2{\widetilde M}}+\frac{\delta {\widetilde M}}{\kappa_1}>0$, which holds since $N_1- \frac{\kappa_1N_2}{2{\widetilde M}}>0$. This proves the result.
	\end{proof}

	The following result gives a necessary and sufficient condition to ensure that $w_0\leq u_1$.
	\begin{lemma}\label{lemma on xM1<u1}
		$w_0< u_1$ if and only if
		\begin{equation}\label{eqn:alpha 2nd ineq}
			\alpha_0<K_{3-}<\overline{\alpha}.
		\end{equation}
	\end{lemma}
	\begin{proof}
		From \eqref{eqn:u1 and u2 w0<u1<u2}, $w_0\leq u_1$ is equivalent to
		\begin{equation}\label{eqn:ineq1}
			\frac{\alpha_{2-}(\gamma_{2-}\alpha_3-1)}{\alpha_{2+}(1-\gamma_{2+}\alpha_3)}> 1.
		\end{equation}
		Moreover, $K_{3-}<\overline \alpha$ is equivalent to $1-\gamma_{2+}\alpha_3>0$, and $K_{3-}>\underline\alpha$ is equivalent to $\alpha_3>0$. Hence, we can rewrite \eqref{eqn:ineq1} as
		\begin{equation*}
			\alpha_3> \frac{N_1}{\delta}-\frac{\kappa_1 N_2}{2\delta {\widetilde M}},
		\end{equation*}
		or, equivalently,
		\begin{equation*}
			K_{3-}> \alpha_0.
		\end{equation*}
		Using Lemma \ref{lemma:alpha0, ul alpha, ol alpha} completes the proof.
	\end{proof}

	We now determine the signs of $K_{3\pm}$. First, we note that a sufficient condition such that $g$ is increasing, particularly in the region $\{u_1<x<u_2\}$, is $K_{3-}\le 0\le K_{3+}$. The following lemma proves that $K_{3-}<0$.
	\begin{lemma}
		$K_{3-}<0$.
	\end{lemma}
	\begin{proof}
		From \eqref{eqn:alpha 1st ineq}, it suffices to prove that
		\begin{equation*}
			\frac{\overline{c}_2}{\delta}\geq  \frac{1}{\gamma_{2+}}-\frac{1}{\gamma_{3+}}.
		\end{equation*}
		Write $k:=2\delta N_2>0$. Suppose otherwise that $\frac{\overline{c}_2}{\delta}< \frac{1}{\gamma_{2+}}-\frac{1}{\gamma_{3+}}$. This is equivalent to
		\begin{equation}\label{eqn:r1}
			\begin{aligned}
				&2\overline c_2\sqrt{(N_1-\overline c_2)^2+k}\cdot\sqrt{N_1^2+k}-\overline c_2\left(k-2N_1(N_1-\overline{c}_2)\right)\\
				&< (2N_1\overline c_2+k)\sqrt{(N_1-\overline c_2)^2+k}+(2\overline c_2(N_1-\overline c_2)-k)\sqrt{N_1^2+k}.
			\end{aligned}			
		\end{equation}
		The left-hand side of \eqref{eqn:r1} can be shown to be always positive. If the right-hand side of \eqref{eqn:r1} is negative, it results in a contradiction. Otherwise, we can square both sides and get
		\begin{equation*}
			\overline{c}_2^2k< 0,
		\end{equation*}
		which is also a contradiction. The proof is complete.
	\end{proof}
	
	We now show that $K_{3+}\geq 0$. Using \eqref{eqn:K3+}, we must have the following:
	\begin{equation*}
		K_{3-}\geq \frac{1-a}{\gamma_{3-}}.
	\end{equation*}
	The following lemma proves that $\frac{1-a}{\gamma_{3-}}$ does not exceed the upper bound $\overline\alpha$.
	\begin{lemma}
		$\frac{1-a}{\gamma_{3-}}<\overline{\alpha}$.
	\end{lemma}
	\begin{proof}
		It suffices to prove that
		\begin{equation}\label{eqn:c2/b<UB}
			\frac{\overline{c}_2}{\delta}<  \frac{1}{\gamma_{2+}}-\frac{1}{\gamma_{3-}}.
		\end{equation}
		From the elementary inequality
		\begin{equation*}
			\sqrt{A^2+B}-A<\frac{B}{2A},\quad A,B>0,
		\end{equation*}
		we obtain the following:
		\begin{equation*}
			\gamma_{2+}<\frac{\delta}{N_1}\Leftrightarrow\frac{1}{\gamma_{2+}}>\frac{N_1}{\delta}.
		\end{equation*}
		If $N_1-\overline{c}_2< 0$, then
		\begin{equation*}
			-\gamma_{3-}<\frac{\delta}{\overline{c}_2-N_1}\Leftrightarrow-\frac{1}{\gamma_{3-}}>\frac{\overline c_2-N_1}{\delta},
		\end{equation*}
		and \eqref{eqn:c2/b<UB} follows. If $N_1-\overline{c}_2\geq 0$, then
		\begin{equation*}
			\gamma_{2+}<\frac{\delta}{N_1}\leq\frac{\delta}{\overline{c}_2}\Leftrightarrow\frac{\overline{c}_2}{\delta}<\frac{1}{\gamma_{2+}},
		\end{equation*}
		and \eqref{eqn:c2/b<UB} follows since $\gamma_{3-}<0$. The proof is complete.
	\end{proof}

	We now determine the conditions under which $\frac{1-a}{\gamma_{3-}}$ or $\alpha_0$ serves as the lower bound for $K_{3-}$ via the lemma below. 
	\begin{lemma}\label{lemma on the 3 LBs}
		$\frac{1-a}{\gamma_{3-}}>\alpha_0$ if and only if 
			\begin{equation}\label{ineq:2}
				\overline c_2>N_1- \frac{\kappa_1N_2}{2{\widetilde M}}+\frac{\delta {\widetilde M}}{\kappa_1}.
			\end{equation}	
		\end{lemma}
	\begin{proof}		
		$\frac{1-a}{\gamma_{3-}}>\alpha_0$ is equivalent to
		\begin{equation*}
			1<\gamma_{3-}\left(\frac{N_1-\overline c_2}{\delta}-\frac{\kappa_1N_2}{2\delta {\widetilde M}}\right).
		\end{equation*}
		A necessary condition for the above inequality to hold is that $\overline c_2>N_1-\frac{\kappa_1N_2}{2{\widetilde M}}$. Moreover, the inequality is equivalent to the following:
		\begin{equation}\label{eqn:r3}
			2\delta {\widetilde M}N_2 +2{\widetilde M}(N_1-\overline c_2)^2-\kappa_1 N_2(N_1-\overline c_2)<(\kappa_1N_2-2{\widetilde M}(N_1-\overline c_2))\sqrt{(N_1-\overline c_2)^2+2\delta N_2}.
		\end{equation}
		The above inequality is immediately satisfied when $2{\widetilde M}(N_1-\overline c_2)^2-\kappa_1 N_2(N_1-\overline c_2)<-2\delta {\widetilde M}N_2$ (i.e., the left-hand side is negative) and $\overline c_2>N_1-\frac{\kappa_1N_2}{2{\widetilde M}}$ (i.e., the right-hand side is positive). We call this scenario the ``trivial" case.
		
		Squaring both sides of \eqref{eqn:r3} yields \eqref{ineq:2}, which satisfies the necessary condition. Note that squaring both sides of \eqref{eqn:r3} is valid if $N_1-\overline c_2<0$. Moreover, if $N_1-\overline c_2>0$, we can rewrite \eqref{ineq:2} as $2{\widetilde M}(N_1-\overline c_2)^2-\kappa_1 N_2(N_1-\overline c_2)<-\frac{2\delta {\widetilde M}^2(N_1-\overline c_2)}{\kappa_1}$. Since $-2\delta {\widetilde M}N_2<-\frac{2\delta {\widetilde M}^2(N_1-\overline c_2)}{\kappa_1}$ if $0<N_1-\overline c_2<\frac{\kappa_1N_2}{{\widetilde M}}$, then the trivial case is covered by \eqref{ineq:2}. The proof is complete.
	\end{proof}
	
	Following the results of Lemmas \ref{lemma:alpha0, ul alpha, ol alpha}, \ref{lemma on xM1<u1}, and \ref{lemma on the 3 LBs}, we now present the two cases that yield the expression for the lower bound of $K_{3-}$ (denoted by $\alpha_{LB}$):
	\begin{itemize}
		\item[(i)] If $\overline c_2\geq N_1-\frac{\kappa_1 N_2}{2{\widetilde M}}+\frac{\delta {\widetilde M}}{\kappa_1}$, then $\alpha_{LB}=\frac{1-a}{\gamma_{3-}}\geq \alpha_0>\underline\alpha$.
		\item[(ii)] If $\overline c_2<N_1-\frac{\kappa_1 N_2}{2{\widetilde M}}+\frac{\delta {\widetilde M}}{\kappa_1}$, then $\alpha_{LB}=\alpha_0> \frac{1-a}{\gamma_{3-}}$.
	\end{itemize}
	Hence, we obtain the formula for $\alpha_{LB}$ given in \eqref{eqn:alpha w0<u1<u2}. Thus, $K_{3-}<0<K_{3+}$ and $w_0<u_1$ if the following inequalities hold:
	\begin{equation*}
		\alpha_{LB}<K_{3-}<\overline \alpha.
	\end{equation*}

    \begin{remark}
		It follows that $g$ is strictly increasing in the region $\{u_1<x<u_2\}$. It can easily be shown that $g$ is also strictly increasing in the other regions. By the continuity of the first derivative, $g$ is strictly increasing for all $x>0$.
		
		It can also be easily shown that $g$ is strictly concave in the regions $\{x<w_0\}$ and $\{x>u_2\}$. Since $g'''>0$ in the regions $\{w_0<x<u_1\}$ and $\{u_1<x<u_2\}$, then, by the continuity of the second derivative, $g$ is strictly concave on $[w_0,u_2]$.
	\end{remark}

    {We now verify that both lines retain all of their risk for $x\geq w_0$, as conjectured above.
    \begin{lemma}\label{lemma: both lines retain}
        Suppose that condition (ii) of Theorem \ref{thm:bdd w0<u1<u2} holds. Then, $\widehat\pi_1(x)\geq\widetilde M$ for $x\geq w_0$.
    \end{lemma}
    \begin{proof}
        Write $y(x):=\widehat\pi_1(x)=-\kappa_1\frac{g'(x)}{g''(x)}$, which is continuous since $g$ is twice continuously differentiable and $g''<0$. In each of the regions $\{w_0<x<u_1\}$, $\{u_1<x<u_2\}$, and $\{x>u_2\}$, the HJB equation has the form $0=\frac{1}{2}N_2g''(x)+(N_1-C)g'(x)-\delta g(x)+D$ for some constants $C\geq 0$ and $D$. Differentiating with respect to $x$ and using $g''(x)=-\frac{\kappa_1}{y(x)}g'(x)$ yields
        \begin{equation}\label{eqn:ode for y}
            y'(x)=F_C(y(x)),\qquad F_C(y):=-\kappa_1+\frac{2(N_1-C)y+\frac{2\delta}{\kappa_1}y^2}{N_2}.
        \end{equation}
        In the region $\{x>u_2\}$, we have $y(x)=-\frac{\kappa_1}{\gamma_{4-}}$ from \eqref{eqn:g_4}. The same computation as in the proof of Lemma \ref{lemma on u2>w0} shows that $-\frac{\kappa_1}{\gamma_{4-}}\geq \widetilde M$ is equivalent to condition (ii). In particular, $y(u_2)=-\frac{\kappa_1}{\gamma_{4-}}$ is a zero of $F_{\overline c_1+\overline c_2}$, which yields
        \begin{equation*}
            F_{\overline c_2}(y(u_2))=F_{\overline c_1+\overline c_2}(y(u_2))+\frac{2\overline c_1y(u_2)}{N_2}=\frac{2\overline c_1y(u_2)}{N_2}>0.
        \end{equation*}
        Since \eqref{eqn:ode for y} with $C=\overline c_2$ is an autonomous first-order equation, $y$ is monotone on $[u_1,u_2]$. Since $y'(u_2-)=F_{\overline c_2}(y(u_2))>0$, $y$ is increasing on $[u_1,u_2]$. In the region $\{w_0<x<u_1\}$, we have $C=0$. From the proof of Lemma \ref{lemma: k1/M1>-gamma2-}, we have $\kappa_1N_2\leq\widetilde MN_1$, so that $F_0(y)\geq -\kappa_1+\frac{2N_1 y}{N_2}\geq \kappa_1>0$ for $y\geq\widetilde M$. Since $y(w_0)=\widetilde M$, it follows that $y$ is increasing on $[w_0,u_1]$. Combining the three regions yields $y(x)\geq \widetilde M$ for $x\geq w_0$.
    \end{proof}}
	
	It can be shown that $\gamma_{4-}-\gamma_{3-}>0$. Hence, $u_2$ is well defined. We now require that $u_1\leq u_2$. The following lemma gives a necessary and sufficient condition such that $u_1\leq u_2$.
	
	\begin{lemma}\label{lemma on u1<u2}
		$u_1\leq u_2$ if and only if
		\begin{equation*}
			K_{3-}\leq \frac{(1-a)(\gamma_{3+}-\gamma_{4-})}{\gamma_{3-}(\gamma_{3+}-\gamma_{3-})}=:\alpha_{UB}.
		\end{equation*}
	\end{lemma}
	\begin{proof}
		From \eqref{eqn:u1 and u2 w0<u1<u2}, $u_1\leq u_2$ is equivalent to
		\begin{equation*}
			\frac{K_{3-}\gamma_{3-}(\gamma_{4-}-\gamma_{3-})}{K_{3+}\gamma_{3+}(\gamma_{3+}-\gamma_{4-})}\geq 1.
		\end{equation*}
		Using \eqref{eqn:K3+} yields the result.
	\end{proof}
	The following lemma shows that $\alpha_{UB}$ is no larger than the previous upper bound $\overline{\alpha}$.
	\begin{lemma}
		$\alpha_{UB}\leq \overline \alpha$.
	\end{lemma}
	\begin{proof}
		$\alpha_{UB}\leq \overline \alpha$ is equivalent to
		\begin{equation}\label{eqn:r4}
			\sqrt{N_1^2+2\delta N_2}\sqrt{(N_1-\overline c_1-\overline c_2)^2+2\delta N_2} \geq -N_1(N_1-\overline c_1-\overline c_2)-2\delta N_2.
		\end{equation}
		The above inequality is always true if $\overline c_1+\overline c_2<N_1+\frac{2\delta N_2}{N_1}$. Suppose $\overline c_1+\overline c_2\geq N_1+\frac{2\delta N_2}{N_1}$. Squaring both sides of the inequality \eqref{eqn:r4} yields
		\begin{equation*}
			(\overline c_1+\overline c_2)^2\geq 0,
		\end{equation*}
		which is always true. The result is proved.
	\end{proof}
	
	The following lemma gives a necessary and sufficient condition such that $\alpha_{LB}\leq \alpha_{UB}$.
	\begin{lemma}\label{lemma on LB<UB}
		$\alpha_{LB}\leq \alpha_{UB}$ if and only if
		\begin{equation}\label{eqn:c1+c2>}
			\overline c_1+\overline c_2\geq N_1-\frac{\kappa_1 N_2}{2{\widetilde M}}+\frac{\delta {\widetilde M}}{\kappa_1}.
		\end{equation}
	\end{lemma}
	\begin{proof}
		Suppose $\overline c_2\geq  N_1-\frac{\kappa_1 N_2}{2{\widetilde M}}+\frac{\delta {\widetilde M}}{\kappa_1}$. Then, $\alpha_{LB}=\frac{1-a}{\gamma_{3-}}$ and $\alpha_{LB}\leq \alpha_{UB}$ are equivalent to
		\begin{equation*}
			\gamma_{3-}-\gamma_{4-}\leq 0,
		\end{equation*}
		which can be shown to be always true. Since $\overline c_1>0$, then \eqref{eqn:c1+c2>} follows. 
		
		Suppose $\overline c_2<  N_1-\frac{\kappa_1 N_2}{2{\widetilde M}}+\frac{\delta {\widetilde M}}{\kappa_1}$. Then, $\alpha_{LB}=\alpha_0$ and $\alpha_{LB}\leq \alpha_{UB}$ are equivalent to
		\begin{equation*}
			-(N_1-\overline c_1-\overline c_2)+\frac{\kappa_1 N_2}{{\widetilde M}}\geq \sqrt{(N_1-\overline c_1-\overline c_2)^2+2\delta N_2},
		\end{equation*}
		which holds only if $\overline c_1+\overline c_2\geq N_1-\frac{\kappa_1 N_2}{{\widetilde M}}$. Squaring both sides of the above inequality yields
		\begin{equation*}
			\overline c_1+\overline c_2 \geq N_1-\frac{\kappa_1 N_2}{2{\widetilde M}}+\frac{\delta {\widetilde M}}{\kappa_1}>N_1-\frac{\kappa_1 N_2}{{\widetilde M}},
		\end{equation*}
		which proves the result for when $\alpha_{LB}=\alpha_0$. The proof is complete.
	\end{proof}
	
	We now state the following result which gives a necessary and sufficient condition such that $w_0< u_1\leq u_2$.
	\begin{lemma}
		$w_0< u_1\leq u_2$ if and only if $K_{3-}\in(\alpha_{LB},\alpha_{UB}]$.
	\end{lemma}
	\begin{proof}
		This is a direct consequence of Lemmas \ref{lemma on xM1<u1}, \ref{lemma on u1<u2}, and \ref{lemma on LB<UB}.
	\end{proof}
	
	\subsubsection*{Solving for $K_{3-}$}
	
	It remains to prove the existence of $K_{3-}$. We do this via the first equation in \eqref{eqn:system at x=u2}, which ensures that the first derivative at $x=u_2$ is continuous. The following lemma gives a necessary and sufficient condition such that $K_{3-}$ exists and is unique.
	
\begin{lemma}\label{lemma on K3- soln to psi}
Assume that $0<a\leq1/2$, $\alpha_{LB}<\alpha_{UB}$, and
\[
\psi(\alpha_{LB}+)<0.
\]
Then there exists a unique
\[
K_{3-}\in(\alpha_{LB},\alpha_{UB}{]}
\]
such that
\[
\psi(K_{3-})=0.
\]
{Moreover, $K_{3-}=\alpha_{UB}$, or, equivalently, $u_1=u_2$, if and only if $a=\frac{1}{2}$.}
\end{lemma}
	\begin{proof}
		We can rewrite $\psi$ as follows:
		\begin{equation*}
			\begin{aligned}
				\psi(z)
				&=(1-a-\gamma_{3-}z)\left[\frac{\gamma_{3-}(\gamma_{4-}-\gamma_{3-})z}{(1-a-\gamma_{3-}z)(\gamma_{3+}-\gamma_{4-})}\right]^{\frac{\gamma_{3+}}{\gamma_{3+}-\gamma_{3-}}}+\gamma_{3-}ze^{\gamma_{3-}\zeta(z)}-a\\
				&=(1-a-\gamma_{3-}z)^{-\frac{\gamma_{3-}}{\gamma_{3+}-\gamma_{3-}}}\left[\frac{\gamma_{3-}(\gamma_{4-}-\gamma_{3-})z}{\gamma_{3+}-\gamma_{4-}}\right]^{\frac{\gamma_{3+}}{\gamma_{3+}-\gamma_{3-}}}+\gamma_{3-}ze^{\gamma_{3-}\zeta(z)}-a.
			\end{aligned}
		\end{equation*}
		Since $\gamma_{3-}<0$ and $\lim_{z\downarrow \frac{1-a}{\gamma_{3-}}}\zeta(z)=+\infty$, then  $\lim_{z\downarrow \frac{1-a}{\gamma_{3-}}}\psi(z)=-a<0$. From Lemma \ref{lemma on u1<u2}, $\zeta(\alpha_{UB})=0$. Since $a\leq \frac{1}{2}$, $\psi(\alpha_{UB})=1-2a\geq0$, {with equality if and only if $a=\frac{1}{2}$}. We now prove that $\psi$ is strictly increasing on $(\frac{1-a}{\gamma_{3-}},\alpha_{UB})$. From the definitions of $\zeta$ and $\psi$ in \eqref{eqn:psi}, we have
		\begin{equation*}
			\begin{aligned}
				\zeta'(z)&=\frac{1-a}{(\gamma_{3+}-\gamma_{3-})(1-a-\gamma_{3-}z)z},\\
				\psi'(z)&=\gamma_{3-}\left(e^{\gamma_{3-}\zeta(z)}-e^{\gamma_{3+}\zeta(z)}\right)+\frac{1-a}{\gamma_{3+}-\gamma_{3-}}\left[\frac{\gamma_{3+}}{z}e^{\gamma_{3+}\zeta(z)}+\frac{\gamma_{3-}^2}{1-a-\gamma_{3-}z}e^{\gamma_{3-}\zeta(z)}\right].    
			\end{aligned}
		\end{equation*}
		Since $\gamma_{3+}>\gamma_{3-}$, it holds that
		\begin{equation*}
			(\gamma_{3+}-\gamma_{3-})\zeta(z)=\ln\left(\frac{\gamma_{3-}(\gamma_{4-}-\gamma_{3-})z}{(1-a-\gamma_{3-}z)(\gamma_{3+}-\gamma_{4-})}\right)<\ln\left(\frac{-\gamma_{3-}^2z}{\gamma_{3+}(1-a-\gamma_{3-}z)}\right).
		\end{equation*}
		It follows that
		\begin{equation*}
			\frac{\gamma_{3+}}{z}e^{\gamma_{3+}\zeta(z)}>\frac{-\gamma_{3-}^2}{1-a-\gamma_{3-}z}e^{\gamma_{3-}\zeta(z)}.
		\end{equation*}
		Now using $\zeta(z)\geq 0$ for $z\in\left(\frac{1-a}{\gamma_{3-}},\alpha_{UB}\right)$, we obtain  $\psi'(z)>0$ on $\left(\frac{1-a}{\gamma_{3-}},\alpha_{UB}\right)$. Hence, $\psi$ is strictly increasing on this interval.
		
Since $\psi(\alpha_{LB}+)<0$ and
$$\psi(\alpha_{UB})=1-2a\geq 0,$$
the intermediate value theorem gives a root in
$(\alpha_{LB},\alpha_{UB}]$. The strict monotonicity of $\psi$ on $(\frac{1-a}{\gamma_{3-}},\alpha_{UB})$ gives uniqueness. The proof is complete.
	\end{proof}
	\begin{remark}
        If $\overline c_2\geq N_1-\frac{\kappa_1 N_2}{2{\widetilde M}}+\frac{\delta {\widetilde M}}{\kappa_1}$, then $\alpha_{LB}=\frac{1-a}{\gamma_{3-}}$, so that $\psi(\alpha_{LB}+)=-a<0$ by the proof of Lemma \ref{lemma on K3- soln to psi}, and $\alpha_{LB}<\alpha_{UB}$ since $\gamma_{3-}<\gamma_{4-}$ (see the proof of Lemma \ref{lemma on LB<UB}). Hence, the existence (and uniqueness) of $K_{3-}$ is guaranteed by Lemma \ref{lemma on K3- soln to psi}. 
	\end{remark}
	
	\subsection{Proof of Theorem \ref{thm:bdd u1<w0<u2}}
	\subsubsection*{Deriving the analytical solution}
	Suppose for now that $u_1\leq w_0\leq u_2$ and $w_0$ exists. In the region $\{x<u_1\}$, the HJB equation becomes \eqref{eqn:hjb x<w0} and its solution is given by $	g_1(x)=K_1\int_0^x \exp\left[\int^{u_1}_z\frac{\kappa_1}{G^{-1}(y)}dy\right]dz$, where $K_1>0$ is a constant and $G^{-1}$ is the inverse of $G$ given by \eqref{eqn:G}. We also obtain ${\widehat\pi}_1(x)=G^{-1}(x)$.
	
	In the region $\{u_1<x<w_0\}$, we {replace $\Pi_i(x)$ by $\widehat\pi_i(x)$ as in the proof of Theorem \ref{thm:bdd w0<u1<u2}}, and the HJB equation becomes
	\begin{equation}\label{eqn:hjb two}
		0=\left[\kappa_1\mu_1( {\widehat\pi}_1(x))+\kappa_2\mu_2\left(\frac{\kappa_2}{\kappa_1}{\widehat\pi}_1(x)\right)-\overline c_2\right]g'(x)+\left[\frac{1}{2}\sigma_1^2( {\widehat\pi}_1(x))+\frac{1}{2}\sigma_2^2\left(\frac{\kappa_2}{\kappa_1} {\widehat\pi}_1(x)\right)\right]g''(x)-\delta g(x)+(1-a)\overline c_2.
	\end{equation}
	Using \eqref{eqn:optimal pi}, we obtain ${\widehat\pi}_1(x)=H(x)$, where $H$ satisfies the differential equation in \eqref{eqn:H}. Hence, $w_0$ exists and satisfies $H(w_0)={\widetilde M}$ via \eqref{eqn:defn of w0}. Similar to solving \eqref{eqn:hjb x<w0}, we obtain the following solution for \eqref{eqn:hjb two}:
	\begin{equation*}
		g_2(x)=K_{2+}\int_{u_1}^x \exp\left[\int^{w_0}_z\frac{\kappa_1}{H(y)}dy\right]dz+K_{2-},
	\end{equation*}
	where $K_{2\pm}$ are unknown constants.
	
	In the region $\{w_0<x<u_2\}$, the HJB equation becomes \eqref{eqn:hjb 3}, whose solution is of the form given in \eqref{eqn:g_3}. In the region $\{x>u_2\}$, the HJB equation simplifies to \eqref{eqn:hjb 4}, which has a solution described by \eqref{eqn:g_4}.
	We conjecture the following solution:
	\begin{equation}\label{eqn:conjecture g two}
		g(x)=
		\begin{cases}
			K_1\int_0^x \exp\left[\int^{u_1}_z\frac{\kappa_1}{G^{-1}(y)}dy\right]dz&\mbox{if $x<u_1$,}\\
			K_{2+}\int_{u_1}^x \exp\left[\int^{w_0}_z\frac{\kappa_1}{H(y)}dy\right]dz+K_{2-}&\mbox{if $u_1<x<w_0$,}\\
			K_{3+}e^{\gamma_{3+}(x-{u_2})}+K_{3-}e^{\gamma_{3-}(x-{u_2})}+\frac{(1-a)\overline c_2}{\delta}&\mbox{if $w_0<x<u_2$,}\\
			K_{4-}e^{\gamma_{4-}(x-u_2)}+\frac{a\overline c_1+(1-a)\overline c_2}{\delta}&\mbox{if $x>u_2$,}
		\end{cases}
	\end{equation}
	where $w_0=H^{-1}({\widetilde M})$ and $K_1,\, K_{2\pm},\, K_{3\pm}, \, K_{4-}, \, u_1,\, u_2$ are yet to be determined. {Since $u_1$ does not belong to the boundary of the region $\{w_0<x<u_2\}$, we center the exponential functions in this region at $u_2$, where the boundary conditions $g'(u_2)=a$ and $g''(u_2)=a\gamma_{4-}$ are imposed.}
	
	To ensure twice continuous differentiability at $x=u_1$, we obtain the following equations:
	\begin{equation}\label{eqn:system at u1 case 2}
		\begin{aligned}
			K_1\int_0^{u_1} \exp\left[\int^{u_1}_z\frac{\kappa_1}{G^{-1}(y)}dy\right]dz&=K_{2-},\\
			K_1=K_{2+} \exp\left[\int^{w_0}_{u_1}\frac{\kappa_1}{H(y)}dy\right]&=1-a,
		\end{aligned}
	\end{equation}
	which imply the following:
	\begin{equation}\label{eqn:equiv to system at u1}
		\begin{aligned}
			K_1&=1-a,\\
			K_{2-}&=(1-a)\int_0^{u_1} \exp\left[\int^{u_1}_z\frac{\kappa_1}{G^{-1}(y)}dy\right]dz,\\
			K_{2+}&=(1-a) \exp\left[-\int^{w_0}_{u_1}\frac{\kappa_1}{H(y)}dy\right],
		\end{aligned}
	\end{equation}
	where $u_1$ is still unknown.

    {To ensure twice continuous differentiability at $x=u_2$, we require $g'(u_2)=a$ and $g''(u_2)=a\gamma_{4-}$, which yield
    \begin{equation*}
        K_{3+}=\frac{a(\gamma_{4-}-\gamma_{3-})}{\gamma_{3+}(\gamma_{3+}-\gamma_{3-})}\quad\mbox{and}\quad K_{3-}=\frac{a(\gamma_{3+}-\gamma_{4-})}{\gamma_{3-}(\gamma_{3+}-\gamma_{3-})}.
    \end{equation*}
    Note that the continuity of $g$ at $x=u_2$ then follows from the HJB equations \eqref{eqn:hjb 3} and \eqref{eqn:hjb 4}, since subtracting \eqref{eqn:hjb 4} from \eqref{eqn:hjb 3} at $x=u_2$ gives $\delta\left[g_3(u_2)-g_4(u_2)\right]=\overline c_1(g'(u_2)-a)=0$.} To ensure twice continuous differentiability at $x=w_0$, we have the following equations:
	\begin{equation*}
		\begin{aligned}
			K_{2+}&=K_{3+}\gamma_{3+}e^{\gamma_{3+}(w_0-{u_2})}+K_{3-}\gamma_{3-}e^{\gamma_{3-}(w_0-{u_2})},\\
			-\frac{\kappa_1K_{2+}}{{\widetilde M}}&=K_{3+}\gamma_{3+}^2e^{\gamma_{3+}(w_0-{u_2})}+K_{3-}\gamma_{3-}^2e^{\gamma_{3-}(w_0-{u_2})}.
		\end{aligned}
	\end{equation*}
	Dividing the second equation by the first equation yields the formula for ${u_2}$ in \eqref{eqn:u1 and u2 u1<w0<u2}. {The first equation yields $K_{2+}$ given in Theorem \ref{thm:bdd u1<w0<u2}, and combining it with \eqref{eqn:equiv to system at u1} yields \eqref{eqn:u1 u1<w0<u2}.
    
    As in the proof of Lemma \ref{lemma: both lines retain}, both lines retain all of their risk for $x\geq w_0$. Indeed, $y(x)=\widehat \pi_1(x)$ satisfies \eqref{eqn:ode for y} with $C=\overline c_2$ on $[w_0,u_2]$ and $y(w_0)=\widetilde M$, and $y(x)=-\frac{\kappa_1}{\gamma_{4-}}\geq \widetilde M$ for $x\geq u_2$ by condition (ii). Since $y'(u_2-)=\frac{2\overline c_1 y(u_2)}{N_2}>0$, $y$ is increasing on $[w_0,u_2]$, and hence $y(x)\geq\widetilde M$ for $x\geq w_0$.}

	\subsubsection*{Establishing the {existence of $u_1$ and $u_2$}}
	
	{We first} establish the conditions such that the formula for ${u_2}$ in \eqref{eqn:u1 and u2 u1<w0<u2} is well defined. We require the following result:
	\begin{lemma}\label{lemma on gamma3- and k1/M1}
		$\gamma_{3-}+\frac{\kappa_1}{{\widetilde M}}<0$ if and only if
		\begin{equation*}
			\overline c_2 < N_1- \frac{\kappa_1 N_2}{2{\widetilde M}}+\frac{\delta {\widetilde M}}{\kappa_1}.
		\end{equation*}
	\end{lemma}
	\begin{proof}
		$\gamma_{3-}+\frac{\kappa_1}{{\widetilde M}}<0$ is equivalent to
		\begin{equation*}
			\sqrt{(N_1-\overline c_2)^2+2\delta N_2}
>
\frac{\kappa_1 N_2}{{\widetilde M}}-(N_1-\overline c_2).
		\end{equation*}
		The inequality holds immediately if $0<\overline c_2<N_1-\frac{\kappa_1 N_2}{{\widetilde M}}$. Suppose that $\overline c_2\geq N_1-\frac{\kappa_1 N_2}{{\widetilde M}}$. Squaring both sides of the above inequality yields the desired result.
	\end{proof}
    {By condition (iii), $\gamma_{3-}+\frac{\kappa_1}{\widetilde M}<0$, so that $\overline c_2 < N_1- \frac{\kappa_1 N_2}{2\widetilde M}+\frac{\delta \widetilde M}{\kappa_1}$ by Lemma \ref{lemma on gamma3- and k1/M1}.}
    
	For $u_2$ defined in \eqref{eqn:u1 and u2 u1<w0<u2}, it is well defined by Lemma \ref{lemma on gamma3- and k1/M1}. The following lemma gives a necessary and sufficient condition such that $u_2\geq w_0$.
	\begin{lemma}\label{lemma on u2>w0}
		$u_2\geq w_0$ if and only if
		\begin{equation*}
			\overline c_1+\overline c_2\geq N_1-\frac{\kappa_1 N_2}{2{\widetilde M}}+\frac{\delta {\widetilde M}}{\kappa_1}.
		\end{equation*}
	\end{lemma}
	\begin{proof}
		$u_2\geq w_0$ is equivalent to
		\begin{equation*}
			\frac{(\gamma_{4-}-\gamma_{3-})\left(-\frac{\kappa_1}{{\widetilde M}}-\gamma_{3+}\right)}{(\gamma_{3+}-\gamma_{4-})\left(\gamma_{3-}+\frac{\kappa_1}{{\widetilde M}}\right)}\geq 1.
		\end{equation*}
		The above inequality is equivalent to $\gamma_{4-}+\frac{\kappa_1}{{\widetilde M}}\geq 0$, which can be rewritten as 
		\begin{equation*}
			-(N_1-\overline c_1-\overline c_2)+\frac{\kappa_1 N_2}{{\widetilde M}}\geq \sqrt{(N_1-\overline c_1-\overline c_2)^2+2\delta N_2}.
		\end{equation*}
		A necessary condition for the above inequality to hold is $\overline c_1+\overline c_2\geq N_1-\frac{\kappa_1 N_2}{{\widetilde M}}$. Suppose $\overline c_1+\overline c_2\geq N_1-\frac{\kappa_1 N_2}{{\widetilde M}}$. Squaring both sides of the above inequality yields the desired result.
	\end{proof}
	
	{\subsubsection*{Determining $u_1$}
    We now determine $u_1$. For $u_1\leq x\leq w_0$, $y(x):=\widehat \pi_1(x)=H(x)$ satisfies \eqref{eqn:H}; that is, $y'(x)=F(y(x))$, where
    \begin{equation*}
        F(y):=-\kappa_1+\frac{2\left(\overline N_1(y)-\overline c_2\right)y+\frac{2\delta}{\kappa_1}y^2}{\overline N_2(y)}.
    \end{equation*}
    Since $\overline N_1$ and $\overline N_2$ are constant on $[\widetilde M,\infty)$, we have $F(y)=\frac{2\delta}{\kappa_1N_2}\left(y+\frac{\kappa_1}{\gamma_{3+}}\right)\left(y+\frac{\kappa_1}{\gamma_{3-}}\right)$ for $y\geq\widetilde M$, so that $F(\widetilde M)>0$ since $\gamma_{3-}+\frac{\kappa_1}{\widetilde M}<0$. Moreover, $F(y)\to-\infty$ as $y\downarrow 0$. Let $y_0\in(0,\widetilde M)$ be the largest zero of $F$. Changing the variable from $x$ to $y$, \eqref{eqn:u1 u1<w0<u2} becomes
    \begin{equation*}
        \Lambda(H(u_1))=\ln\left(\frac{1-a}{K_{2+}}\right),\qquad \Lambda(z):=\int_z^{\widetilde M}\frac{\kappa_1}{yF(y)}\,dy,
    \end{equation*}
    and $w_0-u_1=\int_{H(u_1)}^{\widetilde M}\frac{dy}{F(y)}$. The function $\Lambda$ is continuous and strictly decreasing on $(y_0,\widetilde M]$, $\Lambda(\widetilde M)=0$, and $\Lambda(z)\to\infty$ as $z\downarrow y_0$. Hence, $H(u_1)\in(y_0,\widetilde M]$, and thus $u_1=G(H(u_1))$, is uniquely determined if and only if $K_{2+}\leq 1-a$, and $H(u_1)=\widetilde M$; that is, $u_1=w_0$ if and only if $K_{2+}=1-a$. Since $K_{2+}\leq 1-a$ by condition (iii), this proves that $u_1\leq w_0\leq u_2$, with equality $u_1=w_0$ if and only if $K_{2+}=1-a$.

    \subsubsection*{Relation between Theorems \ref{thm:bdd w0<u1<u2} and \ref{thm:bdd u1<w0<u2}}

    We now show that, under conditions (i) and (ii), exactly one of Theorems \ref{thm:bdd w0<u1<u2} and \ref{thm:bdd u1<w0<u2} applies. The two theorems cannot apply simultaneously, since the value function is unique by Theorem \ref{thm:verification}, and the order of $u_1$ and $w_0$ is determined by $g$ through \eqref{eqn:defn of u1 & u2} and \eqref{eqn:defn of w0}. Suppose now that condition (iii) of Theorem \ref{thm:bdd u1<w0<u2} does not hold. Consider the function $g_3$ on $\{x\leq u_2\}$ with $g_3'(u_2)=a$ and $g_3''(u_2)=a\gamma_{4-}$, given by the third line of \eqref{eqn:g u1<w0<u2}. Since its derivative is a linear combination of $e^{\gamma_{3+}(x-u_2)}$ and $e^{\gamma_{3-}(x-u_2)}$ with positive coefficients, it increases to infinity as $x$ decreases, so that $g_3'(\overline x)=1-a$ for a unique $\overline x<u_2$. Moreover, $y(x):=-\kappa_1\frac{g_3'(x)}{g_3''(x)}$ satisfies \eqref{eqn:ode for y} with $C=\overline c_2$ and is increasing on $(-\infty,u_2]$. If $\gamma_{3-}+\frac{\kappa_1}{\widetilde M}\geq 0$, then $y(x)\geq -\frac{\kappa_1}{\gamma_{3-}}\geq\widetilde M$, since $-\frac{\kappa_1}{\gamma_{3-}}$ is a zero of $F_{\overline c_2}$. If $K_{2+}>1-a$, then $y(x)>\widetilde M$ for $x\geq \overline x$, since $g_3'$ is decreasing and $g_3'(x)=K_{2+}> g_3'(\overline x)$ at the point $x$ at which $y(x)=\widetilde M$. In both cases, both lines retain all of their risk on $[\overline x, u_2]$, and $g_3$ coincides on this interval with the third line of \eqref{eqn:g w0<u1<u2} for $u_1=\overline x$ and some constant $K_{3-}$. By Lemma \ref{lemma on xM1<u1} and Lemma \ref{lemma on u1<u2}, this constant lies in $(\alpha_{LB},\alpha_{UB}]$, and it solves $\psi(z)=0$ by the first equation in \eqref{eqn:system at x=u2}. Since $\psi$ is strictly increasing, it follows that $\psi(\alpha_{LB}+)<0$; that is, condition (iii) of Theorem \ref{thm:bdd w0<u1<u2} holds. Conversely, if condition (iii) of Theorem \ref{thm:bdd w0<u1<u2} holds, then $w_0<u_1$ by Theorem \ref{thm:bdd w0<u1<u2}, so condition (iii) of Theorem \ref{thm:bdd u1<w0<u2} cannot hold, since it would yield $u_1\leq w_0$. Hence, the two conditions (iii) are negations of each other. Note that the boundary case $K_{2+}=1-a$ (with $\gamma_{3-}+\frac{\kappa_1}{\widetilde M}<0$), in which $u_1=w_0$, is covered by Theorem \ref{thm:bdd u1<w0<u2}.
    }

    \subsection{Proof of Theorem \ref{thm: bdd u1<u2<w0 or M infty}}
	
	\subsubsection*{Proving that $w_0$ is infinite when ${\widetilde{M}}<\infty$}
	
	Suppose $\overline c_1+\overline c_2<N_1-\frac{\kappa_1 N_2}{2{\widetilde M}}+\frac{\delta {\widetilde M}}{\kappa_1}$. Suppose further that $w_0$ exists and is finite. In the region $\{x>w_0\}$, we get \eqref{eqn:hjb 4}, whose solution is given by \eqref{eqn:g_4}. From \eqref{eqn:defn of w0} and the assumption $\overline c_1+\overline c_2<N_1-\frac{\kappa_1 N_2}{2{\widetilde M}}+\frac{\delta {\widetilde M}}{\kappa_1}$, we get
	\begin{equation*}
		0=-\kappa_1\frac{g_4'(w_0)}{g_4''(w_0)}-{\widetilde M}=-\frac{\kappa_1}{\gamma_{4-}}-{\widetilde M}<0,
	\end{equation*}
	which is a contradiction. Hence, no such $w_0$ exists and we write $w_0=\infty$.

    \subsubsection*{Deriving the analytical solution for ${\widetilde M}<\infty$}

	Suppose for now that $u_1\leq u_2<w_0=\infty$. In the region $\{x<u_1\}$, the HJB equation becomes \eqref{eqn:hjb x<w0} whose solution that satisfies $g(0)=0$ is given by $g_1(x)=K_1\int_0^x \exp\left[\int^{u_1}_z\frac{\kappa_1}{G^{-1}(y)}dy\right]dz$, where $K_1>0$ is an unknown constant and $G^{-1}$ is the inverse of the function $G$ given by \eqref{eqn:G}. In the region $\{u_1<x<u_2\}$, the HJB equation becomes \eqref{eqn:hjb two} whose solution is given by $g_2(x)=K_{2+}\int_{u_1}^x \exp\left[\int^{u_2}_z\frac{\kappa_1}{H(y)}dy\right]dz+K_{2-}$, where $K_{2\pm}$ are unknown constants and $H$ satisfies \eqref{eqn:H}.

	Since $w_0=\infty$, we conjecture that ${\widehat\pi}_1(x)=M_0$ for all $x\geq u_2$ and some unknown constant $M_0\in(0,{\widetilde M})$. In the region $\{x>u_2\}$, the HJB equation then becomes
	\begin{equation}
		0=\left[\overline N_1(M_0)-\overline c_1-\overline c_2\right]g'(x)+\frac{1}{2}\overline N_2(M_0)g''(x)-\delta g(x)+a\overline c_1+(1-a)\overline c_2,
	\end{equation}
	whose solution that satisfies $\lim_{x\to\infty}g(x)=\frac{a\overline c_1+(1-a)\overline c_2}{\delta}$ is given by
	\begin{equation*}
		g_3(x)=K_3e^{\overline\gamma_{4-}(M_0)x}+\frac{a\overline c_1+(1-a)\overline c_2}{\delta},
	\end{equation*}
	where $\overline\gamma_{4-}(y)$ is defined in \eqref{eqn:notations}. We conjecture the following solution:
	\begin{equation*}
		g(x)=
		\begin{cases}
			K_1\int_0^x \exp\left[\int^{u_1}_z\frac{\kappa_1}{G^{-1}(y)}dy\right]dz&\mbox{if $x<u_1$,}\\
			K_{2+}\int_{u_1}^x \exp\left[\int^{u_2}_z\frac{\kappa_1}{H(y)}dy\right]dz+K_{2-}&\mbox{if $u_1<x<u_2$,}\\
			K_3e^{\overline \gamma_{4-}(M_0)x}+\frac{a\overline c_1+(1-a)\overline c_2}{\delta}&\mbox{if $x>u_2$,}
		\end{cases}
	\end{equation*}
	where $K_1,\, K_{2\pm},\, K_{3}, \, M_0, \, u_1,\, u_2$ are yet to be determined.	
	
	To ensure twice continuous differentiability at $x=u_1$, we obtain the equations in \eqref{eqn:system at u1 case 2}, which yield the formulas for $K_1$ and $K_{2-}$ in \eqref{eqn:equiv to system at u1}. Moreover, we have
	\begin{equation*}
		K_{2+}=(1-a) \exp\left[-\int^{u_2}_{u_1}\frac{\kappa_1}{H(y)}dy\right].
	\end{equation*}
	
	To ensure twice continuous differentiability at $x=u_2$, we obtain the following equations:
	\begin{equation}\label{eqn:system at u2 case 3}
		\begin{aligned}
			(1-a)\exp\left[\int_{u_2}^{u_1}\frac{\kappa_1}{H(y)}dy\right]=K_3\overline \gamma_{4-}(M_0)e^{\overline \gamma_{4-}(M_0)u_2}&=a,\\
			-\frac{(1-a)\kappa_1}{M_0}\exp\left[\int_{u_2}^{u_1}\frac{\kappa_1}{H(y)}dy\right]=K_3\overline\gamma_{4-}^2(M_0)e^{\overline\gamma_{4-}(M_0)u_2}.
		\end{aligned}
	\end{equation}
	From the first equation, we obtain
	\begin{equation*}
		K_3=\frac{a}{\overline \gamma_{4-}(M_0)}e^{-\overline \gamma_{4-}(M_0)u_2}.
	\end{equation*}	
	Dividing the second equation in \eqref{eqn:system at u2 case 3} by the first equation yields
	\begin{equation}\label{eqn:gamma3(M0)}
		-\frac{\kappa_1}{M_0}=\overline \gamma_{4-}(M_0).
	\end{equation}
	We still have to prove that a unique $M_0$ exists.
	\begin{lemma}\label{lemma:M0 unique}
		Suppose $\overline c_1+\overline c_2<N_1-\frac{\kappa_1 N_2}{2{\widetilde M}}+\frac{\delta {\widetilde M}}{\kappa_1}$. Then there exists a unique solution $M_0\in(0,\widetilde M)$ of
		\begin{equation*}
			-\frac{\kappa_1}{y}=\overline \gamma_{4-}(y), \quad y\in(0,{\widetilde M}).
		\end{equation*}
	\end{lemma}
	\begin{proof}
		Since $\overline N_1(0)=\overline N_2(0)=0$, we have
		\begin{equation*}
			\overline \gamma_{4-}(0):=\lim_{y\to 0} \overline \gamma_{4-}(y)=\lim_{y\to 0} \frac{-2\delta}{-(\overline N_1(y)-\overline c_1-\overline c_2)+\sqrt{(\overline N_1(y)-\overline c_1-\overline c_2)^2+2\delta \overline N_2(y)}}=\frac{-\delta}{\overline c_1+\overline c_2}.
		\end{equation*}
		Define $f(y):=\overline \gamma_{4-}(y)+\frac{\kappa_1}{y}$. We have the following:
		\begin{equation*}
			\lim_{y\downarrow 0} f(y)=\infty\quad\mbox{and}\quad f({\widetilde M})=\gamma_{4-}+\frac{\kappa_1}{{\widetilde M}}.
		\end{equation*}
		By the intermediate value theorem, $f$ has a solution in $(0,{\widetilde M})$ if $\gamma_{4-}+\frac{\kappa_1}{{\widetilde M}}<0$. From the proof of Lemma \ref{lemma on u2>w0}, $\gamma_{4-}+\frac{\kappa_1}{{\widetilde M}}<0$ is equivalent to $\overline c_1+\overline c_2<N_1-\frac{\kappa_1 N_2}{2{\widetilde M}}+\frac{\delta {\widetilde M}}{\kappa_1}$, which proves the existence.

    It remains to prove uniqueness. Since $\overline\gamma_{4-}(y)$ is the negative
root of
\[
\frac12\overline N_2(y)\gamma^2+
\left(\overline N_1(y)-\overline c_1-\overline c_2\right)\gamma
-\delta=0,
\]
the equation $-\kappa_1/y=\overline\gamma_{4-}(y)$ is equivalent to
\[
\overline c_1+\overline c_2=\Phi(y),
\]
where
\[
\Phi(y):=
\overline N_1(y)
-\frac{\kappa_1\overline N_2(y)}{2y}
+\frac{\delta y}{\kappa_1}.
\]
We show that $\Phi$ is strictly increasing on $(0,\widetilde M)$. Wherever the
derivatives exist,
\[
\overline N_1'(y)
=
\kappa_1\overline F_1(y)
+
\frac{\kappa_2^2}{\kappa_1}
\overline F_2\left(\frac{\kappa_2}{\kappa_1}y\right),
\]
and
\[
\overline N_2'(y)
=
2y\overline F_1(y)
+
2\frac{\kappa_2^2}{\kappa_1^2}y
\overline F_2\left(\frac{\kappa_2}{\kappa_1}y\right).
\]
Hence
\[
\overline N_1'(y)=\frac{\kappa_1}{2y}\overline N_2'(y).
\]
Therefore,
\[
\begin{aligned}
\Phi'(y)
&=
\overline N_1'(y)
-\frac{\kappa_1}{2}
\left(
\frac{\overline N_2'(y)}{y}
-\frac{\overline N_2(y)}{y^2}
\right)
+\frac{\delta}{\kappa_1} \\
&=
\frac{\kappa_1\overline N_2(y)}{2y^2}
+\frac{\delta}{\kappa_1}
>0.
\end{aligned}
\]
Since $\Phi$ is continuous and $\Phi'(y)>0$ for all but countably many $y$, $\Phi$ is strictly increasing on $(0,\infty)$, and in particular on $(0,\widetilde M)$. Consequently,
$\Phi(y)=\overline c_1+\overline c_2$ has at most one solution. Therefore, the solution obtained above is unique.
    \end{proof}
	Similarly to the previous arguments, we have ${\widehat\pi}_1(x)=H(x)$. Hence, it follows that $u_2$ satisfies
	\begin{equation*}
		H(u_2)=M_0.
	\end{equation*}
    Finally, from the first equation in \eqref{eqn:system at u2 case 3}, $u_1$ satisfies \eqref{eqn:formula for u1 case 3}. The following lemma proves that $u_1$ is the unique solution to the above equation.
	\begin{lemma}
		$u_1$ is the unique solution to \eqref{eqn:formula for u1 case 3} in $(0,u_2]$. 
	\end{lemma}
	\begin{proof}
		Define $f(x):=-\int_{x}^{u_2}\frac{\kappa_1}{H(y)}dy-\ln\left(\frac{a}{1-a}\right)$.  Since $a\leq \frac{1}{2}$,  we have $f(u_2)=-\ln\left(\frac{a}{1-a}\right)\geq 0$. Since $H(0)=0$, we have
		\begin{equation*}
			\lim_{x\to 0}f(x)=-\infty.
		\end{equation*}
		Since $f'(x)=\frac{\kappa_1}{H(x)}>0$, $f$ is strictly increasing. Thus, by the intermediate value theorem, there exists a unique $x_0\in(0,u_2]$ such that $f(x_0)=0$. Choosing $x_0=u_1$ proves the result.
	\end{proof}
	
	\subsubsection*{Proving that $w_0$ is infinite if ${\widetilde M}=\infty$}
	
	We now consider the case where {at least one claim-size distribution has unbounded support} (i.e., ${\widetilde M}=\infty$). We first prove that the configurations $w_0\leq u_1\leq u_2$ and $u_1< w_0\leq u_2$ do not hold. Suppose the configuration $w_0\leq u_1\leq u_2$ holds and $w_0$ exists. Similar to the bounded support case (i.e., ${\widetilde M}<\infty$), we obtain the form of the candidate solution in \eqref{eqn:conjecture g one}. The argument is the same as in the bounded-support case. However, from Lemma \ref{lemma on LB<UB}, we must have $\overline c_1+\overline c_2\geq \infty$ for $\alpha_{LB}\leq \alpha_{UB}$ to hold. This implies that $K_{3-}$ does not exist. Thus, the configuration $w_0\leq u_1\leq u_2$ is not possible.
	
	Suppose the configuration $u_1< w_0\leq u_2$ holds and $w_0$ exists. Similar to the bounded support case, we also obtain the form of the candidate solution in \eqref{eqn:conjecture g two}, which leads to Lemma \ref{lemma on u2>w0}. The configuration will hold if and only if $\overline c_1+\overline c_2 \geq \infty$, which implies that $u_1< w_0\leq u_2$ does not hold.

	\subsubsection*{Deriving the analytical solution when ${\widetilde M}=\infty$}
	Suppose now that the configuration $u_1\leq u_2<w_0$ holds. The discussion is similar to that in the bounded-support case. We obtain the equation \eqref{eqn:gamma3(M0)}. The following lemma states that $M_0$ indeed solves \eqref{eqn:gamma3(M0)}:
	\begin{lemma}
		$M_0$ is a unique solution of
		\begin{equation*}
			-\frac{\kappa_1}{y}=\overline \gamma_{4-}(y), \quad y\in(0,{\infty}).
		\end{equation*}
	\end{lemma}
	\begin{proof}
		Since $\overline N_1(0)=\overline N_2(0)=0$, then we have
		\begin{equation*}
			\overline \gamma_{4-}(0):=\lim_{y\to 0} \overline \gamma_{4-}(y)=\frac{-\delta}{\overline c_1+\overline c_2}.
		\end{equation*}
		Define $f(y):=\overline \gamma_{4-}(y)+\frac{\kappa_1}{y}$. We have the following:
		\begin{equation*}
			\lim_{y\downarrow 0} f(y)=\infty\quad\mbox{and}\quad {\lim_{y\to\infty}f(y)=\lim_{y\to\infty} \overline\gamma_{4-}(y)}<0.
		\end{equation*}
		By the intermediate value theorem, $f$ has a solution in $(0,{\infty})$. Uniqueness follows from the proof of Lemma \ref{lemma:M0 unique}, where it is shown that $\Phi$ is strictly increasing on $(0,\infty)$.

\begin{comment}
It remains to prove uniqueness. Since $\overline\gamma_{4-}(y)$ is the negative
root of
\[
\frac12\overline N_2(y)\gamma^2+
\left(\overline N_1(y)-\overline c_1-\overline c_2\right)\gamma
-\delta=0,
\]
the equation $-\kappa_1/y=\overline\gamma_{4-}(y)$ is equivalent to
\[
\overline c_1+\overline c_2
=
\Phi(y),
\]
where
\[
\Phi(y):=
\overline N_1(y)
-\frac{\kappa_1\overline N_2(y)}{2y}
+\frac{\delta y}{\kappa_1}.
\]
We show that $\Phi$ is strictly increasing. Wherever the derivatives exist,
\[
\overline N_1'(y)
=
\kappa_1\overline F_1(y)
+
\frac{\kappa_2^2}{\kappa_1}
\overline F_2\left(\frac{\kappa_2}{\kappa_1}y\right),
\]
and
\[
\overline N_2'(y)
=
2y\overline F_1(y)
+
2\frac{\kappa_2^2}{\kappa_1^2}y
\overline F_2\left(\frac{\kappa_2}{\kappa_1}y\right).
\]
Hence
\[
\overline N_1'(y)=\frac{\kappa_1}{2y}\overline N_2'(y).
\]
Therefore,
\[
\begin{aligned}
\Phi'(y)
&=
\overline N_1'(y)
-\frac{\kappa_1}{2}
\left(
\frac{\overline N_2'(y)}{y}
-\frac{\overline N_2(y)}{y^2}
\right)
+\frac{\delta}{\kappa_1}\\
&=
\frac{\kappa_1\overline N_2(y)}{2y^2}
+\frac{\delta}{\kappa_1}
>0.
\end{aligned}
\]
Thus, the solution obtained above is unique.
\end{comment}
	\end{proof}
	
	\subsection{Proof of Theorem \ref{thm:ubdd M1 finite}}
	
	Suppose ${\widetilde M}<\infty$. The optimal reinsurance strategies {are still given by \eqref{eqn:optimal pi constrained}}. In the region $\{x<{w_0}\}$, the HJB equation becomes \eqref{eqn:hjb x<w0}, whose solution is given by $g_1(x)=K_1\int_0^x \exp\left[\int^{w_0}_z\frac{\kappa_1}{G^{-1}(y)}dy\right]dz$, where $K_1>0$ is an unknown constant and $G^{-1}$ is the inverse of the function $G$ given by \eqref{eqn:G}. We also obtain ${\widehat\pi}_1(x)=G^{-1}(x)$. Since ${\widehat\pi}_1'(x)>0$ and ${\widehat\pi}_1(0)=0$, $w_0$ exists and satisfies $w_0=G({\widetilde M})$.
	
	In the region $\{w_0<x<u_1\}$, {both lines retain all of their risk by the same argument as in the proof of Lemma \ref{lemma: both lines retain}, and} the HJB equation becomes \eqref{eqn:hjb w0<x<u1}, whose solution is given by \eqref{eqn:g_2}. In the region $\{x>u_1\}$, we must have $g'(x)=1-a$. Hence, 
	\begin{equation*}
		g_3(x)=(1-a)\left[x-u_1+K_3\right].
	\end{equation*}
	We conjecture the following solution:
	\begin{equation*}
		g(x)=
		\begin{cases}
			K_1\int_0^x \exp\left[\int^{w_0}_z\frac{\kappa_1}{G^{-1}(y)}dy\right]dz&\mbox{if $x<w_0$,}\\
			K_{2+}e^{\gamma_{2+}(x-w_0)}+K_{2-}e^{\gamma_{2-}(x-w_0)}&\mbox{if $w_0<x<u_1$,}\\
			(1-a)\left[x-u_1+K_3\right]&\mbox{if $x>u_1$,}
		\end{cases}
	\end{equation*}
	where $w_0=G({\widetilde M})$, $\gamma_{2\pm}:=\overline \gamma_{2\pm}({\widetilde M})$ are defined in \eqref{eqn:notations}, and $K_1,\, K_{2\pm},\, K_3, \, u_1$ are yet to be determined.
	
	We now solve for the unknowns using the principle of smooth fit. At $x=u_1$, we have the following system of equations:
	\begin{equation*}
		\begin{aligned}
			K_{2+}e^{\gamma_{2+}(u_1-w_0)}+K_{2-}e^{\gamma_{2-}(u_1-w_0)}&=(1-a)K_3,\\
			K_{2+}\gamma_{2+}e^{\gamma_{2+}(u_1-w_0)}+K_{2-}\gamma_{2-}e^{\gamma_{2-}(u_1-w_0)}&=1-a,\\
			K_{2+}\gamma^2_{2+}e^{\gamma_{2+}(u_1-w_0)}+K_{2-}\gamma^2_{2-}e^{\gamma_{2-}(u_1-w_0)}&=0.
		\end{aligned}
	\end{equation*}
	From the second and third equations, we obtain
	\begin{equation*}
		K_{2+}
=
-\frac{\gamma_{2-}(1-a)e^{-\gamma_{2+}(u_1-w_0)}}
{\gamma_{2+}(\gamma_{2+}-\gamma_{2-})}
>0\quad\mbox{and}\quad K_{2-}
=
\frac{\gamma_{2+}(1-a)e^{-\gamma_{2-}(u_1-w_0)}}
{\gamma_{2-}(\gamma_{2+}-\gamma_{2-})}
<0.
	\end{equation*}
	Substituting these expressions into the first equation yields
	\begin{equation*}
		K_3=\frac{\gamma_{2+}+\gamma_{2-}}{\gamma_{2+}\gamma_{2-}}=\frac{N_1}{\delta},
	\end{equation*}
	where $N_1:=\overline N_1({\widetilde M})$ is defined in \eqref{eqn:notations}.
	
	At $x=w_0$, it suffices to show that the derivatives are continuous. We have the following system of equations:
	\begin{align}
		K_1&=\frac{1-a}{\gamma_{2+}-\gamma_{2-}}\left[\gamma_{2+}e^{\gamma_{2-}(w_0-u_1)}-\gamma_{2-}e^{\gamma_{2+}(w_0-u_1)}\right],\label{eqn:K1}\\
		-\frac{\kappa_1}{{\widetilde M}}K_1&=\frac{(1-a)\gamma_{2+}\gamma_{2-}}{\gamma_{2+}-\gamma_{2-}}\left[e^{\gamma_{2-}(w_0-u_1)}-e^{\gamma_{2+}(w_0-u_1)}\right].
	\end{align}
	Combining the two equations yields
	\begin{equation}\label{eqn:xM1 to u1}
		e^{(\gamma_{2+}-\gamma_{2-})(w_0-u_1)}=\frac{\gamma_{2+}(\kappa_1+\gamma_{2-}{\widetilde M})}{\gamma_{2-}(\kappa_1+\gamma_{2+}{\widetilde M})}=\frac{\frac{\kappa_1}{\gamma_{2-}}+{\widetilde M}}{\frac{\kappa_1}{\gamma_{2+}}+{\widetilde M}}.
	\end{equation}
	Since we assume that $w_0\leq u_1$, it must be the case that $e^{(\gamma_{2+}-\gamma_{2-})(w_0-u_1)}\in(0,1]$. Moreover, since $\gamma_{2-}<0$, we have $\frac{\frac{\kappa_1}{\gamma_{2-}}+{\widetilde M}}{\frac{\kappa_1}{\gamma_{2+}}+{\widetilde M}}<1$. From \eqref{eqn:xM1 to u1} and Lemma \ref{lemma: k1/M1>-gamma2-}, the formula for $u_1$ in \eqref{eqn:u1 unbdd M1 finite} is well defined. Note that since $\frac{\gamma_{2-}(\kappa_1+\gamma_{2+}{\widetilde M})}{\gamma_{2+}(\kappa_1+\gamma_{2-}{\widetilde M})}>1$, it holds that $w_0< u_1$. Thus, we have obtained the form of the value function in \eqref{g:unbdd M1 finite}. Similar arguments from the previous results yield that $g$ is increasing and concave.

    {Finally, we show that $u_1=G(\infty)$. In the region $\{w_0<x<u_1\}$, $y(x)=\widehat \pi_1(x)$ satisfies \eqref{eqn:ode for y} with $C=0$. Since $\overline N_1$ and $\overline N_2$ are constant on $[\widetilde M,\infty)$, we have $F_0(y)=\left(\frac{\dd G(y)}{\dd y}\right)^{-1}$ for $y\geq\widetilde M$, where $F_0$ is defined in \eqref{eqn:ode for y}. Moreover, $y(w_0)=\widetilde M$ and, since $g''(u_1)=0$, $y(x)\to\infty$ as $x\uparrow u_1$. Hence,
    \begin{equation*}
        u_1-w_0=\int_{\widetilde M}^{\infty}\frac{\dd y}{F_0(y)}=\int_{\widetilde M}^{\infty}\frac{\dd G(y)}{\dd y}\,\dd y=G(\infty)-G(\widetilde M),
    \end{equation*}
    which yields $u_1=G(\infty)$ since $w_0=G(\widetilde M)$.}
	
	\subsection{Proof of Theorem \ref{thm: unbdd M1 infty}}
	In this scenario, ${\widetilde M}=\infty$. Hence, there is no finite aggregate reserve level at which {both lines retain all of their risk}. Therefore, the reinsurance threshold level $w_0$ defined in \eqref{eqn:defn of w0} is infinite, i.e., $w_0=\infty$.
	
	We first derive the candidate value function in the continuation region. As in the previous cases, the optimal reinsurance strategies {are given by \eqref{eqn:optimal pi constrained}, where}
	\begin{equation}\label{eqn:pi FOC unbdd M infinity}
		{\widehat\pi}_1(x)=-\kappa_1\frac{g'(x)}{g''(x)},
		\qquad
		{\widehat\pi}_2(x)=\frac{\kappa_2}{\kappa_1}{\widehat\pi}_1(x).
	\end{equation}
	In the region where no dividends are paid, the HJB equation becomes
	\begin{equation}\label{eqn:hjb unbdd cont M infinity}
		0=
		\frac{1}{2}\overline N_2({\widehat\pi}_1(x))g''(x)
		+\overline N_1({\widehat\pi}_1(x))g'(x)
		-\delta g(x),
	\end{equation}
	where $\overline N_1$ and $\overline N_2$ are defined in \eqref{eqn:notations}, {and we used that $\mu_i$ and $\sigma_i^2$ are constant on $[M_i,\infty)$}. Using \eqref{eqn:pi FOC unbdd M infinity}, we have
	\begin{equation*}
		\frac{g''(x)}{g'(x)}
		=-\frac{\kappa_1}{{\widehat\pi}_1(x)}.
	\end{equation*}
	Substituting this into \eqref{eqn:hjb unbdd cont M infinity} yields
	\begin{equation}\label{eqn:hjb reduced unbdd M infinity}
		0=
		\left[
		\overline N_1({\widehat\pi}_1(x))
		-\frac{\kappa_1}{2{\widehat\pi}_1(x)}\overline N_2({\widehat\pi}_1(x))
		\right]g'(x)
		-\delta g(x).
	\end{equation}
	Following the same argument used in the bounded-dividend case, differentiating \eqref{eqn:hjb reduced unbdd M infinity} with respect to $x$ gives
	\begin{equation*}
		{\widehat\pi}_1(x)=G^{-1}(x),
	\end{equation*}
	where $G$ is defined in \eqref{eqn:G}. Equivalently, ${\widehat\pi}_1$ satisfies
	\begin{equation*}
		\frac{\dd}{\dd x}{\widehat\pi}_1(x)
		=
		\left(
		\frac{\dd}{\dd y}G(y)\bigg|_{y={\widehat\pi}_1(x)}
		\right)^{-1}.
	\end{equation*}
	Since {$g''(u_1)=0$ at the dividend barrier, $\widehat \pi_1(x)$ tends to infinity as $x\uparrow u_1$}. Thus, the barrier is
	\begin{equation*}
		u_1=G(\infty):=\lim_{y\to\infty}G(y).
	\end{equation*}
	We now show that $G(\infty)<\infty$. Since
	\[
		\mu_i(y)\to\widetilde\mu_i
		\quad\mbox{and}\quad
		\sigma_i^2(y)\to\widetilde\sigma_i^2
		\quad\mbox{as } y\to\infty,
	\]
	the integrand in the definition of $G$ satisfies
	\[
		\frac{\sigma_1^2(y)+\sigma_2^2\left(\frac{\kappa_2}{\kappa_1}y\right)}
		{
		2\kappa_1y\mu_1(y)
		+2\kappa_2y\mu_2\left(\frac{\kappa_2}{\kappa_1}y\right)
		+\frac{2\delta}{\kappa_1}y^2
		-\kappa_1\left[
		\sigma_1^2(y)+\sigma_2^2\left(\frac{\kappa_2}{\kappa_1}y\right)
		\right]
		}
		=O\left(\frac{1}{y^2}\right)
	\]
	as $y\to\infty$. Hence $G(\infty)<\infty$.
	
	In the continuation region $\{x<u_1\}$, the relation
	\[
		\frac{g''(x)}{g'(x)}
		=-\frac{\kappa_1}{G^{-1}(x)}
	\]
	implies that
	\begin{equation*}
		g'(x)
		=K_1\exp\left[\int_x^{u_1}\frac{\kappa_1}{G^{-1}(y)}dy\right],
	\end{equation*}
	for some constant $K_1>0$. Since $g(0)=0$, we obtain
	\begin{equation}\label{eqn:g1 unbdd M infinity}
		g_1(x)
		=
		K_1\int_0^x
		\exp\left[\int_z^{u_1}\frac{\kappa_1}{G^{-1}(y)}dy\right]dz,
		\qquad x<u_1.
	\end{equation}
	At the dividend barrier, the marginal value of one additional unit of aggregate reserve is $1-a$, because Line 2 pays the marginal dividend. Therefore,
	\begin{equation*}
		g'(u_1-)=1-a.
	\end{equation*}
	From \eqref{eqn:g1 unbdd M infinity}, this gives $K_1=1-a$. Hence
	\begin{equation*}
		g_1(x)
		=
		(1-a)\int_0^x
		\exp\left[\int_z^{u_1}\frac{\kappa_1}{G^{-1}(y)}dy\right]dz,
		\qquad x<u_1.
	\end{equation*}
	
	For $x\geq u_1$, the surplus above $u_1$ is immediately paid as dividends or transferred between lines. Hence the candidate value function is linear:
	\begin{equation}\label{eqn:g2 unbdd M infinity}
		g_2(x)=(1-a)(x-u_1+K_3),
		\qquad x\geq u_1,
	\end{equation}
	where $K_3$ remains to be determined.
	
	We now determine $K_3$ using the limiting form of the HJB equation at $u_1$. For $x<u_1$, we have
	\[
		{\widehat\pi}_1(x)=G^{-1}(x),
		\qquad
		{\widehat\pi}_2(x)=\frac{\kappa_2}{\kappa_1}G^{-1}(x).
	\]
	Since $u_1=G(\infty)$, it follows that
	\[
		G^{-1}(x)\to\infty
		\quad\mbox{as } x\uparrow u_1.
	\]
	Moreover,
	\[
		g'(x)\to 1-a
		\quad\mbox{and}\quad
		g''(x)
		=
		-\frac{\kappa_1}{G^{-1}(x)}g'(x)
		\to 0
		\quad\mbox{as } x\uparrow u_1.
	\]
	Taking the limit $x\uparrow u_1$ in \eqref{eqn:hjb unbdd cont M infinity} yields
	\begin{equation*}
		0
		=
		N_1^\infty(1-a)
		-\delta g(u_1-),
	\end{equation*}
	because
	\[
		\overline N_1(G^{-1}(x))\to N_1^\infty
		\quad\mbox{and}\quad
		\overline N_2(G^{-1}(x))g''(x)\to0.
	\]
	By continuity of $g$ at $u_1$, \eqref{eqn:g2 unbdd M infinity} gives
	\[
		g(u_1-)=g(u_1+)=(1-a)K_3.
	\]
	Therefore,
	\[
		0=N_1^\infty(1-a)-\delta(1-a)K_3,
	\]
	and hence
	\begin{equation*}
		K_3=\frac{N_1^\infty}{\delta}.
	\end{equation*}
	Thus the candidate value function is precisely the function given in \eqref{eqn:g unbdd M infty}.
	
	It remains to verify that this candidate satisfies the HJB equation. Let
	\[
		V(x_1,x_2)=g(x_1+x_2).
	\]
	Then
	\[
		\frac{\partial V}{\partial x_1}
		=
		\frac{\partial V}{\partial x_2}
		=
		g'(x_1+x_2),
	\]
	so the capital-transfer terms in \eqref{eqn:hjb unbdd} satisfy
	\[
		\frac{\partial V}{\partial x_1}-\frac{\partial V}{\partial x_2}
		=
		\frac{\partial V}{\partial x_2}-\frac{\partial V}{\partial x_1}
		=0.
	\]
	
	For $x=x_1+x_2<u_1$, we have $g'(x)>1-a\geq a$. Hence
	\[
		a-g'(x)<0,
		\qquad
		1-a-g'(x)<0.
	\]
	Moreover, by construction, the choice
	\[
		{(\Pi_1,\Pi_2)(x)=\left(\min\left\{M_1,G^{-1}(x)\right\},\min\left\{M_2,\frac{\kappa_2}{\kappa_1}G^{-1}(x)\right\}\right)}
	\]
	maximizes the reinsurance part of the HJB equation {by \eqref{eqn:optimal pi constrained}} and makes the continuation-region HJB equal to zero. Therefore, the HJB equation holds on $\{x<u_1\}$.
	
	For $x\geq u_1$, we have
	\[
		g'(x)=1-a,
		\qquad
		g''(x)=0.
	\]
	Therefore,
	\[
		1-a-g'(x)=0
		\quad\mbox{and}\quad
		a-g'(x)=2a-1\leq0,
	\]
	where we use the standing assumption $a\leq 1/2$. Furthermore,
	\begin{equation*}
		\sup_{\pi_i\in[0,M_i]}\mathcal L^\pi(g)
		\leq
		(1-a)N_1^\infty
		-\delta(1-a)\left[x-u_1+\frac{N_1^\infty}{\delta}\right]
		=
		-\delta(1-a)(x-u_1)
		\leq0.
	\end{equation*}
	Hence the HJB equation also holds on $\{x\geq u_1\}$.
	
	The function $g$ is increasing and concave. Indeed, for $x<u_1$,
	\[
		g'(x)
		=
		(1-a)\exp\left[\int_x^{u_1}\frac{\kappa_1}{G^{-1}(y)}dy\right]>0
	\]
	and
	\[
		g''(x)
		=
		-\frac{\kappa_1}{G^{-1}(x)}g'(x)<0.
	\]
	For $x\geq u_1$, $g'(x)=1-a>0$ and $g''(x)=0$. Moreover,
	\[
		g'(u_1-)=g'(u_1+)=1-a,
		\qquad
		g''(u_1-)=g''(u_1+)=0.
	\]
	Thus $g$ is concave and is twice continuously differentiable at $u_1$.

Finally, by {Theorem \ref{thm:verification}}, the classical solution $V(x_1,x_2)=g(x_1+x_2)$ to \eqref{eqn:hjb unbdd} equals the value function of the optimization problem. The optimal reinsurance strategy in the continuation region is therefore {given in the feedback form $\pi_i^*(t)=\Pi_i(X_1^*(t)+X_2^*(t))$, $i=1,2$, where}
\begin{equation*}
			({\Pi_1,\Pi_2}) (x)=
			\left({\min\left\{M_1,G^{-1}(x)\right\},\min\left\{M_2,\frac{\kappa_2}{\kappa_1}G^{-1}(x)\right\}}\right),
			\quad x < u_1.
		\end{equation*}
		For $x\geq u_1$, the surplus above the barrier is immediately removed by the singular dividend and capital-transfer strategy, so the reinsurance control in this region is irrelevant.

    \section{Conclusion}\label{sec:conclusion}

In this paper, we investigate optimal dividend payout, reinsurance, and capital-injection strategies for insurers with two business lines, where reinsurance combines proportional and excess-of-loss coverage. We establish that, {under the diffusion approximation,} the optimal reinsurance strategy is pure excess-of-loss and identify distinct, mutually exclusive dividend payout strategies for both bounded and unbounded dividend rates. {The optimal retention levels increase with the aggregate reserve level, and reinsurance is phased out line by line, in an order determined by the ratios $\frac{M_i}{\kappa_i}$.} The optimal capital-transfer strategy keeps both business lines solvent as long as the aggregate reserve is positive. In the bounded-dividend case, transfers are used to prevent the ruin of one line; in the unbounded-dividend case, transfers also help implement the aggregate dividend barrier.

Future research could explore alternative types of dividends, such as periodic dividends and immediate dividends that incorporate transaction costs \citep[see, e.g.,][]{kelbert2025,avanzi2020,avanzi2021}. Alternative processes that model the reserve level, such as L\'evy processes, could also be considered \citep[e.g.,][]{matalopez2024}. The objective function could be modified to account for penalties for early ruin \citep[e.g.,][]{strini2023,xu2020}. It would also be of interest to investigate deep learning approaches \citep[e.g.,][]{cheng2020}. {Finally, the problem could be studied directly under the CL model, where the optimality of pure excess-of-loss reinsurance is no longer guaranteed, and dependence between the lines could be incorporated, as in \cite{boonen2026uncertainty} for proportional reinsurance.}

{\appendix
\section{Proof of Theorem \ref{thm:verification}}\label{app:verification}
Fix $x_1,x_2\geq 0$, with $x:=x_1+x_2>0$; the case $x=0$ is trivial since no dividends can be paid and $w(0)=0$. Let $u$ be an admissible strategy, and denote by $X:=X_1+X_2$ the aggregate reserve level. Since the capital transfers cancel in the aggregate reserve level, we have
\begin{equation*}
    dX(t)=\sum_{i=1}^2\left[\kappa_i\mu_i(\pi_i(t))dt+\sigma_i(\pi_i(t))dW_i(t)\right]-d(C_1+C_2)(t),
\end{equation*}
and $X(t)\geq 0$ for $t<\tau$.} {Since $w'(x)$ may be unbounded as $x\downarrow 0$, we use a localization argument. For $n>\frac{1}{x}$, define $\tau_n:=\tau\wedge \inf\{t\geq 0:X(t)\leq\frac{1}{n}\}$. Applying It\^{o}'s formula to $e^{-\delta t}w(X(t))$ yields
\small\begin{equation}\label{eqn:ito verification}
    \begin{aligned}
        e^{-\delta(t\wedge\tau_n)}w(X(t\wedge\tau_n))&=w(x)+\int_0^{t\wedge\tau_n}e^{-\delta s}\left[\sum_{i=1}^2\left(\kappa_i\mu_i(\pi_i(s))w'(X(s))+\frac{1}{2}\sigma_i^2(\pi_i(s))w''(X(s))\right)-\delta w(X(s))\right]ds\\
		&\quad-\int_0^{t\wedge\tau_n}e^{-\delta s}w'(X(s-))\,d(C_1+C_2)^c(s)+\sum_{0\leq s\leq t\wedge\tau_n}e^{-\delta s}\left[w(X(s))-w(X(s-))\right]+\mathcal M(t),
    \end{aligned}
\end{equation}\normalsize
where $(C_1+C_2)^c$ is the continuous part of $C_1+C_2$, $X(0-):=x$, and $\mathcal M(t):=\sum_{i=1}^2\int_0^{t\wedge\tau_n}e^{-\delta s}w'(X(s))\sigma_i(\pi_i(s))dW_i(s)$. Since $w$ is concave, $0<w'(X(s))\leq w'(\frac{1}{n})$ for $s<\tau_n$ and $\sigma_i(\pi_i(s))\leq \widetilde\sigma_i$, so that $\mathcal M$ is a martingale. Moreover, by admissibility, dividend payments do not make the aggregate reserve level negative, so that $X(s)\geq 0$ after each jump.

\textit{Part (i).} In this case, $d(C_1+C_2)(s)=(c_1(s)+c_2(s))ds$ and there are no jumps. By the assumption on $w$, the integrand of the first integral in \eqref{eqn:ito verification}, after subtracting $(c_1(s)+c_2(s))w'(X(s))$, is at most $-e^{-\delta s}(ac_1(s)+(1-a)c_2(s))$. Taking expectations and using $w\geq 0$, we obtain
\begin{equation*}
    \mathbb E\left[\int_0^{t\wedge\tau_n}e^{-\delta s}\left(ac_1(s)+(1-a)c_2(s)\right)ds\right]\leq w(x)-\mathbb E\left[e^{-\delta(t\wedge\tau_n)}w(X(t\wedge\tau_n))\right]\leq w(x).
\end{equation*}
As $n\to\infty$, $\tau_n$ increases to a limit $\tau_\infty\leq\tau$. On the event
$\{\tau_\infty<\tau\}$, the aggregate reserve level is zero at time $\tau_\infty$,
and hence $X_1(\tau_\infty)=X_2(\tau_\infty)=0$. By the convention stated in
Section \ref{sec:model}, $(0,0)$ is an absorbing terminal state without loss of value:
choosing positive retention leads to immediate exit from the nonnegative quadrant,
whereas choosing zero retention leaves the process at $(0,0)$. Therefore no further
dividends can be generated after $\tau_\infty$. Letting $t\to\infty$ and $n\to\infty$ and using the monotone convergence theorem yields $J(x_1,x_2;u)\leq w(x)$. For $u^*$, the inequalities above become equalities, and the capital transfers ensure that ruin occurs only when the aggregate reserve level becomes negative. Since $w$ is bounded, $w(X(\tau_\infty))=w(0)=0$ on $\{\tau_\infty<\infty\}$, and $e^{-\delta t}\to 0$, the dominated convergence theorem yields $\mathbb E[e^{-\delta(t\wedge\tau_n)}w(X(t\wedge\tau_n))]\to 0$. Hence $J(x_1,x_2;u^*)=w(x)$.

\textit{Part (ii).} By assumption, the first integral in \eqref{eqn:ito verification} is non-positive. Since $w'(x)\geq 1-a$ on $(0,\infty)$ and $w$ is concave, the continuous part
satisfies
\[
w'(X(s))\,d(C_1+C_2)^c(s)\geq (1-a)\,d(C_1+C_2)^c(s),
\]
and each jump satisfies
\[
w(X(s-))-w(X(s))
\geq (1-a)\big[(C_1+C_2)(s)-(C_1+C_2)(s-)\big].
\]
Thus, the dividend finite-variation terms are bounded above by
\[
-(1-a)\int_0^{t\wedge\tau_n} e^{-\delta s}\,d(C_1+C_2)(s).
\] Taking expectations in \eqref{eqn:ito verification} and using $w\geq 0$ and $a\leq 1-a$, we obtain
\begin{equation*}
    w(x)\geq(1-a)\,\mathbb E\left[\int_0^{t\wedge\tau_n}e^{-\delta s}d(C_1+C_2)(s)\right]\geq\mathbb E\left[a\int_0^{t\wedge\tau_n}e^{-\delta s}dC_1(s)+(1-a)\int_0^{t\wedge\tau_n}e^{-\delta s}dC_2(s)\right].
\end{equation*}
Letting $t\to\infty$ and $n\to\infty$ as in Part (i) yields $J(x_1,x_2;u)\leq w(x)$. Under $u^*$, Line 2 pays the amount $(x-b)^+$ at time 0, for which $w(x)-w(x\wedge b)=(1-a)(x-b)^+$, and afterwards the aggregate reserve level is continuous, stays in $[0,b]$, and dividends are paid by Line 2 only when $X(s)=b$, where $w'(b)=1-a$. Since the HJB equation holds with equality on $(0,b)$, and by continuity at $b$, all of the inequalities above become equalities. Since $w(X^{u^*}(t))\leq w(b)$, the same limiting argument as in Part (i) yields $J(x_1,x_2;u^*)=w(x)$. This completes the proof.}

\section*{Disclosure statement}
The authors declare no competing interests.

%\section*{CRediT authorship contribution statement}
%\textbf{Tim J. Boonen}: Conceptualization, Formal analysis, Supervision, Writing - original draft, Writing - review and editing. \textbf{Engel John C. Dela Vega}: Conceptualization, Formal analysis, Software, Writing - original draft, Writing - review and editing
	
		\bibliographystyle{apalike} 
	\bibliography{References}

	\end{document}